\documentclass[11pt]{article}
\usepackage{srcltx}
\usepackage{a4wide}
\usepackage[dvips]{graphicx}
\usepackage{amsmath,amsthm}
\usepackage{mathrsfs}
\usepackage{wasysym}
\usepackage{amsfonts}
\usepackage{amssymb}
\usepackage{mdwlist}
\usepackage{epsfig}
\usepackage{pst-all}
\usepackage{export}
\input xy
\xyoption{all}

\newenvironment{SauveCompteurs}[1]{%
\newcommand{\monparametre}{#1}
\openexport{\monparametre_sauve}
  \Export{thm}\Export{section}\Export{subsection}\Export{subsubsection}
\closeexport}{}

\newenvironment{UtiliseCompteurs}[1]{%
\newcommand{\monparametre}{#1}
\openexport{\monparametre_aux}
  \Export{thm}\Export{section}\Export{subsection}\Export{subsubsection}
\closeexport
\Import{\monparametre_sauve}%
\renewcommand{\label}[1]{}
}{\Import{\monparametre_aux}}

\newcommand{\G}{\Gamma}

\newcommand{\calD}{\mathcal{D}}
\newcommand{\calE}{\mathcal{E}}
\newcommand{\calF}{\mathcal{F}}

\newcommand{\calL}{\mathcal{L}}

\newcommand{\calO}{\mathcal{O}}
\newcommand{\calP}{\mathcal{P}}

\newcommand{\calV}{\mathcal{V}}

\newcommand {\cala} {{\mathcal {A}}}   
\newcommand {\calb} {{\mathcal {B}}}   
\newcommand {\calc} {{\mathcal {C}}}   
\newcommand {\cald} {{\mathcal {D}}}   
\newcommand {\cale} {{\mathcal {E}}}   
\newcommand {\calf} {{\mathcal {F}}}   
\newcommand {\calg} {{\mathcal {G}}}

\newcommand {\call} {{\mathcal {L}}}   
\newcommand {\calm} {{\mathcal {M}}}   
   
\newcommand {\calo} {{\mathcal {O}}}   
\newcommand {\calp} {{\mathcal {P}}}   
\newcommand {\calq} {{\mathcal {Q}}}   
\newcommand {\calr} {{\mathcal {R}}}

\newcommand {\calv} {{\mathcal {V}}}

\newcommand {\calz} {{\mathcal {Z}}}

\newcommand {\bbF} {{\mathbb {F}}}   
   
\newcommand {\bbH} {{\mathbb {H}}}

\newcommand {\bbR} {{\mathbb {R}}}

\newcommand {\bbZ} {{\mathbb {Z}}}

\newcommand{\R}{\mathbb{R}}

\newcommand{\Zmax}{{\mathcal{Z}_{\mathit{max}}}}
\newcommand{\Z}{{\mathcal{Z}}}

\newtheorem{defi}{\textbf{Definition}}[section]

\newtheorem{theo}[defi]{\textbf{Theorem}}
\newtheorem{lemma}[defi]{\textbf{Lemma}}
\newtheorem{prop}[defi]{\textbf{Proposition}}
\newtheorem{coro}[defi]{\textbf{Corollary}}

\newtheorem{dfn}[defi]{Definition}
\newtheorem{thm}[defi]{Theorem}
\newtheorem{thmbis}{Theorem}
\newtheorem{propbis}[thmbis]{Proposition}
\newtheorem{corbis}[thmbis]{Corollary}
\newtheorem{lem}[defi]{Lemma}
\newtheorem{cor}[defi]{Corollary}

\theoremstyle{remark}
\newtheorem{rem}[defi]{Remark}
\newtheorem*{rem*}{Remark}
\newtheorem{remark}[defi]{Remark}

\newcommand{\inn}{\mathrm{ad}}

\makeatletter
\edef\@tempa#1#2{\def#1{\mathaccent\string"\noexpand\accentclass@#2 }}
\@tempa\rond{017}
\makeatother

\newcommand{\ie} {i.~e.\ } 

\newcommand{\es}{\emptyset}
\renewcommand{\phi}{\varphi} 
\newcommand{\m}{^{-1}} 
\newcommand{\eps} {\varepsilon} 

\newcommand {\ra} {\rightarrow}

\newcommand {\onto} {\twoheadrightarrow}
\newcommand {\into} {\hookrightarrow}
\newcommand {\xra} {\xrightarrow}    

\newcommand{\actson}{\curvearrowright}

\newcommand{\ol}[1]{\overline{#1}}

\newcommand{\normal} {\vartriangleleft}
\renewcommand{\subsetneq}{\varsubsetneq}

\newcommand{\dunion}{\sqcup}

\newcommand{\grp}[1]{{\langle #1 \rangle}}
\newcommand{\Isom} {{\mathrm{Isom}}}
\renewcommand{\Im} {\mathop{\mathrm{Im}}}

\newcommand{\Stab} {{\mathrm{Stab}}}

\newcommand{\Out} {{\mathrm{Out}}}

\newcommand{\Aut} {{\mathrm{Aut}}}

\newcommand{\Homeo} {{\mathrm{Homeo}}}

\newcommand{\id} {\mathrm{id}}

\newcommand{\qie}{quasi-isometrically embeddable }
\newcommand{\qiers}{quasi-isometrically embeddable rational subset }
\newcommand{\Outu} {{\mathrm{Out_u}}}
\newcommand{\Outm} {{\mathrm{Out_m}}}
\newcommand{\Autu} {{\mathrm{Aut_u}}}
\newcommand{\Autm} {{\mathrm{Aut_m}}}
\newcommand{\Periph} {\mathit{Periph}}
\newcommand{\Cay} {\mathop{\mathrm{Cay}}}

\newcommand{\TZmax}{T_{\Zmax}}

\title{The isomorphism problem for all hyperbolic groups.}
\date{}
\author{Fran\c{c}ois Dahmani, Vincent Guirardel}

\begin{document} 

\maketitle

\begin{abstract}
 We give a solution to Dehn's isomorphism problem for the class of all hyperbolic groups,
possibly with torsion. 
We also prove a relative version for groups with peripheral structures.
As a corollary, we give a uniform solution to Whitehead's problem asking whether two tuples of elements of a hyperbolic group $G$
are in the same orbit under the action of $\Aut(G)$.
We also get an algorithm computing a generating set of 
the group of automorphisms of a hyperbolic group preserving a peripheral structure.
\end{abstract}

\section{Introduction}
 
In 1912, Dehn asked about three fundamental algorithmic problems for groups:
the word problem,  the conjugacy problem, and the isomorphism problem.
The word problem and the conjugacy problem in a group $G$ consist in deciding algorithmically
whether two words in some finite generating set represent the same or conjugate elements in $G$.
On the other hand, the isomorphism problem for a class of groups consists in deciding algorithmically whether
two  
 group presentations in this class represent isomorphic groups.
It is remarkable that the answers to such algorithmic problems, positive or negative, in generality or in particular classes, 
have repeatedly revealed deep and fruitful structures in group theory. 

In the 1950's, it was discovered that all of these problems have negative  answers in the class of all finitely presented groups. 
More precisely, Boone and Novikov proved the existence of a finitely presented group for which no algorithm can solve the word problem \cite{Boone_WP,Novikov_WP}. 
Adyan and Rabin used such a group to prove that no algorithm can decide whether an arbitrary finite presentation defines a non-trivial group 
\cite{Adyan_algorithmic,Rabin_recursive}. 

However there are many interesting  and large classes of groups for which algorithms solving  the word and conjugacy problems are well know. 
Therefore, groups with unsolvable word problem are often regarded as monsters, or as constructed ``on purpose''.

In striking contrast, the isomorphism problem is unsolvable for some very natural classes of groups,
including the class of
free-by-free groups (Miller \cite{Miller_decision}), the class of [free abelian]-by-free groups (Zimmermann \cite{Zimmermann_klassifikation}),
or the class of solvable groups of derived length $3$ (Baumslag-Gildenhuys-Strebel \cite{BGS_algorithmically2}
 following \cite{Kharlampovich_unsolvable81}). 

In fact, until recently and the use of geometric group theory, the isomorphism problem was known to be decidable in only a few cases,
including notably the class of virtually polycyclic and nilpotent groups by  Grunewald and Segal (\cite{GruSe_nilpotent,Segal_decidable}).
The isomorphism problem for the class of Coxeter groups, the class of generalised Baumslag-Solitar groups, and of one-relator groups
have been investigated but remain unsettled \cite{Bahls_isomorphism,ClFo_isomorphism,Pietrowski_isomorphism,Pride_isomorphism}.\\

Z.~Sela's solution of the isomorphism problem for the class of rigid torsion-free hyperbolic groups was certainly a great breakthrough \cite{Sela_isomorphism}.
Sela also had a  
solution for the class of all torsion-free hyperbolic groups but did not publish it.
This program was continued by D.~Groves and the first author, who simplified Sela's initial approach, and gave a proof for the class of all
 torsion free hyperbolic groups and toral relative hyperbolic groups \cite{DaGr_isomorphism}.

\paragraph{Statement of main results.}
In this paper, we give a solution to the isomorphism problem for the class of \emph{all} word-hyperbolic groups (as defined in \cite{Gromov_hyperbolic}),  
possibly with torsion.

\begin{thmbis}\label{thmbis_iso}
There is an explicit algorithm that takes as input two presentations of hyperbolic groups, and
which decides whether these groups are isomorphic or not.
\end{thmbis}

A result by Newman shows that one-relator groups with non-trivial torsion are hyperbolic
\cite{Newman_one-relator}.
We thus get the following corollary:

\begin{corbis} The isomorphism problem for one-relator groups with non-trivial torsion is solvable.
\end{corbis}

In our methods, 
the solution of the isomorphism problem is symbiotic with the computation of
a generating set of the group of automorphisms.

\begin{thmbis}\label{thmbis_gene}
  There is an explicit algorithm that takes as input a presentation of a hyperbolic group $G$, 
and which computes a generating set of $\Aut(G)$ and $\Out(G)$.
\end{thmbis}

In our solution of the isomorphism problem (as well as in \cite{DaGr_isomorphism}), 
one needs to compute various decompositions of hyperbolic groups
as  amalgamated free products, HNN extensions, and more generally,
as graphs of groups. 
This raises the natural question whether vertex groups of a graph of groups are isomorphic \emph{relative to their adjacent edge subgroups}. 
A variation of this problem is to consider groups equipped with \emph{marked peripheral structures},  
namely a finite ordered collection of tuples of elements that are thought of as generating sets of the adjacent edge groups. More precisely, 
$(S_1,\dots, S_n)$ is a marked peripheral structure on $G$ if each $S_i$ is a tuple of elements of $G$ (each
tuple being understood up to conjugacy).

Given $G,G'$ two groups, and marked peripheral structures $(S_1,\dots,S_n)$ and $(S'_1,\dots,S'_n)$, 
 the \emph{marked isomorphism problem} consists in deciding if there exists an isomorphism $f:G\ra G'$ sending
 $S_i$ to a conjugate $S_i'^{g'_i}$ 
of $S'_i$ for all $i$.

\begin{thmbis}\label{thmbis_marked}[see Theorem \ref{thm_whitehead_marque}]
  The marked isomorphism problem is solvable among hyperbolic groups with marked peripheral structures. 

Moreover, one can algorithmically compute  a generating set of the group of automorphisms of a hyperbolic group with marked peripheral structure.
\end{thmbis}

Although of more technical appearance, Theorem \ref{thmbis_marked} has a particularly nice consequence. 
The  \emph{Whitehead problem} in a  group $G$,
asks whether two tuples of elements $G$ are in the same orbit under the automorphism group of $G$.
Theorem \ref{thmbis_marked} gives a uniform solution to the Whitehead problem for all hyperbolic groups.

\begin{corbis}[see cor.\ \ref{cor_whitehead_1234}]\label{cor_whitehead}
 Given $G$  a hyperbolic group and $g_1,\dots,g_n, g'_1,\dots, g'_n\in G$,
 one can decide if there exists an automorphism $\phi\in \Aut(G)$ sending  $g_i$ to $g'_i$ for all $i$.

One can also decide if there exists an automorphism $\phi\in \Aut(G)$ sending $g_i$ to a conjugate of $g'_i$ for all $i$.
\end{corbis}

A solution of the Whitehead problem was known for free groups \cite{Whitehead_equivalent,HigLyn_Whitehead}, and surface groups
\cite{LeVo_Whitehead}. 
It is interesting to notice that even in these cases,  
 our approach is quite different from previous ones
as it is ultimately based on the computation of relative Grushko and JSJ decompositions 
of $G$. 
\\

\paragraph{Structural features.}
Our approach follows the strategy initiated by Sela \cite{Sela_isomorphism}, and continued by Groves and the first author 
as exposed in \cite[Section 2]{DaGr_isomorphism}.

Our result is based on the following  three main structural features: 
\begin{itemize*}
\item  the Stallings-Dunwoody deformation space of maximal decompositions of $G$ over finite groups
\item a  rigidity criterion, saying that for a one-ended hyperbolic group $G$, 
$\Out(G)$ is infinite if and only if $G$ has an ``interesting''  splitting providing an infinite group of Dehn twists,
\item a particular kind of canonical JSJ decomposition, adapted to the rigidity criterion. 
\end{itemize*}
Moreover, we prove that these features are algorithmic: one can compute these invariants, and decide whether the rigidity criterion holds.
The two major algorithmic tools we use are  Gerasimov's algorithm which detects whether 
 a given  hyperbolic group splits over a finite subgroup (\cite{Gerasimov}, see also a published version in \cite{DaGr_detecting}), 
and a solution to the problem of equations in hyperbolic groups, \cite{DG1}. 
These algorithms themselves rely on interesting structures. We do not  detail them  here, instead we refer the interested reader to the indicated bibliography.
\\

Let us now give more details about these structural features and their computation.

\paragraph{The Stallings-Dunwoody deformation space.} 
The \emph{Stallings-Dunwoody deformation space} of $G$
is the set of decompositions of $G$ as graphs of groups with finite edge groups and finite or one-ended vertex groups. 
Existence of such a decomposition for a finitely presented group is Dunwoody's original accessibility,
and Stallings' theorem shows its relation with the number of ends of $G$.

Gerasimov's algorithm \cite{Gerasimov,DaGr_detecting} allows one to compute some decomposition in this deformation space.
In absence of torsion, 
the uniqueness property of the Grushko decomposition up to isomorphism 
reduces immediately the isomorphism problem to the case of freely indecomposable groups.

In presence of torsion, there is no such nice uniqueness statement. 
Instead, we use the fact that the action of $\Out(G)$ on the
Stallings-Dunwoody deformation space is cocompact, 
and that two reduced trees in this deformation space are connected by \emph{slide moves}.
Starting from a particular decomposition in this deformation space given by Gerasimov's algorithm,
we are able to compute   
the vertices of the quotient of this deformation space by $\Out(G)$, \ie 
the finite set of all isomorphism classes of reduced Stallings-Dunwoody decompositions. 
Once this is done, the isomorphism problem for several-ended hyperbolic groups reduces
to the isomorphism problem for one-ended  hyperbolic groups with a peripheral structure consisting of finite groups.
 This is done in Section \ref{sec_ends}.

\paragraph{The rigidity criterion.}
To introduce our version of the rigidity criterion, first consider the torsion-free case. 
If $G$ is a one-ended torsion-free hyperbolic group with $\Out(G)$ infinite, 
the Bestvina-Paulin argument shows that $G$ has a small action on an $\bbR$-tree \cite{Be_degenerations,Pau_topologie}.
Then by Rips theory, 
one can construct a non-trivial splitting $G=A*_C B$ (or an HNN extension $G=A*_C$, but let's ignore this case in this introduction)
over a maximal cyclic subgroup $C$ of $G$ \cite[Theorem 9.1]{Sela_isomorphism}.
Conversely, a splitting $G=A*_C B$, any $c\neq 1$ in the centre of $C$ defines a Dehn twist $\tau_c$ as the identity on $A$ and the conjugation
by $c$ on $B$. 
If $C$ is a \emph{maximal cyclic} subgroup of $G$, then  $A$ and $B$ have finite centre which guarantees
that $\tau_c$ has infinite order in $\Out(G)$.

In presence of torsion,   
the class of virtually cyclic groups naturally generalises that of cyclic groups.
However, because the infinite dihedral group $D_\infty$ has trivial centre, there is no non-trivial Dehn twist (in the sense we just discussed) 
arising from an amalgamated free product $A*_{D_\infty} B$, even though $D_\infty$ is indeed virtually cyclic.
 With this example in mind, a good dichotomy over virtually cyclic groups is to distinguish whether  their centre is infinite 
or finite.
 We call \emph{$\Z$-groups} the virtually cyclic groups with infinite centre, and \emph{$\Zmax$-subgroups} the $\Z$-subgroups maximal for inclusion.
 Only splittings over $\Z$-groups provide Dehn twists that can be of infinite order (see \cite{MNS_downunder} where this difficulty was already spotted),
and the fact that they are of infinite order is guaranteed for splittings over $\Zmax$-subgroups. 
An important observation we make, is that a  rigidity criterion remains true in presence of torsion:

\begin{propbis}\label{alt}(Rigidity criterion, see Proposition \ref{prop_alt})

  Let $G$ be a one-ended hyperbolic group.  Then the following assertions are equivalent:
  \begin{enumerate*}
  \item $G$ does not split non-trivially over a $\Zmax$-subgroup,
 \item $\Out(G)$ is finite, 
 \item There exists  $R>0$  such that, modulo inner automorphisms, 
there are only finitely many 
endomorphisms of $G$ injective on the ball of radius $R$. 
 \item For all hyperbolic group $H$, there exists  $R>0$  such that, modulo inner automorphisms of $H$, 
there are only finitely many 
morphisms $G\ra H$ injective on the ball of radius $R$ of $G$. 
  \end{enumerate*}
\end{propbis}

See  \cite[Theorem 1.4]{Lev_automorphisms} for  a similar statement.
The first part of this equivalence is proved by observing that a  splitting over $\Zmax$-subgroups
provides Dehn twists making $\Out(G)$ infinite. 
Assuming the negation of the   
fourth assertion, 
the Bestvina-Paulin argument produces an action on an $\bbR$-tree, whose careful analysis shows the existence of a $\Zmax$ splitting.

\paragraph{Isomorphism problem and recognition of rigid groups.}
An important idea of Sela is to use, in conjunction to such a rigidity criterion,  
 a solution to the problem of equations 
to  get a finite list of morphisms $G\ra H$ containing a representative of all monomorphisms \cite{Sela_isomorphism}.
This approach was simplified in \cite{DaGr_isomorphism}, thanks to the use of \emph{rational constraints} in systems of equations.
In \cite{DG1}, we developed a solution to the problem of equations with (\emph{quasi-isometrically embedded}) rational constraints 
in hyperbolic groups with torsion.
We follow the same approach,   
using this solution to the problem of equations together with our rigidity criterion mentioned above.

Let us describe this approach.
Morphisms $\phi:G\ra H$ can be encoded by solutions in $H$ of a system of equations corresponding to the presentation of $G$;
injectivity of $\phi$ on the ball of radius $R$ can be encoded by inequations; and roughly speaking, rational constraints
can be used to ensure that two morphisms $\phi, \phi'$ do not coincide modulo inner automorphisms of $H$.
Thus, given $\phi_1,\dots,\phi_n:G\ra H$, one can produce a system of equations with rational constraints saying that
$\phi:G\ra H$ is a morphism injective on the ball of radius $R$, distinct from $\phi_1,\dots,\phi_n$ modulo inner automorphisms of $H$.
Now if Assertion 4 of the rigidity criterion holds, 
one can enumerate all morphisms $G\ra H$, and one will eventually find a finite family
$\phi_1\dots,\phi_n:G\ra H$ such that  the corresponding system has no solutions. This attests that every monomorphism $\phi:G\ra H$
is post-conjugate to some $\phi_i$.

This argument has two important consequences. First, this allows to recognise whether a one-ended hyperbolic group $G$ is rigid 
(\ie satisfies Proposition \ref{alt}) or not.
Indeed, if $G$ is rigid, one can apply the argument above with $H=G$, thus attesting that Assertion 3 of Proposition \ref{alt} holds.
If $G$ is not rigid, the fact that Assertion 1 fails can be attested by producing a splitting of $G$.
Second, if both $G$ and $H$ are rigid, one can compute two finite list of morphisms $\phi_i:G\ra H$ and $\psi_j:H\ra G$ containing
a representative of all monomorphisms up to inner automorphisms.
Then one solves the isomorphism problem between $G$ and $H$ 
by checking whether there exists $i,j$ such that $\phi_i\circ \psi_j$ and $\psi_j\circ \phi_i$ are inner automorphisms.
 This is the content of Section \ref{subsubsec_resolving}.

\paragraph{JSJ decompositions.}
Thanks to the decidability of the rigidity criterion, one can decide if $G$ has a splitting over a $\Zmax$-subgroup.
Then, using a relative version of the rigidity criterion (Proposition \ref{prop_alt}), one can decide if the vertex groups
split over $\Zmax$-subgroups relative to the incident edge groups. Iterating this procedure,
one can compute a maximal splitting over $\Zmax$-subgroups of $G$.

However, such a splitting is not unique up to automorphisms, and we cannot use it to reduce the isomorphism problem to the case of rigid hyperbolic groups 
(even rigid relative to a peripheral structure).  
 This is why we need a particular kind of JSJ decomposition, encoding splittings over $\Zmax$-subgroups.

For any class $\cala$ of subgroups of $G$, invariant under conjugation and stable under taking subgroups,
one can define JSJ decompositions over $\cala$ \cite{GL3a}.
In general, this defines a deformation space, but in the cases we consider, this deformation space
contains  a preferred canonical (\ie $\Out(G)$-invariant) decomposition, so we speak about \emph{the} JSJ decomposition over $\cala$.
In a JSJ decomposition, one distinguishes between \emph{rigid} and \emph{flexible} vertex groups, 
according to whether they are elliptic in all splittings of $G$  over subgroups in $\cala$.  

We will discuss three possibilities for $\cala$: the class of virtually cyclic groups and their subgroups,
the class $\calz$ of virtually cyclic groups with infinite centre and their subgroups,
and the class $\Zmax$ of maximal $\Z$-subgroups.

The JSJ decomposition over virtually cyclic groups is now classical. 
It coincides with Bowditch's decomposition \cite{Bo_cut} and has been widely studied  \cite{DuSa_JSJ,FuPa_JSJ},
but we still do not know whether it is computable.
Its edge groups are virtually cyclic and its flexible
subgroups are hanging \emph{bounded Fuchsian groups}, \ie finite extensions of fundamental groups of hyperbolic $2$-orbifolds, possibly with mirrors.

Over the class of $\Z$-subgroups, the flexible subgroups of the JSJ decomposition
 are finite extensions of fundamental groups of hyperbolic $2$-orbifolds
\emph{without mirrors}, and edge groups are $\calz$-groups.
 Maybe surprisingly, this decomposition can be non-trivial for 
the fundamental group $G$ of a closed orbifold with mirrors.
Indeed, we prove that the $\calz$-JSJ decomposition of an orbifold with mirrors
is the splitting over the boundary of a regular neighbourhood of the union of mirrors.
The reason is that splittings over $\Z$-subgroups of $G$ are dual to simple closed curves which don't intersect the singular locus.
One can state an interesting corollary of this observation.

\begin{UtiliseCompteurs}{aut_hyp}
\begin{cor}[see also \cite{Fujiwara_outer}]\label{cor_aut_hyp}
  Let $G$ be a one-ended hyperbolic group,   
possibly with torsion (for instance, the fundamental group of a closed orbifold with mirrors $\Sigma$).

Then there is a finite index subgroup $\Out_f(G)$ of $\Out(G)$ which is an extension
$$1\ra \mathit{Ab} \ra \Out_f(G_v)\ra \prod_{i=1}^n PMCG_f^*(S_i)\ra 1$$
where  $\mathit{Ab}$ is virtually abelian, and
$PMCG_f^*(S_i)$ is a finite index subgroup of the pure extended mapping class group of a surface with boundary.
\end{cor}
\end{UtiliseCompteurs}

Unfortunately, we are not able to compute algorithmically the JSJ decomposition over $\calz$,
so we consider (a variant of) the JSJ decomposition over the class $\cala=\Zmax$.
This decomposition is different from the more usual ones, 
 and it should play a helpful role in order to extend Sela's program of elementary equivalence among hyperbolic groups 
to the case of hyperbolic groups with torsion.   

As $\Zmax$ is not stable under taking subgroups, the JSJ decomposition over $\Zmax$-subgroups does not fit into the definition of JSJ decompositions from \cite{GL3a},
but can nevertheless be defined (see Section \ref{sec_TZmax}).

Its rigid vertex groups are those with no $\Zmax$-splitting relative to incident edge groups. 
According to the rigidity criterion, they are those  with finite outer automorphism group relative to incident edge groups.
Its flexible vertices 
are \emph{orbisockets}. 
These are finite extensions of fundamental groups of $2$-orbifolds without mirrors, whose boundary subgroups are amalgamated to larger $\Z$-groups.
A typical example consists in adding a root to a boundary component as follows:
the orbisocket group is $\Sigma*_{b=c^k}\grp{c}$ where $\Sigma$ is the fundamental group of a surface with boundary,
and $\grp{b}$ is the fundamental group of a boundary component.

A key step in our proof consists in computing this JSJ decomposition over $\Zmax$-subgroups. 
To do so, one first computes a maximal decomposition $\Gamma$ of $G$ over $\Zmax$-subgroups as explained above.
In such a decomposition, an orbisocket is cut into pieces, called \emph{basic orbisockets}. 
One reconstructs the JSJ decomposition by first recognising the vertex groups of $\Gamma$ which are basic orbisockets,
and by gluing together the pieces that match.
This recognition is not immediate, even in the torsion free case \cite{DaGr_isomorphism}. 
The situation here is even more delicate, and will occupy a significant part of the study.

Once the $\Zmax$ JSJ decompositions of one-ended hyperbolic groups have been computed,
the isomorphism problem reduces to the isomorphism problem for the vertex groups  (with marked peripheral structures).
A relative version of the isomorphism problem for rigid groups (relative to marked peripheral structure)
allows to do so for rigid vertex groups,
and the isomorphism problem for orbisockets
is easy once the basic orbisockets it is made of are identified.
\\

Let us comment the structure of the paper. 
We decided to include a rather extended toolbox, in which we recall classical, but sometimes subtle, material, 
including elements of Bass-Serre theory, 
and isomorphisms of  graphs of groups. 
Section \ref{sec_rigid} is devoted to the rigidity criterion, and its application to the isomorphism problem for
rigid hyperbolic groups.
In Section \ref{sec_JSJ}, we introduce the definition and properties of the JSJ decompositions over $\Z$ and $\Zmax$-subgroups,
and we introduce orbisockets as flexible vertices of the $\Zmax$ JSJ decomposition.
Sections \ref{sec_orbi} and \ref{sec_1end} are mainly devoted to the computation of this $\Zmax$ JSJ decomposition.
The main part of section \ref{sec_orbi} is devoted to the recognition of \emph{basic} orbisockets, 
from which follows a solution of the isomorphism problem 
for orbisockets with their marked peripheral structure. 
In section \ref{sec_1end}, we compute the $\Zmax$ JSJ decomposition by gluing together basic orbisockets
of some non-canonical maximal decomposition, and we conclude our solution of the isomorphism problem for one-ended hyperbolic groups, 
Section \ref{sec_ends} is devoted to hyperbolic groups with several ends, and to the computation of the Stallings-Dunwoody deformation space,  
and finishes the solution to the isomorphism problem for all hyperbolic groups. 
Finally, in Section \ref{sec_Wh}, 
we prove a relative version of the isomorphism problem, and
we deduce a solution of Whitehead problems.

{\small 
\setcounter{tocdepth}{2}
\tableofcontents
}

\section{Tool box}\label{sec_toolbox}

\subsection{Actions of finitely generated groups on finite sets}

 Several times, we will use orbit decidability of group actions on finite sets. 
The following lemma is rather elementary, but we need a somewhat general statement.

\begin{lemma} (Orbit decidability in finite sets)  \label{lem_probleme_orbite}
  Let $G$ be a group  acting on a finite set $X$, and $\tilde{X}$ a set with a surjection $\pi:\tilde{X}\to X$ on $X$.   
  Assume that the following is known:
  \begin{itemize*}
  \item a finite generating set $S$ of $G$,
  \item an algorithm deciding whether two given elements of $\tilde{X}$ have same image in $X$.
  \item  an algorithm that, given $\tilde x\in \tilde{X}$ and $s\in S$, computes an element  $\tilde y\in \tilde{X}$ with $\pi(\tilde y)=s.\pi(\tilde x)$,
  \end{itemize*}
Then given  two elements in $\tilde x,\tilde y\in \tilde{X}$, one can decide whether $\pi(\tilde x)$ and $\pi(\tilde y)$ 
are in the same orbit under the action of $G$. 
 If they are in the same orbit, one can compute an element of $G$ (as a word on $S\cup S\m$) 
sending $\pi(\tilde x)$ to $\pi(\tilde y)$.

Moreover, given  $\tilde x\in\tilde{X}$, one can compute a generating set of the stabiliser of $\pi(\tilde x)$.
\end{lemma}

Let us emphasize that  the entire set $X$, and even its cardinality, is not assumed to be known, that $\Tilde X$ might be infinite,
and that $G$ does not act on $\Tilde X$. 

\begin{proof}
Consider $x=\pi(\tilde x)$ and $y=\pi(\tilde x)$.
Let $B_n(x)\subset X$ be the set of images of $x$ under elements of $G$ of length at most $n$.

By hypothesis,  one can compute from $\tilde x$ representatives in $\tilde{X}$ of $B_n(x)$. 
One can also check whether $y\in B_n(x)$
and check whether $B_n(x)=B_{n+1}(x)$.
Since $X$ is finite, $B_n(x)=B_{n+1}(x)$ for some $n$,
so $B_n(x)=Gx$ since $B_n(x)$ is invariant under the generators of $G$.
One can therefore check whether $y$ lies in $Gx$.
In this case, one easily finds a word of length $\leq n$ sending $x$ to $y$.

Let's compute generators for the stabiliser of $x$.  The argument above allows one to compute
the  Schreier graph $\Sigma$ of the action of $G$ on $Gx$:
its vertex set is $Gx$ and $x,y$ are joined by a directed edge labelled by $s\in S$ if $sx = y$. 
One can obtain a set of generators of $\Stab(x)$ by considering the words labelled by a generating set of $\pi_1(\Sigma,x)$ 
(see  also \cite[Theorem 2.7, p.89]{MKS_combinatorial}). 
\end{proof}

\subsection{Extensions}
\label{subsec_extensions}
      Let $N$ be a fixed group,
      and  $ 1\to N \to E_1 \to G_1 \to 1$,  $ 1\to N \to E_2 \to G_2 \to 1$ some extensions of $G_1,G_2$ by $N$.
      We say that an isomorphism $\alpha:G_1\ra G_2$ \emph{lifts} if there exists $\phi:E_1\ra E_2$
      making the following diagram commute:
      $$\xymatrix{
        N \ar[d]^{\id_N} \ar[r] & E_1 \ar[r] \ar[d]^{\phi} & G_1 \ar[d]^{\alpha} \\
        N               \ar[r] & E_2 \ar[r]               & G_2 
      }$$
      (note that this requires the restriction of $\phi$ to be the identity on $N$).
      Two extensions $E_1,E_2$ of the same group $G$ by $N$ are \emph{equivalent} if $\id_G$ lifts.
      We denote by $\cale(G,N)$ the set of equivalence classes of extensions of $G$ by $N$.
      Note that $\Aut(G)$ acts on the right on $\cale(G,N)$ as follows: if $N\xra{\iota} E \xra{\pi} G$ is an extension,
      then $N\xra{\iota} E \xra{\pi\circ\alpha} G$ is the new extension. This new extension is 
equivalent to the original one
      if and only if $\alpha$ lifts.

We will consider the case where $N$ is finite.
Given two groups $G,N$ described via presentations, 
an extension $1\ra N\ra E \ra G$ will be described algorithmically by
a finite presentation of $E$, a finite subgroup $N'\normal E$ (described by the list of its elements, and a presentation),
an isomorphism $N\ra N'$,
and a isomorphism between $G$ and $E/N'$ (a presentation of $E/N'$ can be deduced from the presentation of $E$ by adding elements of $N'$ as relators).

\begin{lem}\label{lem_extens}
  Given two extensions $E_1,E_2\in \cale(G,N)$ of $G$ by the same finite group $N$, and solutions to the word problem in these groups,
  one can decide whether $E_1,E_2$ are equivalent.  If they are, one can compute an isomorphism $E_1\ra E_2$ realising the equivalence.
\end{lem}

        \begin{proof}
          Using the given isomorphisms, we view $N$ as a subgroup of $E_1$ and $E_2$.
          Let $S_1=(s_1,\dots,s_r)$ be an ordered generating set of $E_1$, 
          and $\ol S$ its image in $G=E_1/N=E_2/N$ (recall that we are given isomorphisms
          between $E_i/N$ and $G$).
          Let $S_2=(t_1,\dots,t_r)$ be a lift of $\ol S$ in $E_2$  (maybe generating or not).   
          Since $G$ is given as a quotient of $E_2$, one can compute such $S_2$.
          Now $E_1$ and $E_2$ are equivalent if and only if there exists an isomorphism $f:E_1\ra E_2$
          inducing the identity on $N$, and
          sending $(s_1,\dots,s_r)$ to $(t_1 n_1,\dots, t_r n_r)$ where $n_i\in N$. 
          One can tell if a mapping $(s_1,\dots,s_r)\mapsto(t_1 n_1,\dots, t_r n_r)$ extends
          a homomorphisms $E_1\ra E_2$
          using the presentation of $E_1$ and a solution to the word problem in $E_2$.
          This way, we get the list of all morphisms $E_1\ra E_2$  sending $(s_1,\dots,s_r)$ to $(t_1 n_1,\dots, t_r n_r)$.
          One can similarly list all morphisms $E_2\ra E_1$ sending $(t_1,\dots,t_r)$ to $(s_1 n'_1,\dots, s_r n'_r)$.
          One can check if two such morphisms are inverse of each other, 
          and if their restriction to $N$ is the identity          using the word problem.
        \end{proof}

        \begin{lemma}
          Let $G$ be a finitely presented group, and $N$ be a finite group.
          The set $\cale(G,N)$ is finite.
        \end{lemma}

        \begin{proof}  An equivalence class of extensions of $G$ by $N$ determines a morphism
        $\psi: G\to \Out (N)$ by conjugation. Such a morphism $\psi$ 
        induces a morphism $G\ra \Aut(Z(N))$ which makes the centre $Z(N)$ a $G$-module.
        By                 \cite[Theorem IV 6.6]{Brown_cohomology}, 
        the set of equivalence classes of extensions of $G$ by $N$ inducing $\psi:G\ra\Out(N)$,
        is either empty, or admits a free transitive action by the abelian group $H^2(G,Z(N))$. 
        Since $G$ is finitely presented,  and $Z(N)$ is finite,
        $H^2(G,Z(N))$ is finite. Since there are only finitely many
        possible morphisms $\psi$, $\cale(G,N)$ is itself finite.
        \end{proof}

      \begin{prop}\label{prop_iso_extension}
        There is an algorithm that takes as input two finitely presented equivalence classes of
        extensions $E_1,E_2\in \cale(G,N)$ of $G$ by the same finite group $N$, 
        a solution of the word problem in $E_1$ and $E_2$, and some automorphisms $\alpha_1,\dots,\alpha_n\in \Aut(G)$
        and which decides whether $E_1$ is in the orbit of $E_2$
        under the action of  $\grp{\alpha_1,\dots,\alpha_n}$.
        
        Moreover,  one can compute a generating set of the stabiliser of $E_1$ under the action of $\grp{\alpha_1,\dots,\alpha_n}$.
      \end{prop}

      \begin{proof}
        The previous lemma asserts that $\cale(G,N)$ is finite.
        By Lemma \ref{lem_extens}, 
        we can compute the action of $\grp{\alpha_1,\dots,\alpha_n}$ on $\cale(G,N)$ 
        in the sense of Lemma \ref{lem_probleme_orbite} (which does not require to be able to enumerate $\cale(G,N)$),
        and thus decide whether $E_1$ and $E_2$ are in the same orbit or not, and compute a generating set of 
        the stabiliser of $E_1$.
      \end{proof}

\subsection{Algorithmic tools for hyperbolic groups}

\subsubsection{Basic algorithms}

Let $G$ be a hyperbolic group. 
Given a finite presentation of $G$, one can compute  a hyperbolicity constant $\delta$ of the property of $\delta$-thin triangles 
(\cite{Gromov_hyperbolic,Papasoglu_algorithm}).   Let us remark that there are many classical characterisations of hyperbolicity, where the constant slightly vary when passing from a definition to another (see \cite{CDP}); this is completely harmless since the formulas for changing the constant are explicit.
Many algorithms are explicitly defined from a presentation of $G$  and a hyperbolicity constant $\delta$ for this presentation.
Since $\delta$ itself is computable in terms of a presentation, one can view these algorithms as
taking only a presentation of $G$ as input.
These problems include the word problem, the conjugacy problem, 
and more generally, the problem of satisfiability of finite systems 
of equations and inequations, see \cite{DG1}. 
 Similarly, one can compute representatives of every 
conjugacy class of finite subgroup of $G$ since every such group has a conjugate consisting of elements of 
length at most $20\delta$; one can also solve the root problem, \ie determine whether  an element is a proper power  \cite{Lys_algorithmic}.

\begin{lemma}\label{lem_cent_finite}
  There is an   
algorithm that computes a set of generators 
  of the centraliser, and of the normaliser, of any given finite subgroup 
in a hyperbolic group.
\end{lemma}

\begin{proof}
  Let $F<G$ be a finite subgroup of a hyperbolic groups $G$.
  By  \cite[Proposition 3.9]{BridsonHaefliger_metric} 
  (and its proof, in \cite[Proposition  4.15, pp. 477-478]{BridsonHaefliger_metric}),    
  its centraliser is quasi-convex  and in fact, generated by elements of length 
  bounded by an explicit constant, depending on the hyperbolicity constant, and the length of the given elements. 
  After trying all these elements we obtain a generating set.

Now the normaliser $N$ of $F$ maps  
to a subgroup $\bar{N}$ of $Aut(F)$, with kernel  
the centraliser of $F$.
For each $\alpha\in Aut(F)$, using a solution to the simultaneous conjugacy problem in $G$, it is possible to check whether $\alpha \in \bar{N}$, and if so, to find an element  $g_\alpha \in N$ inducing $\alpha$.
A generating set of $N$ is obtained  by taking a generating set of the centraliser, and, for each $\alpha \in \bar{N}$, an element  $g_\alpha$ as found above.
\end{proof}

Let $G$ be a hyperbolic group and $\grp{S}$ a finite generating set.
Let $(S\cup S\m)^*$ be the free monoid of all words on $S\cup S\m$, and
$\pi:(S\cup S\m)^*\ra G$ the natural projection.
A \emph{regular language} of $(S\cup S\m)^*$ is a subset recognised by a finite automaton.

Recall that a $(\lambda,\mu)$-quasigeodesic in  a space $(M,d)$  is a path $p:[\alpha,\beta] \to M$ such that,  for all $t,t' \in [\alpha,\beta]$ one has  $(1/\lambda) |t'-t| - \mu \leq  d(p(t),p(t'))  \leq \lambda |t'-t| +\mu$.

A word over the generators of a group labels a path in its Cayley graph. A word is said quasi-geodesic if the labelled path is so. 
\begin{dfn}\label{dfn_qiers}
  A  subset $\calr\subset G$ is a \emph{\qie rational subset} if there exist $\lambda\geq 1,\mu\geq 0$
  and a regular language $\Tilde\calr\subset(S\cup S\m)^*$ consisting of $(\lambda,\mu)$-quasigeodesics
  such that $\pi(\Tilde\calr)=\calr$.
\end{dfn}

Given a system of equation and inequations in $G$ over the set of variables $X$,
a set of \qie \emph{rational constraints} is a family $(\calr_x)_{x\in X}$ of \qie rational languages
indexed by the variables.
A \emph{solution} of the system of equations and inequations with these rational constraints
is a solution $(g_x)\in G^X$ of the system of equations and inequations such that
for all $x\in X$, $g_x\in\calr_x$.

We will use the following result from \cite{DG1}.

\begin{theo}\label{theo;eq}
  There is an explicit algorithm that, given a hyperbolic group, 
  with a finite system of equations, inequations, 
  and \qie rational constraints, determines whether there is a solution or not.
\end{theo}

\subsubsection{Virtually cyclic groups and $\Z$-groups}

Because we will treat separately phenomenas related to finite groups, 
and to  infinite virtually cyclic groups, we make the convention that a virtually cyclic group is infinite by definition.

Recall that a virtually cyclic group $H$ can be written as an extension of exactly one of the following two forms:
$$N\normal H \onto \bbZ \quad\text{or}\quad N\normal H \onto D_\infty$$
where $N$ is finite and $D_\infty=\bbZ/2*\bbZ/2$ denotes the infinite dihedral group.
In the first case, it has infinite centre, we say that $H$ is a \emph{$\Z$-group}.
In the second case, $H$ has finite centre and  we say that $H$ is of \emph{dihedral type}.
Note that in both cases, $N$ can be characterised as the maximal finite  
normal subgroup of $H$.

Given a group $G$, we denote by $\Z$ the family of its $\Z$-subgroups, 
and by $\Zmax$ the family of $\Z$-subgroups of $G$ which are maximal for the inclusion.

 Let $G$ be a hyperbolic group. 
 Given a  virtually cyclic group $H\subset G$, 
 there exists a unique
largest virtually cyclic subgroup $VC(H)$ of $G$ 
containing $H$ (the stabiliser of the pair of points fixed by $H$ in the boundary). 
If $H$ is a $\Z$-group,
there also exists a unique $\Zmax$-subgroup $\Zmax(H)\subset G$  containing $H$: this is  
the pointwise fixator of the same pair of points in the boundary.
It has index at most 2 in $VC(H)$.

\begin{lem}\label{lem_VC2}
There is an algorithm that given a finite set $S\subset G$,
\begin{itemize}
\item decides whether $\grp{S}$ is virtually cyclic, and whether $\grp{S}$ is a $\Z$-group
\item if $\grp{S}$ is virtually cyclic, it computes its maximal finite normal subgroup and a presentation of $\grp{S}$
\item if $\grp{S}$ is virtually cyclic (resp. a $\Z$-group), it computes $VC(\grp{S})$ (resp. $\Zmax(\grp{S})$)
by giving a generating set, its maximal finite normal subgroup and a presentation, and in particular, it indicates whether $\grp{S}$ is $\Zmax$.
\end{itemize}
\end{lem}

\begin{proof}
  In order to prove the two first points, we propose two algorithms, one of which will terminate if $\grp{S}$ is virtually cyclic,
  will produce a presentation, a maximal finite normal subgroup, and will indicates whether the group is a $\Z$-group, and the
  other will terminate if the group is non-virtually cyclic.

  As previously remarked, one can compute representatives of every conjugacy class of finite subgroup. Then one can find a maximal
  finite subgroup $N$ of $G$ normalised by $\grp{S}$. Indeed, for each representative of finite group $F$, we can solve the
  disjunction of systems of equations with one unknown $x$ that states that $ \forall s\in S, \, ( F^x )^s = F^x$.  Consider the
  largest $F$ for which this has a solution $x$, 
  then $N=F^x$.

  Now one can easily check whether $ S \subset N$ or not.  Hence, we assume that $\grp{S}$ is infinite.   Although it may
  happen that $\grp{S}$ does not contain $N$, its maximal finite normal subgroup is contained in $N$, and one can obtain the list
  of subgroups of $N$ that are normalised by the elements of
  $S$. 

  We now look for a generating set of $\grp{S}$ of the form $M\cup\{g\}$, or $M\cup\{g,\sigma\}$ where $M<N$ is normalised by $S$,
  where $g,\sigma \notin M$, and where $\sigma^2 \in M, \, (\sigma g \sigma g) \in M$. This can be done, in parallel for each
  possible such $M$, by enumeration of the elements $g,\sigma$ of $\grp{S}$, and for each family of elements, checking (using a
  solution to the word problem) whether it generates a group containing $S$. This process must terminate if the group $\grp{S}$ is
  indeed virtually cyclic. When it terminates, it shows $\grp{S}$ as a virtually cyclic group, provides the maximal finite normal
  subgroup, and the quotient (note that $g$ is necessarily of infinite order, otherwise $\grp{S}$ would have been finite).  In
  particular it determines whether $\grp{S}$ is a $\Z$-group.   Once $N,g$ and maybe $\sigma$ have been determined, one
  easily computes a presentation of $\grp{S}$.

  We now discuss the case where the enumeration above does not terminate. In this case $\grp{S}$ must be non-elementary in $G$,
  and therefore it must contain a non-abelian free subgroup. In particular there exists two elements $g,h \in \grp{S}$ such that
  $[g^2,h^2]\notin N$. Note that this cannot happen if $\grp{S}$ is virtually cyclic, since the quotient by $\grp{S}\cap N$ would
  have a cyclic subgroup of index $2$.  Since an enumeration process will find two such elements if they exists, we obtain another
  algorithm that terminates if $\grp{S}$ is not virtually cyclic. This shows the two first points of the lemma.

  For the third point, there are in principle three cases: 
  first $\grp{S} = \grp{M,g}$ is a $\Z$-group, and we look for $\Zmax(\grp{S})$;
  second, $\grp{S} = \grp{M,g}$ is a $\Z$-group, and we look for $VC(\grp{S})$;
  and third, $\grp{S}=\grp{M,g,\sigma}$ is not a $\Z$-group and we look for $VC(\grp{S})$. 
  Since $S$ normalizes $N$, the group $\grp{S,N}$ is virtually cyclic,
  and $N$ is the maximal finite normal subgroup of $VC(\grp{S})$. 
  If additionnally $\grp{S}$ is a $\Z$-group, then $N$ is the maximal finite normal subgroup of $\Zmax({N,g})$.
 
  In any case, for each element $n\in N$, compute a maximal root $h_n$ of $gn$ 
  (\ie a generator of a maximal cyclic group containing $gn$) \cite{Lys_algorithmic}
  and $k_n>0$ such that $h_n^{k_n}=gn$. For $n$ such that $k_n$ is maximal, one has $\Zmax(\grp{N,g})=\grp{N,h_n}$,
  which solves the first case.
  In the third case, $VC(\grp{S})=\grp{M,h_m,\sigma}$.
  To solve the second case, one checks the existence of $\sigma\in G$ such that $\sigma g \sigma\m= gm$ for some $n\in N$.
  If it does not exist, $VC(\grp{S})=\Zmax(\grp{S})=\grp{M,h_n}$, and if it does, $VC(\grp{S})=\grp{M,h_n,\sigma}$.  
\end{proof}

We will also need algorithms concerning virtually cyclic groups. 

\begin{lemma} \label{lem_autZ}
 \begin{enumerate*}
 \item   There is an algorithm that, given a presentation of a $\Z$-group, computes its (finite) automorphism group.

 \item There is an algorithm that,  given two presentations of $\Z$-groups, 
 decides whether they are isomorphic,   and lists the set of all isomorphisms between them.

\item 
 There is an algorithm that, given a presentation of a virtually cyclic group with finite centre, computes its (finite) outer automorphism group.

\item 
  There is an algorithm that,  given two presentations of virtually cyclic groups with finite centre, 
 decides whether they are isomorphic,  and lists the set of all isomorphisms between them, up to inner automorphisms.     
 \end{enumerate*}
\end{lemma}

Each isomorphism is described by giving for each generator, a word representing its image.
The meaning of the third statement is that one can produce a finite list of automorphisms of $G$
which is in bijection with the group of outer automorphisms of $G$.

\begin{proof}
Assertions 1 and 3 follow from assertions 2 and 4.

Given virtually cyclic groups $H_1,H_2$, one can compute their maximal finite normal subgroup $N_i$ by the previous lemma. Let us also remark that  we can easily deduce explicit solutions to the word problem $H_1, H_2$.

Being characteristic, we can assume that $N_1\simeq N_2$, and we can compute the finitely many isomorphisms between them.

If $H_i/N_i$ is infinite cyclic  for $i=1,2$,
then let $t_i \in H_i$ be a preimage of a generator of the cyclic group $H_i/N_i$. 
Any isomorphism $H_1 \to H_2$ is determined by the induced isomorphism $N_1\to N_2$ and by the image of $t_1$, necessarily of the form 
 $t_2^{\pm 1} n_2$, for some $n_2 \in N_2$.  One can check whether such a map extends to a morphism using 
  the presentation of $H_1$ and the solution of the word problem in $H_2$.
  Any such morphism is clearly an isomorphism, so this makes the list of
  all isomorphisms $H_1\ra H_2$.
Using the solution of the word problem in $H_2$, we can remove repetitions in this list.

If $H_i/N_i$ is infinite dihedral for $i=1,2$, 
consider $\sigma_i,\tau_i\in H_i$ mapping to elements of order $2$ generating $H_i/N_i$ (one can find such elements by enumeration).
Since any (unordered) pair of elements of order $2$ generating the infinite dihedral group is conjugate to the standard one,
any isomorphism $H_1\ra H_2$ is conjugate to one sending $\{\sigma_1,\tau_1\}$ to $\{\sigma_2n_2,\tau_2n'_2\}$
for some $n_2,n'_2\in N_2$. As above, one can extract from this list the mappings which extend to an isomorphism,
 and remove repetitions using the simultaneous conjugacy problem in $H_2$.
\end{proof}

\subsection{Marked and unmarked peripheral structures}\label{sec_periph}

Let $G$ be a group, and $\calE$ be a class of finitely generated subgroups of $G$ invariant under automorphisms of $G$. 
Typically, $\calE$ will be the class of finite subgroups, 
$\Z$, or $\Zmax$-subgroups.
If $P$ is a group in $\calE$, we denote by $[P]$ its conjugacy class in $G$.
An \emph{unmarked $\calE$-peripheral structure} on $G$ is a  
finite  unordered set $\calp=\{[P_1],\dots,[P_p]\}$
of conjugacy classes of subgroups $P_i$ in $\calE$.
When the context is clear, we just say \emph{peripheral structure} instead of \emph{unmarked $\calE$-peripheral structure}.
The groups $P_i$ and their conjugates are called the \emph{peripheral subgroups}.
Algorithmically, one represents an unmarked peripheral structure by choosing a generating set of each $P_i$.

In a graph of groups, each vertex group  
inherits a natural peripheral structure
consisting of the conjugacy classes of the images of the incident edge groups.
This fundamental example will be the source of most peripheral structures we will consider.

Given $G,G'$ two groups  with unmarked peripheral structures $\calp=\{[P_1],\dots,[P_p]\}$ and $\calp'=\{[P'_1],\dots,[P'_{p'}]\}$,
an isomorphism $\phi:(G;\calp)\ra (G';\calp')$ is an isomorphism $\phi:G\ra G'$
such that $\{[\phi(P_1)],\dots,[\phi(P_p)]\}=\{[P'_1],\dots,[P_{p'}']\}$ 
as unordered sets (in particular, $p=p'$ if   
the $[P_i]$'s and the $[P'_i]$'s are distinct).
We denote by $\Autu(G;\calp)$ the group of automorphisms of $(G;\calp)$ (where $u$ recalls that we are talking about an \emph{unmarked} peripheral structure)
\ie the subgroup of $\Aut(G)$ permuting the conjugacy classes of $P_1,\dots,P_n$.
Of course, $\Autu(G;\calp)$ contains all inner automorphisms, and we denote by $\Outu(G;\calp)$ its image in $\Out(G)$.

If $S=(s_1,\dots,s_n)\in G^n$ is a tuple, its conjugacy class $[S]$ is
the set of tuples $(s'_1,\dots,s'_n)=(s_1^g,\dots,s_n^g)$ for some $g\in G$.
We use the notation $\#S=n$.
A \emph{marked  $\calE$-peripheral structure} $\calp=([S_1],\dots,[S_p])$ on $G$ 
is a tuple of conjugacy classes of tuples $S_1\in G^{n_1},\dots,S_p\in G^{n_p}$, such that $\grp{S_i}$ lies in $\calE$.
When the context is clear, we also just say \emph{peripheral structure} instead of \emph{marked $\calE$-peripheral structure}.

Given $G,G'$ two groups with marked peripheral structures $\calp=([S_1],\dots,[S_p])$ and $\calp'=([S'_1],\dots,[S'_p])$,
an isomorphism  (resp.\ a homomorphism) $\phi:(G;\calp)\ra (G';\calp')$ is an isomorphism (resp.\ a homomorphism) $\phi:G\ra G'$
such that $[\phi(S_i)]=[S'_i]$ for all $i=1,\dots,p$,  \ie such that $\phi(S_i)=S'_i{}^{g_i}$ for some $g_i\in G'$. 
The groups $\grp{S_i}$ and their conjugates are also called the \emph{peripheral subgroups} of $\calp$.
We denote by $\Autm(G;\calp)$ the group of automorphisms of $(G;\calp)$ (where the subscript $m$ recalls that we are talking about a \emph{marked} peripheral structure),
\ie the subgroup of $\Aut(G)$ which maps each $S_i$ to a conjugate, \ie whose restriction to $\grp{S_i}$ is the conjugation by some element of $G$.
As above, $\Autm(G;\calp)$ contains all inner automorphisms, and we denote by $\Outm(G;\calp)$ its image in $\Out(G)$.

A marked peripheral structure $\calp_m=([S_1],\dots,[S_p])$ naturally induces an unmarked peripheral structure $\calp_u=\{[\grp{S_1}],\dots,[\grp{S_p}]\}$.
We also say that $\calp_m$ is a marking of $\calp_u$.
Note that $\Autm(G;\calp_m)\subset\Autu(G;\calp_u)$.

\begin{lem}\label{lem_markings}
  Let $(G,\calp_m)$ be a group with marked peripheral structure $\calp_m=(S_1,\dots,S_p)$.
Let $\calp_u=\{[\grp{S_1}],\dots,[\grp{S_p}]\}$ be the induced unmarked peripheral structure.

If $\Out(\grp{S_i})$ is finite for all $i$ then $\Out_m(G,\calp_m)$ has finite index in $\Out_u(G,\calp_u)$.
\end{lem}

This applies in particular if peripheral subgroups are virtually cyclic or finite. Then
$\Out_m(G,\calp_m)$ is finitely generated if and only if $\Out_u(G,\calp_u)$ is.

\begin{proof}
The marked peripheral structure $\calp_m=(S_1,\dots,S_p)$,
can be viewed as an element of $\calm=(G^{\#S_1}\times\dots \times G^{\#S_p})/G^p$ where $G^p$
acts on each factor by conjugation.
Let $\mathit{Markings}(\calp_u)\subset\calm$ be the set of marked peripheral structures inducing the unmarked peripheral structure $\calp_u$.
$\Out_u(G,\calp_u)$ acts on $\mathit{Markings}(\calp_u)$, and since $\Out(\grp{S_i})$ is finite,
each orbit is finite. Since $\Out_m(G,\calp_m)$ is the stabilizer of an element
of $\mathit{Markings}(\calp_u)$, it has finite index in $\Out_u(G,\calp_u)$.
\end{proof}

\begin{dfn}[Extended isomorphism problem]\label{dfn_eip}
  Consider a class $\calg$ of groups with unmarked (resp.\ marked) 
peripheral structures,
 such that $\Out_u(G,\calp)$ (resp. $\Outm(G,\calp)$) is finitely generated for all $(G,\calp)\in \calg$.
Then the \emph{extended isomorphism problem} for $\calg$ consists of the two following problems: \begin{enumerate}
\item given $(G,\calp)$, $(G',\calp')\in \calg$, decide whether $(G,\calp)\simeq(G',\calp')$
\item given $(G,\calp)\in \calg$, compute a  finite generating set of $\Outu(G,\calp)$ (resp. of $\Outm(G,\calp)$).
\end{enumerate}
\end{dfn}

\begin{rem}
  What we mean in the second assertion, is computing a finite subset in $\Aut(G)$ (each automorphism being encoded
by giving a word representing the image of each generator), whose image
generates  $\Outu(G,\calp)$ (resp.\ $\Outm(G,\calp)$). 
Clearly, computing a generating set for $\Outu(G,\calp)$ (resp. of $\Outm(G,\calp)$)
is equivalent to computing a generating set for $\Aut_u(G,\calp)$ (resp. of $\Aut_m(G,\calp)$).
\end{rem}

\begin{lem}\label{lem_eqv_um}
Let $\cale$ be the class of finite or virtually cyclic groups.
  Let $\calg_u$ be a class of groups with unmarked $\cale$-peripheral structures, with 
$\Out_u(G,\calp_u)$ finitely generated for all $(G,\calp_u)\in \calg_u$.
Let  $\calg_m$ be the class of groups with marked $\cale$-peripheral structures
whose induced unmarked peripheral structure lies in $\calg_u$.
 
Assume that for groups in  $\calg_u$, we have a solution of the simultaneous conjugacy problem,
and for subgroups in $\cale$, we can find a presentation from a generating system.

Then the extended isomorphism problem is solvable for $\calg_u$ if and only if it is solvable for $\calg_m$.
\end{lem}

\begin{proof}
Note that $\Out_m(G,\calp_m)$ is finitely generated for all $(G,\calp_m)\in \calg_m$ by Lemma \ref{lem_markings}.

Assume that the extended isomorphism problem is solvable for unmarked peripheral structures.
Consider $(G,\calp_m),(G',\calp'_m)\in \calg_m$, and we want to decide whether $(G,\calp_m)\simeq(G',\calp'_m)$.
We denote by $\calp_u,\calp'_u$ be the  unmarked peripheral structure defined by $\calp_m,\calp'_m$.
We can decide whether $(G,\calp_u)\simeq (G',\calp'_u)$, and we can assume that it is the case.
Using some isomorphism $f$ between $(G,\calp_u)$ and $ (G',\calp'_u)$ (one can compute one by enumeration),
we can define $\calp''_m=f\m(\calp'_m)$ a marked peripheral structure of $G$ inducing $\calp_u$.
Write $\calp_m=(S_1,\dots,S_p)$ and $\calp''_m=(S''_1,\dots,S''_p)$ where can assume $\#S_i=\#S''_i$.
Using the notations of the proof of Lemma \ref{lem_markings},
$\calp_m,\calp''_m$ are elements of $\mathit{Markings}(\calp_u)$, and 
we have to decide whether they are in the same orbit under $\Outu(G,\calp_u)$.
We can compute  some generating set $S_u$ of $\Outu(G,\calp_u)$.
Since $\Out(\grp{S_i})$ is finite for all $i$, 
the orbit of $\calp_m$ and $\calp''_m$ in $\mathit{Markings}(\calp_u)$ are finite.
Using the simultaneous conjugacy problem, one can decide whether two tuples represent the same element of $\mathit{Markings}(\calp_u)$.
Applying Lemma \ref{lem_probleme_orbite} to $X$ consisting of the union of these two orbits,
one can decide whether these orbits coincide, and compute a system of generators of $\Outm(G;\calp)$.
This concludes the proof of the first implication.

Conversely, assume that the extended isomorphism problem is solvable for marked peripheral structures.
Consider $(G,\calp_u),(G',\calp'_u)\in \calg_u$, and we want to decide whether $(G,\calp_u)\simeq(G',\calp'_u)$.
We first note that we can decide when two subgroups $\grp{S},\grp{S'}$ in the class $\cale$ are conjugate.
Indeed one can find a presentation of these groups by hypothesis, and
one can list all isomorphisms $\grp{S}\ra \grp{S'}$ up to conjugacy by Lemma \ref{lem_autZ},
and then solve the simultaneous conjugacy problem to check whether such an isomorphism is induced by a conjugation.
This way, we can decide which peripheral subgroups in $\calp_u$ (resp. in $\calp'_u$)
are conjugate to each other in $G$ (resp. in $G'$), and remove redundant conjugacy classes.
Write $\calp_u=\{[P_1],\dots,[P_p]\}$ and $\calp'_u=\{[P'_1],\dots,[P'_{p}]\}$ (we can assume that $\#\calp_u=\#\calp'_u$).
 Choose a generating system $S_i$ of $P_i$. This defines a marked peripheral structure $\calp_m$ inducing $\calp_u$. 
Compute presentations of each peripheral subgroup.
For each permutation $\sigma$ of $\{1,\dots,n\}$,
make a list of all conjugacy classes of isomorphisms $f_i:P_i\ra P'_{\sigma(i)}$  (Lemma \ref{lem_autZ}).
For each such choice of $\sigma$ and of isomorphisms $f_i$, define $S'_i=f_i(S_i)$, and let $\calp'_m$ be 
the corresponding marked peripheral structure of $G'$.
Then $(G,\calp_u)\sim(G',\calp'_u)$ if and only if $(G,\calp_m)\sim(G',\calp'_m)$ for some choice,
which solves the unmarked isomorphism problem.
Computing a generating set of $\Outu(G,\calp_u)$ is similar: fix a marking $\calp_m$ of $\calp_u$ as above,
and consider all other markings $\calp'_m$ corresponding to some permutation $\sigma$ and isomorphisms $f_i$.
One can decide if $(G,\calp_m)\simeq(G,\calp'_m)$, and in this case, find an isomorphism $\alpha$.
Such $\alpha$ lies in $\Outu(G,\calp_u)$, and adding all those elements $\alpha$
to a generating set of $\Outm(G,\calp_m)$ gives a generating set of $\Outu(G,\calp_u)$.
\end{proof}

\begin{lem}\label{lem_periph_fini}
  Let $\calg$ be a class of hyperbolic groups groups with marked peripheral structures whose peripheral subgroups are infinite,
and for which the extended isomorphism problem is solvable.
Let $\calg_F$ be the class of groups $(G,\calp\cup\calq)$ with marked peripheral structures
such that $(G,\calp)\in\calg$, and all peripheral subgroups in $\calq$ are finite.

Then the extended isomorphism problem is solvable for $\calg_F$.
\end{lem}

\begin{proof}
Consider $(G;\calp\cup\calq),(G';\calp'\cup\calq')\in\calg_F$.
Since peripheral subgroups of $\calp$ are infinite and since those of $\calq$ are finite,
 $(G;\calp\cup\calq)\simeq(G';\calp'\cup\calq')$ implies $\#\calp=\#\calp'$, $\#\calq=\#\calq'$ and
  $(G;\calp)\simeq(G';\calp')$.
One can decide whether $(G;\calp)\simeq(G',\calp')$, and we can assume that this is the case.
Moreover, one can compute a finite generating set of $\Out(G,\calp)$.

Consider $f:G\ra G'$ an isomorphism sending $\calp$ to $\calp'$ (one can construct such $f$ by enumeration and
using the simultaneous conjugacy problem).
Let $\calq''=f\m(\calq')$ be the pullback of the marked peripheral structure.
Denote by $\calq=(S_1,\dots,S_p)$ and $\calq''=(S''_1,\dots,S''_p)$ the marked peripheral structures.
Let $\Periph$ be the following set of marked peripheral structures:
$\Periph$ is set of tuples $(\Sigma_1,\dots,\Sigma_p)\in \calm=(G^{\#S_1}\times\dots \times G^{\#S_p})/G^p$ 
such that $\#\grp{S_i}<\infty$.
 The simultaneous conjugacy problem allows one to decide when two tuples represent the same marked peripheral structure.
We view $\calq$ and $\calq''$ as two elements of $\Periph$.
Then $(G;\calp\cup\calq)\simeq(G';\calp'\cup\calq')$ if and only if $\calq$ is in the same orbit as $\calq''$ 
under the natural action of $\Out(G,\calp)$ on $\Periph$.

Since $G$ has only finitely many conjugacy classes of finite subgroups, $\Periph$ is finite.
Since we know a generating set of $\Out(G,\calp)$, the orbits of its action on $\Periph$ can be computed
by Lemma \ref{lem_probleme_orbite}.
Computing a system of generators of $\Outm(G;\calp\cup\calq)$ also follows from  Lemma
  \ref{lem_probleme_orbite}. 
\end{proof}

\subsection{Orbifolds and their mapping class groups}
\label{sec_orbi_mcg}

      Consider $\Sigma$  a $2$-orbifold with boundary, whose fundamental group is not virtually abelian (hence it is hyperbolic). 
      We assume that $\Sigma$ is \emph{of conical type}, \ie that all
      its singularities are cone points (in other words $\Sigma$ has no mirrors).
      The boundary subgroups of $\Sigma$ define an unmarked peripheral structure $\calb$ of $\pi_1(\Sigma)$.
      Choosing an ordering of the boundary components and a generator of each peripheral subgroup,
      we obtain a corresponding marked peripheral structure $\calb_m$.
      The group $\Autu(\pi_1(\Sigma);\calb)$ is the group of automorphisms of $\pi_1(\Sigma)$ preserving the 
      conjugacy class of the peripheral subgroups, and $\Aut_m(\pi_1(\Sigma);\calb)$ is its subgroup
      sending each generator of a boundary subgroup to a conjugate of itself.

      Let $\Sigma_0$ be the surface underlying $\Sigma$, with a marked point at each conical singularity,
      this point carrying a weight corresponding to the order of its isotropy group in $\Sigma$.
      Let us denote by $MCG^*(\Sigma)$ the group of isotopy classes of homeomorphisms $h$ of $\Sigma_0$ and preserving
      the weights of the marked points. This is an extended mapping class group since we don't
      assume $h$ to preserve the orientation (which makes sense only if $\Sigma$ is orientable).
      Homeomorphisms and isotopies are not required to be the identity on boundary components of $\Sigma_0$.
      On the other hand, we define $PMCG^*(\Sigma)$ (where $P$ stands for pure) as the subgroup of $MCG^*(G)$
      of isotopy classes of homeomorphisms whose restriction to each boundary component $b$
      maps $b$ to itself, preserving the orientation of $b$.
Note that we still don't assume that $h$ preserves the orientation (when $\Sigma$ is orientable) although this is necessarily the
case if the boundary of $\Sigma$ is non-empty.
      Any $h\in MCG^*(\Sigma)$ (resp. in $PMCG^*(\Sigma)$) induces an outer automorphism of $\pi_1(\Sigma_0)$ preserving the 
      unmarked (resp.\ marked) peripheral structure.
                                
      \begin{prop}\label{prop_extMCG} 
        The morphisms described above induce isomorphisms
        $$MCG^*(\Sigma)\xra{\,\textstyle{}_{\sim}} \Outu(\pi_1(\Sigma);\calb),\qquad 
        PMCG^*(\Sigma)\xra{\,\textstyle{}_{\sim}} \Outm(\pi_1(\Sigma);\calb_m).$$ 
        In particular, we have extensions
        $$1\to \pi_1(\Sigma) \to \Autu(\pi_1(\Sigma);\calb ) \to MCG^*(\Sigma) \to 1$$ 
        $$1\to \pi_1(\Sigma) \to \Autm(\pi_1(\Sigma);\calb_m ) \to PMCG^*(\Sigma) \to 1.$$ 
\end{prop}

      We give references.       
      The isomorphism is due to Dehn, Nielsen and Baer  for
      closed oriented surface groups, 
      and was generalised by  Magnus and Zieschang for surfaces with boundary,  
      by Maclachlan and Harvey for 
      orientable orbifolds with conical singularities and boundaries \cite[Theorem 1]{MacH_MCG}, 
      and by Fujiwara \cite[sec.3, p.281]{Fujiwara_outer}, for the non-orientable case
      (these authors consider the unmarked case, the marked case is similar).
      Maclachlan and Harvey also observed in \cite[Corollary 3]{MacH_MCG} that in the definition of the mapping class group,
      one can replace boundary components by marked points (of infinite weight).

Recall that the extended isomorphism problem consists in solving the isomorphism problem, and in
computing generators of the corresponding automorphism groups (Definition \ref{dfn_eip}).

\begin{prop}\label{prop_IP_orbifold} 
The extended isomorphism problem is solvable for the class of pairs $(\pi_1(\Sigma),\calb)$ 
  where $\Sigma$ is a compact $2$-orbifold of conical type  whose fundamental group is not virtually abelian,
  and $\calb$ is the unmarked peripheral structure defined by its boundary subgroups.
\end{prop}
                                
\begin{proof}
      An explicit generating system for $MCG^*(\Sigma)$  can  be obtained from Lickorish \cite{Lickorish_homeotopy} in the orientable case, 
      and Korkmaz \cite{Korkmaz_MCG} (see also Chillingworth \cite{Chillingworth_generator}) in the non-orientable case. 
      One can compute the effect of this generating system on a generating system of $\pi_1(\Sigma)$,
      and thus obtain a generating set for  $\Outu(\pi_1(\Sigma);\calb)$.
      The isomorphism problem among orbifold groups is solved as follows: each orbifold group has a standard presentation (with a corresponding peripheral structure)
      on which on can read the topology of the orbifold. Starting from  an arbitrary presentation of an orbifold group,
      using Tietze transformation, one will find  such a standard presentation, in which one can decide if the given peripheral structure
      coincide with the standard one. This procedure will terminate sometime, which solves the unmarked 
 extended isomorphism problem. \end{proof}

      By Lemma \ref{lem_eqv_um}, the same statement holds in the marked case.

\subsection{Bounded Fuchsian groups}\label{sec_bfg}

        Following \cite{Bo_cut}, we call \emph{bounded Fuchsian group} a non-elementary finitely generated group $G_0$ 
        having a proper discontinuous action by isometries without parabolics on $\bbH^2$.
        We don't ask for the action to be faithful, the kernel may be a finite group.
        Equivalently, a bounded Fuchsian group is a finite extension of a non-elementary convex co-compact discrete subgroup of $PSL_2(\bbR)$.
        Note that the kernel $F$ of the action of $G_0$ on $\bbH^2$ is the largest finite normal subgroup of $G_0$.

        The quotient $\ol G_0=G_0/F$ is the fundamental group of the orbifold $K/ \ol G_0$ where $K\subset \bbH^2$ is the convex core of $G_0$.
        We call $\ol G_0$ the \emph{orbifold group} of $G_0$.

        We say that $G_0$ is \emph{without reflection} if no element of $G_0$ acts as a reflection on $\bbH^2$,
        \ie if the quotient orbifold has no mirror.
        We also say that $G_0$ is of \emph{conical type} in this case.

             The stabiliser of a  boundary component of $K$ 
             is called a \emph{boundary subgroup}.
        Boundary subgroups are virtually cyclic subgroups of $G_0$, and $\Z$-subgroups if $G_0$ is without reflection.
        The set of boundary subgroups consists of finitely many conjugacy classes of virtually cyclic subgroups $B_1,\dots,B_n$,
        which defines an unmarked peripheral structure of $G_0$.
        This peripheral structure is empty if and only if $G_0$ acts cocompactly on $\bbH^2$.
  We will mostly work with a marked peripheral structure $\calb$ inducing this unmarked peripheral structure
        (see section \ref{sec_periph} for definitions).

We will need a solution to the marked isomorphism problem for bounded Fuchsian groups
with a marking of their natural peripheral structure.
Recall that we represent algorithmically the groups $G$, $G'$ by some presentations, and the marked peripheral structures $\calp_m,\calp'_m$
by tuples of words in the generators.

\begin{prop}\label{prop_IP_BFG}
The extended isomorphism problem is solvable for bounded Fuchsian group without reflection. More precisely:

  There is an explicit algorithm that takes as input two bounded Fuchsian groups $(G_1;\calb_{1})$, $(G_2;\calb_{2})$  
without reflection 
  with marked peripheral structure,
  and decides whether  $(G_1;\calb_{1})\simeq (G_2;\calb_{2})$.
              
  Moreover, there is an algorithm that, given a bounded Fuchsian group $(G,\calb)$ with a marked peripheral structure,  
  computes a set of generators for the group $\Autm(G;\calb)$.
\end{prop}

\begin{proof}
Let $N_1, N_2$ be the (unique) maximal finite normal subgroups of the two bounded Fuchsian groups $G_1, G_2$.  
Let $\ol G_i = G_i/N_i$, so we have $1 \to N_i \xra{j_i} G_i \xra{\pi_i} \ol G_i \to 1$.
Since the finite groups $N_i$ can be computed, one can find a presentation of $\ol G_i$.

Denote the marked peripheral structures by $\calb_{i}=(S_1^{(i)},\dots,S_p^{(i)})$. 
We can obviously assume that
$\calb_{1}$ and $\calb_{2}$ have the same number $p$ of tuples, and that $\#S_i^{(1)}=\#S_i^{(2)}$ for all $i=1,\dots,p$
as otherwise, $(G;\calb_m)\not\simeq (G';\calb'_m)$.
Whether  the order preserving map
$S_j^{(1)}\ra S_j^{(2)}$ extends to an isomorphism $\phi_j:\grp{S_j^{(1)}}\ra\grp{S_j^{(2)}}$ can be checked
(see Lemma \ref{lem_VC2} and \ref{lem_autZ}),
so we can assume that $\phi_j$ exists.

Since the peripheral subgroups are $\Zmax$, they contain the normal finite group $N_i$, and the quotient is cyclic. 
Denote by $\ol\calb_{i}$ the image of $\calb_{i}$ in $\ol G_i$.
Of course, any isomorphism
 $\phi:(G_1;\calb_{1})\ra (G_2;\calb_{2})$ sends $N_1$ to $N_2$ 
  and induces an isomorphism $\ol \phi:(\ol G_1,\ol \calb_{1})\ra (\ol G_2,\ol \calb_{2})$. 

By Proposition \ref{prop_IP_orbifold},  
we can determine if there exists an isomorphism $\ol \phi:(\ol G_1,\ol \calb_{1})\ra (\ol G_2,\ol \calb_{2})$,
  and produce one if it exists. If there is no such isomorphism, we are done.
Still using  Proposition \ref{prop_IP_orbifold},  
compute a generating set of  $\Autm(G_1,\ol \calb_{1})$.

We list all isomorphisms $\nu:N_1\ra N_2$, and for each of them, 
we look for isomorphisms  $\phi:G_1\ra G_2$ which coincide with $\nu$ on $N_1$.
In a first step, we don't ask $\phi$ to preserve the peripheral structures.

\begin{lem}\label{lem_obs3}
One can decide whether there exists an isomorphism 
$\phi':G_1\ra G_2$ whose restriction to $N_1$ is $\nu$,
and such that the induced isomorphism $\ol\phi':\ol G_1\ra \ol G_2$ maps the marked peripheral structure $\ol \calb_{1}$ to $\ol\calb_{2}$.
We can find such a $\phi'$ if it exists.

Moreover, one can construct a generating set of the group $\Aut_0(G_1)$ of all automorphisms of 
$G_1$ whose restriction to $N_1$ is the identity, and whose induced automorphism on $\ol G_1$
preserves $\ol\calb_{1}$.
\end{lem}

\begin{proof}
We already have an isomorphism $\ol \phi:(\ol G_1,\ol \calb_{1})\ra (\ol G_2,\ol \calb_{2})$.
One has to decide the existence of $\ol\alpha\in \Autm(\ol G_1,\ol \calb_{1})$,
such that $\ol\phi\circ\ol\alpha$ lifts to a isomorphism $\phi':G_1\ra G_2$ whose restriction to $N_1$ is $\nu$.
The isomorphisms $\ol\phi$, and $\nu$ allow us to view $G_2$ as an extension of $\ol G_1$ by $N_1$:
$1\to N_1\xra{j_2\circ\nu} G_2 \xra{\ol\phi\m\circ\pi_2} \ol G_1$, and 
we are asking whether the two extensions $G_1,G_2\in \cale(\ol G_1,N_1)$ 
are in the same orbit under the action of $\Autm(\ol G_1,\ol \calb_{1})$
(see Section \ref{subsec_extensions}).
Since we have computed a generating set of  $\Autm(\ol G_1,\ol \calb_{1})$,
Proposition \ref{prop_iso_extension} allows us to decide whether this holds or not, and gives such $\ol\alpha$ if it exists.
One finds a lift of  $\ol\phi\circ\ol\alpha$ inducing $\nu$ by trying all possibilities as in Proposition \ref{prop_iso_extension}.

The same proposition gives a generating set of the stabiliser of the extension $G_1$ under the action of $\Autm(\ol G_1,\ol \calb_{1})$.
By listing all lifts to $G_1$ of this generating set, one gets a generating set of $\Aut_0(G_1)$.
\end{proof}

If no isomorphism $\phi'$ as in Lemma \ref{lem_obs3} exists, we are done.
Otherwise, we now modify $\phi'$ to ensure that it preserves the peripheral structures.

Fix $S_j=S_j^{(1)}=(s_1,\dots,s_r)\in G_1^r$ a tuple of $\calb_{1}$,
and $S'_j=S_j^{(2)}\in G_2^r$ be the corresponding tuple of $\calb_{2}$.
Let $\ol S_j \in \ol G_1^r$ be the image of $S_j$ in $\ol G_1$
 and $X_j=s_1N_1\times \dots \times s_rN_1 \subset G_1^r$ 
be the finite set of tuples in $G_1^r$ whose image in $G_1$ is $\ol S_j$.
Let $\sim_j$ be the equivalence relation on $X_j$ defined by
$(t_1,\dots,t_r)\sim (t'_1,\dots,t'_r)$ if and only if $(t_1,\dots,t_r)=(t'_1{}^g,\dots,t'_r{}^g)$ for some $g\in G_1$.
Note that $\Aut_0(G_1)$ acts on $X_j/\sim_j$. 
Indeed, since any $\psi\in\Aut_0(G_1)$ preserves the marked peripheral structure $\ol \calb_1$ on $\ol G_1$, 
there exists $g\in G_1$ such that $\psi(S_j)^g\in X_j$, and its equivalence class in $X_j/\sim_j$ does not depend on the choice of $g$.

Since $\ol\phi'$ preserves marked peripheral structures, 
there exists $g\in G_1$ such that $\phi'{}\m(S'_j)^g\in X_j$, (and the choice of $g$ does not change the
class of $S'_j$ modulo $\sim_j$). 

The existence of an isomorphism $\phi'':(G_1,\calb_1)\ra (G_2,\calb_2)$ 
is equivalent to the existence of $\psi\in \Aut_0(G_1)$
 mapping $S_j$ to $S'_j$ in $X_j/\sim_j$ for all $j$.
Indeed, if $\phi''$ exists then $\psi=\phi'{}\m\phi''\in \Aut_0(G_1)$ sends $S_j$ to $S'_j$ in $X_j/\sim_j$,
and conversely, if $\psi$ sends  $S_j$ to $S'_j$ in $X_j/\sim_j$, then one can take $\phi''=\phi'\psi$.

Since the simultaneous conjugacy problem is solvable, the equivalence relation $\sim_j$ and the action of $\Aut_0(G_1)$ on $\prod_j X_j/\sim_j$ are computable.
Since we know a generating set for $\Aut_0(G_1)$, Lemma \ref{lem_probleme_orbite} allows us to
decide the existence of $\psi$, and proves the first assertion of the proposition.
The same lemma allows us to compute the stabiliser of $(S_1^{(1)},\dots,S_r^{(1)})\in \prod_j X_j/\sim_j$
which is exactly $\Autm(G_1;\calb_{1})$.
This ends the proof of Proposition \ref{prop_IP_BFG}.
\end{proof}

  \subsection{Trees, graph of groups decompositions}

\subsubsection{Generalities}

A \emph{$G$-tree} is a simplicial tree endowed with an action of $G$.
All actions are supposed to be without inversion: no edge is mapped to itself reversing the orientation.
A $G$-tree is \emph{non-trivial} if no point of $T$ is fixed by $G$,
and \emph{minimal} if $T$ has no  proper $G$-invariant subtree.
All $G$-trees considered are minimal, unless otherwise mentioned.
If $(G;\calp)$ is a peripheral structure of $G$, then an action of the pair $(G;\calp)$ on a tree
(or an action of $G$ relative to $\calp$)
is a $G$-tree such that all peripheral subgroups of $\calp$ are elliptic (\ie fix a point) in $T$. 
This makes sense whether $\calp$ is a marked or unmarked peripheral structure.

A $G$-tree $T$ is \emph{irreducible} if it is non-trivial, and no end and no line of $T$ is $G$-invariant
(in other words $T$ is not abelian and not dihedral). If $G$ is hyperbolic, and not virtually cyclic,
any non-trivial $G$-tree with elementary edge stabiliser is irreducible \cite[Prop. 2.6]{Pau_Gromov}.

We identify two $G$-trees if there is a $G$-equivariant isomorphism between $T$ and $T'$.
Equivalently, this means that $T$ and $T'$ have the same length function \cite{AlperinBass_length,CuMo}.

If $\cale$ is a class of subgroups of $G$  (\textit{e.g.} the class  $\Z$), 
we say that a $G$-tree with edge stabilisers in $\cale$ is an $\cale$-tree.

\subsubsection{Bass-Serre theory} \label{sec;trees}

We first recall some classical material, but we hope that stating this will help to clarify certain issues.
    
We use Serre's notations for graphs \cite{Serre_arbres}: a \emph{graph} $X$ consists of a vertex set $V(X)$, a set of oriented edges $E(X)$
endowed with the fixed-point free involution $e\mapsto \ol e$ reversing the orientation of an edge, 
and $t:E(X)\ra V(X)$ a map assigning to an oriented edge its terminal vertex.
The origin $o(e)$ of an edge $e$ is $t(\ol e)$.

A \emph{graph of groups} consists of a graph $X$, a vertex group $\G_v$  for each $v\in V(X)$,
an edge group $\G_e$ for each $e\in E(X)$ with $\G_e=\G_{\ol e}$, and for each $e\in E(X)$, 
a monomorphism $i_e: \G_e\to \G_{t(e)}$. If $\G$ is a graph of groups, we often, by a slight abuse of notation, denote by $\G$ the underlying graph.

The \emph{Bass group} $B(\Gamma)$ of $\Gamma$ is the free product of the vertex groups and of the free group on $E(\Gamma)$ subject to the relations
\begin{itemize}
\item $\ol e=e^{-1}$ for all $e\in E(\Gamma)$
\item $e i_e(g) e^{-1} = i_{\ol e} (g)$ for all $e\in E(\Gamma)$ and $g\in \G_{e}$.
\end{itemize}
An element of the Bass group is a \emph{path} from $v$ to $w$ if it is of the form
$g_0 e_1 g_1 \dots e_n g_n$ where $e_1\dots e_n$ is a path in the graph $\Gamma$ joining $v$ to $w$
and $g_i\in\Gamma_{t(e_i)}=\Gamma_{o(e_{i+1})}$.

The \emph{fundamental group} $\pi_1(\Gamma,v)$, where $v\in V(\Gamma)$ is a base point,
is by definition the subgroup of $B(\Gamma)$ consisting of elements which are paths from $v$ to $v$.
Given a base point $v\in \Gamma$, there is a \emph{universal covering}
$\Tilde \Gamma$ (which is called the Bass-Serre tree of $(\Gamma,v)$)  
endowed with an action of $\pi_1(\Gamma,v)$.
If $p\in B(\Gamma)$ is a path from $w$ to $v$, the map $\phi_p:\pi_1(\Gamma,v)\ra\pi_1(\Gamma,w)$
defined by $g\mapsto pgp\m$ is an isomorphism which induces a $\phi_p$-equivariant
isomorphism between the universal coverings of $\Gamma$ relative to $v$ and $w$.

If $G$ is a group, a \emph{marking} of $\pi_1(\Gamma,v)$ is an isomorphism  $\mu:\pi_1(\Gamma,v)\ra G$.
This marking defines an action of $G$ on the universal cover $\Tilde \Gamma$ (relative to the base point $v$).
Allowing to change the marking by some isomorphism $\phi_p$,
one can suppress the reference to a base point, and the marking is only defined up to the choice of a base point and composition
by an inner automorphism. Since these operations do not change the action $G\actson \Tilde \Gamma$ up to equivariant isometry,
we make the abuse of saying that a marking is an isomorphism $\mu:\pi_1(\Gamma)\ra G$ (without reference to a base point).

If $\tau\subset E(\Gamma)$ is a maximal subtree, the fundamental group $\pi_1(\Gamma,\tau)$
is the quotient of $B(\Gamma)$ by the edges occurring $\tau$.
The composition $\pi_1(\Gamma,v) \into B(\Gamma) \onto \pi_1(\Gamma,\tau)$ is an isomorphism, see \cite[Chap. 1 Prop. 20]{Serre_arbres}.

\emph{Convention about notations} 
 We choose to write $G$ for a group, and keep the letter $\G$ for  graphs of groups. However, by convention, in a graph of groups $\G$, the symbols $\G_v, \G_e$ will denote the groups of the vertex $v$, of the edge $e$. But, if a group $G$ acts on a tree, the stabiliser of a vertex $v$ will be denoted by $G_v$. 
\\

Given a $G$-tree, the graph $X=T/G$ can be endowed with a structure of a graph of groups $\G$ as follows.
For each $x\in V(X)\cup E(X)$, choose a lift $\Tilde x$ in $T$, so that $\Tilde {\ol e}=\ol{\Tilde e}$. 
Contrary to \cite{Serre_arbres}, we don't assume the connectedness of the lift of $X$, we rather choose arbitrary lifts of edges and vertices.
For all $e\in E$, consider $v=t(e)$, then there exists $h_e\in G$ such that $h_e.t(\Tilde e)=\Tilde v$.
For all $x\in E(X)\cup V(X)$, we define $\Gamma_x=G_{\Tilde x}<G$ as the stabiliser of $\Tilde x$, 
and $i_e$ as the restriction to $\Gamma_{\Tilde e}$ of the inner automorphism $\inn_{h_e}:g\mapsto h_e g h_e\m$.

Given these choices, there is a morphism $\mu:B(\Gamma)\ra G$ defined by its value on its generating set as follows:
$\mu$ is defined on $\Gamma_v$ as the inclusion in $G$, and for each edge $e\in E(\Gamma)$, 
$\mu(e)=h_{\ol e}h_e\m$.
The restriction of $\mu$ to $\pi_1(\Gamma,v)$ is an isomorphism to $G$ 
\cite[Section 3.3]{Bass_covering} (where the proof there is given only in the case the lift of $\Gamma$ is a tree).
Moreover, to the marking $\mu$ corresponds
a $\mu$-equivariant isomorphism $\Tilde\mu:\Tilde \Gamma\ra T$.

The graph of groups $\Gamma$ and the marking depend on choices of the lifts $\tilde x$ of all $x\in V(X)\cup E(X)$, 
of the elements $h_e\in G$, and of the base point $v\in V(X)$. 
But of course, the action of $G$ on $\Tilde \Gamma$  induced by $\mu$ does not depend on choices up to equivariant isometry since
it is isomorphic to $G\actson T$.

Algorithmically, one describes a graph of groups by a presentation of edge and vertex groups,
and the monomorphisms $i_e$ by the images of the generators of $G_e$.
One can compute a presentation of $\pi_1(\Gamma,v)$ using the isomorphism with $\pi_1(\Gamma,\tau)$.
One describes an action of $G$ on a tree up to $G$-equivariant isomorphism by
giving a graph of groups $\Gamma$, and an isomorphism $\pi_1(\Gamma,v)\ra G$.

\subsubsection{Isomorphisms of graph of groups}

All graph of groups are supposed to be minimal and irreducible (\ie their Bass-Serre trees are assumed to be so).

\begin{dfn}\label{dfn_Phi}
An \emph{isomorphism} $\Phi=(F,(\phi_e),(\phi_v),(\gamma_e))$ of graph of groups is 
a graph isomorphism $F:\Gamma\ra \Gamma'$,
a collection of isomorphisms $\phi_e:\Gamma_e\ra \Gamma'_{F(e)}$  and $\phi_v:\Gamma_v\ra \Gamma'_{F(v)}$ for all $e\in E(\Gamma)$ and all 
$v\in V(\Gamma)$,
and a collection elements $\gamma_e\in \Gamma'_{F(t(e))}$ for all $e\in E(\Gamma)$,
such that $\phi_e=\phi_{\ol e}$ and such that the following diagram commutes:
$$\xymatrix{
\Gamma_e\ar[d]^{\phi_e} \ar[rr]^{i_e}  & &\Gamma_v\ar[d]^{\phi_v}\\
\Gamma'_{F(e)}\ar[r]^{i_{F(e)}} &\Gamma'_{F(v)}\ar[r]^{\inn_{\gamma_e}}  &\Gamma'_{F(v)}\\
}$$
where $v=t(e)$.  
\end{dfn}

\begin{rem*}
 Bass' definition of an isomorphism of graphs of groups 
is more general than ours as it also involves elements $(\gamma_v)_{v\in V(\Gamma)}$ in the Bass group \cite{Bass_covering}.  
In terms of Bass' notations, we work only with morphisms of
the form $\delta\Phi$, where $\Phi$ is a morphism in Bass sense, $\delta\Phi$ being defined in
\cite[2.9]{Bass_covering}.
Lemma \ref{lem_tree2Phi} below shows that any tree isomorphism is induced
by some $\Phi$ as in our definition.
\end{rem*}

Such an isomorphism of graphs of groups $\Phi:\Gamma\ra \Gamma'$ induces an isomorphism $\Phi_B$ of the Bass groups 
defined on the generating set of $B(\Gamma)$ as follows: on $\Gamma_v$, $\Phi_B$ is defined as $\phi_v$,
and $\Phi_B(e)=\gamma_{\ol e}\m F(e) \gamma_e$ for $e\in E(\Gamma)$.
The restriction of $\Phi_B$ induces  an isomorphism of fundamental groups 
$\Phi_*:\pi_1(\Gamma,v)\ra\pi_1(\Gamma',F(v))$.
Note however that $\Phi_B$ does not usually factor into a map $\pi_1(\Gamma,\tau)\ra \pi_1(\Gamma',F(\tau))$.

\begin{lem}[{\cite[\S 2.3]{Bass_covering}}]\label{lem_Phi2tree}
An isomorphism of graphs of groups $\Phi$ induces a $\Phi_*$-equivariant 
isomorphism of the universal coverings $\Tilde\Phi:\Tilde \Gamma\ra\Tilde \Gamma'$.

Moreover,  
if all vertex and edge groups are finitely presented, one can algorithmically compute $\Phi_*$ from $\Phi$.
\end{lem}

\begin{proof} 
Given \cite[\S 2.3]{Bass_covering}, we only need to show that $\Phi_*: \pi_1(\Gamma,v)\ra\pi_1(\Gamma',F(v)) $ can be algorithmically computed. 
One can easily compute explicit presentation of $\pi_1(\Gamma,v)$ and of the Bass group $B(\Gamma)$,
and can compute the embedding $\pi_1(\Gamma,v)$ in $B(\Gamma)$ (by giving the image of the generators).
One can also easily compute the image by $\Phi_B$ (hence by $\Phi_*$)  of the generators of $\pi_1(\Gamma,v)$, 
in  $B(\Gamma')$ (in fact we know that their image are in $\pi_1(\Gamma',F(v))$, but they are not expressed in terms of the generators of this group yet).  
Since $B(\Gamma)$ is finitely presented, one can enumerate the different expressions of a given element. 
This allows us to eventually express the images by $\Phi_*$ of the  generators of $\pi_1(\Gamma,v)$ in terms of the generators of  $\pi_1(\Gamma',F(v))$.
\end{proof}

\begin{rem}
  Let $\Gamma$ be a graph of groups, and $\Gamma'$ be the graph of groups 
with the same underlying graph, the same edge and vertex groups as $\Gamma$
but whose edge morphisms are defined by $i'_e=\inn_{\gamma_e}\circ i_e$
for some $\gamma_e\in \Gamma_{t(e)}$.
Then $\Gamma$ and $\Gamma'$ are isomorphic under $\Phi=(\id_\Gamma,\id_{G_e},\id_{G_v},\gamma_e)$.
In other words, changing the conjugacy classes of the edge morphisms of $\Gamma$ does not change the dual tree
up to the change of marking above.
\end{rem}

The lemma below shows that conversely, an isomorphism of trees is always induced by a (not uniquely defined) 
isomorphism of graph of groups.

\begin{lem}[{\cite[Bass, Prop 4.]{Bass_covering}}] \label{lem_tree2Phi} 
Any equivariant isomorphism of trees induces an automorphism of graph of groups. More precisely one has the following.   
 
 Let $G\actson T$ and $G'\actson T'$ be two minimal irreducible actions on trees. 
Let $\Gamma,\Gamma'$ be the corresponding quotient graphs of groups corresponding to some choices,
and $\mu:\pi_1(\Gamma,v)\ra G$,  $\mu':\pi_1(\Gamma',v')\ra G'$, $\Tilde\mu:\Tilde \Gamma\ra T$,
 $\Tilde \mu':\Tilde \Gamma'\ra T'$ the corresponding marking.

If $\phi:G\ra G'$ is an isomorphism, and $f:T\ra T'$ is a $\phi$-equivariant isomorphism,
then, after performing a change of marking induced by a path $p\in B(\Gamma')$, 
there exists an isomorphism $\Phi$ from $\Gamma$ to $\Gamma'$ inducing $\phi$ and $f$ in the sense
that the following diagrams commute (for the new markings):
$$\xymatrix{G \ar[r]^{\phi} & G'                      && T \ar[r]^f & T'              \\
\Tilde\pi_1(\Gamma,v) \ar[r]^{\Phi_*} \ar[u]^{\mu}& \pi_1(\Gamma',F(v)) \ar[u]^{\Tilde\mu'}    &&  
\Tilde\Gamma \ar[r]^{\Tilde \Phi}\ar[u]^{\Tilde\mu} & \Tilde \Gamma' \ar[u]^{\Tilde\mu'}    
}$$
\end{lem}

Applying the Lemma to the case $T=T'$, we see that changing the choices (of lifts and of elements $h_e$) needed
to define the quotient graph of groups does not change the graph of groups, up to isomorphism and change of base point.

\begin{proof}
The given reference \cite[Proposition 4.4]{Bass_covering} says that there exists such a morphism $\Phi$, but in the more general definition of
an automorphism used in \cite{Bass_covering}.  We use Bass' notations. 
We prove that $\delta\Phi$ (which satisfies our definition
of a isomorphism of graph of groups) still satisfies the Lemma.

Let $\Tilde\Phi:\Tilde\Gamma\ra \Tilde\Gamma'$ (resp. $\Tilde{\delta\Phi}$) be the $\Phi_*$-equivariant (resp.\ $\delta\Phi_*$) 
isometry induced by $\Phi$ (resp. $\delta\Phi)$.
Then $\Tilde\Phi\m\Tilde{\delta\Phi}$ is a $\Phi\m_*\delta\Phi_*$-equivariant automorphism of $T$.
By \cite[6.3(6)]{BaJi_automorphism}, $\Phi\m_*\delta\Phi_*$ is an inner automorphism, say $\inn_g$.
Up to changing the marking of $\Gamma'$ using a path in $B(\Gamma')$, we can ensure that $g=1$.
By Lemma \ref{lem_unq} below, for any isomorphism $G\ra G'$ there is at most one $f:T\ra T'$ which is $G$-equivariant. 
It follows that $\Tilde\Phi=\Tilde{\delta\Phi}$, so
$\delta\Phi$ still satisfies the Lemma. 
\end{proof}

\begin{lem}\label{lem_unq}
Let $T$ (resp. $T'$) be a  minimal, irreducible $G$-tree (resp. $G'$-tree). 
Let  $\phi:G\ra G'$ be an isomorphism. Then,  there is at most one $\phi$-equivariant
  isomorphism $f:T\ra T'$. 
\end{lem}

\begin{proof}
For each hyperbolic element $g\in G$, any such $f$ maps the axis of $g$ to the axis of $\phi(g)$.
In particular, if the axes of $g,h$ are disjoint, $f$ maps the bridge between them to the bridge between the axes of $\phi(g),\phi(h)$.
Since $T$ is irreducible, every segment of $T$ is contained in the bridge between two hyperbolic elements \cite{Pau_Gromov}. The lemma follows.
\end{proof}

Two graph of groups automorphisms can be composed \cite[2.11]{Bass_covering} by
$$(F',(\phi'_{e'}),(\phi'_{v'}),(\gamma'_{e'}))\circ (F,(\phi_e),(\phi_v),(\gamma_e))=
(F'\circ F,(\phi'_{F(e)}\circ \phi_e),(\phi'_{F(v)}\circ \phi_v),\phi'_{F(t(e))}(\gamma_e)\gamma'_{F(e)})).$$
By \cite[\S 2.12]{Bass_covering} this composition law makes the set of
of graph of groups automorphisms of $\Gamma$ a group. Following \cite{Bass_covering}, we denote this group by $\delta \Aut(\Gamma)$.

Recall that $\Out(G)$ acts on the set of $G$-trees up to equivariant isomorphism, by precomposition of the action.
We denote by $\Out_T(G)<\Out(G)$ the stabiliser of $T$ for this action.
Lemmas \ref{lem_Phi2tree} and \ref{lem_tree2Phi} say that if $G\actson T$ is dual to $\Gamma$ under some marking,
then one has a natural epimorphism $\delta\Aut(\Gamma)\onto \Out_T(G)$.

Let $\Homeo(\Gamma)$ be the group of graph automorphisms of the graph underlying  $\Gamma$.
By definition of $\delta Aut(\Gamma) $,  one has natural morphisms: 
\begin{itemize}
\item $q_\Gamma: \delta Aut(\Gamma) \to \Homeo(\Gamma)$, defined by $\Phi=(F,(\phi_e),(\phi_v),(\gamma_e))\mapsto F$
\item $q_E:  \ker (q_\Gamma) \to \prod_{e\in E(\Gamma)} \Aut (\G_e)$ 
defined by $\Phi=(\id_\Gamma,(\phi_e),(\phi_v),(\gamma_e))\mapsto (\phi_e)$
\item  $q_V: \ker(q_E) \to \prod_{v\in V(\Gamma)} \Aut (\G_v)$,
defined by $\Phi=(\id_{\Gamma},(\id_{\Gamma_e}),(\phi_v),(\gamma_e))\mapsto (\phi_v)$ 
\item $q_D: \ker (q_V) \to \prod_{e\in E(\Gamma)} {\G_{t(e)}} $  
defined by $\Phi=(\id_\Gamma,(\id_{\Gamma_e}),(\id_{\Gamma_v}),(\gamma_e))\mapsto (\gamma_e)$. 
\end{itemize}
Note that these morphisms need not be surjective.

\begin{lemma}\label{lem_ext}
The morphism  $q_D$ is injective, and its image is  $\prod_{e} Z_{\G_{t(e)}}(\G_e)$.

  The image of $q_V$ is  $\prod_v \Autm(\Gamma_v;\calp_v)$, where $\calp_v$ is a marked peripheral
structure induced by $\Gamma$ on $\Gamma_v$; in other words, $\Autm(\Gamma_v;\calp_v)$ 
is the group of automorphisms of $\Gamma_v$ which coincide with an inner automorphism of $\Gamma_v$
in restriction to each neighbouring edge group.
\end{lemma}

\begin{proof}
 Injectivity of $q_D$ is clear, and the image is easily obtained by the commutation of the diagram of Definition \ref{dfn_Phi}.

  The fact that $\Im q_V\subset \prod_v \Autm(\Gamma_v;\calp_v)$ follows from the commutation of this diagram with
  $\phi_e=\id$.  Conversely, consider $(\phi_v)$ is a collection of automorphisms in $\Autm(\Gamma_v;\calp_v)$, for all edge $e$
  there exists $\gamma_e\in G_{t(e)}$ such that $\phi_v\circ i_e=\inn_{\gamma_e}\circ i_e$.  Then the automorphism $\Phi\in
  \delta\Aut(\Gamma)$ defined by $F=\id$, $\phi_e=\id$, $\phi_v$ and $\gamma_e$ shows that $(\phi_v)\in\Im q_V$.
  $\Phi$ is called an \emph{extension} of $(\phi_v)_{v\in V}$ to $\Gamma$.
\end{proof}

The following lemma is routine, given the chain of quotients above. 
          
\begin{lemma}\label{lem_gen_set_deltaAut}
  $\delta\Aut(\Gamma)$ is generated by:
  \begin{itemize}
  \item elements of the form $\Phi=(Id,\id_{\Gamma_e},\id_{\Gamma_v},(\gamma_e))$ where $(\gamma_e)_{e\in E(\Gamma)}$ 
    ranges over a generating set of $\prod_{e} Z_{\G_{t(e)}}(\G_e) $ 
  \item for each element $(\phi_v)$ of a generating set of $\prod_v \Autm(\G_v;\calp_v)=\Im q_V$, 
    an extension of $(\phi_v)$ to $\Gamma$ (as in the proof of Lemma \ref{lem_ext})
  \item for every element in the image of $q_E$, a preimage by $q_E$  
  \item for every element in the image of $q_\Gamma$, a preimage by $q_\Gamma$.    
  \end{itemize}  
  
If one has a marking of $\Gamma$ by $G$,
the image of this generating set under the epimorphism $\delta\Aut(\Gamma)\onto \Out_T(G)$ 
gives a generating set for $\Out_T(G)$. 
\end{lemma}

\subsubsection{Isomorphism problem for graph of groups}

The \emph{isomorphism problem for graph of groups} asks for an algorithm that, given two  finite graphs of groups $\Gamma_1,\Gamma_2$, 
decides whether there exists an isomorphism of graph of groups $\Phi$ between $\Gamma_1$ and $\Gamma_2$.
If $\pi_1(\Gamma_i)$ is induced with a marked peripheral structure $\calq_i$,
we want additionally that the induced map $\Phi_*:\pi_1(\Gamma_1)\ra\pi_1(\Gamma_2)$ sends $\calq_1$ to $\calq_2$.
The goal of this section is to obtain a solution to the isomorphism problem for certain graph of groups.

\begin{lem}\label{lem_gog_periph}
  Let $\Gamma_1,\Gamma_2$ be two graphs of groups,
and $\calq_i=(S_1^{(i)},\dots,S_{p}^{(i)})$ be a marked peripheral structure of $\pi_1(\Gamma_i)$ 
where $\grp{S_j^{(i)}}$ is contained in a vertex group $\Gamma_{v^{(i)}_j}$ of $\Gamma_i$,
 and not conjugate into any incident edge group.
For each vertex $v\in \Gamma_i$, the sub-tuple of $\calq_i$ consisting of $S_j^{(i)}$'s with $v^{(i)}_j=v$ 
defines a marked peripheral structure $\calq_v$ of $\Gamma_v$.

Let $\Phi=(F,(\phi_e),(\phi_v),(\gamma_e))$ be an isomorphism of graphs of groups.

Then $\Phi_*:\pi_1(\Gamma_1)\ra \pi_1(\Gamma_2)$ is compatible with the peripheral structures $\calq_1,\calq_2$
if and only if 
 for each $v\in V(\Gamma_1)$, $\phi_{v}:\Gamma_v\ra\Gamma_{F(v)}$ is compatible with the peripheral structures $\calq_v,\calq_{F(v)}$,
and
$F$ is compatible with the peripheral structure in the sense that
$F(v_j^{(1)})=v_j^{(2)}$ for all $j\in \{1,\dots,p\}$. 
\end{lem}

\begin{rem}
A priori, the $\pi_1(\Gamma_i)$-conjugacy class of a  subgroup $\grp{S}$ may intersect $\Gamma_v$ in several $\Gamma_v$-conjugacy classes.
The hypothesis that $\grp{S}$ is not conjugate into an incident edge group actually prevents this.
\end{rem}

\begin{proof}
Assume $\phi_v$ maps $\calq_v$ to $\calq_{F(v)}$ for all $v$, and that $F( v_j^{(1)} )=v_j^{(2)}  $ for all $j$. Consider the image of $S_j^{(1)} \subset \G_{v_j^{(1)}} $ by $\Phi_*$. It is in $\G_{F(v_j^{(1)})} = \G_{v_j^{(2)}}$, and conjugated in this group to $S_j^{(2)}$.  Therefore,  $\Phi_*$ maps $S_j^{(1)}$ to a   $\pi_1(\Gamma_2)$-conjugate of $S_j^{(2)}$,
so $\Phi_*$ maps $\calq_1$ to $\calq_2$.

Conversely, consider $S=S_j^{(1)}\subset \Gamma_{v}$ of $\calq_1$, and $S'=S_j^{(2)}\subset \Gamma_{v'}$ the corresponding one in $\calq_2$.
Assuming that $\Phi_*(S)=S'^g$ for some $g\in \pi_1(\Gamma_2,F(v))$, we need to prove that $v'=F(v)$ and that $\phi_{v}(S)$
is conjugate to $S'$ in $\Gamma_{v'}$.

We take $v$ (resp.\  $F(v)$) as a base point for $\Gamma_1$ (resp.\ $\Gamma_2$).
This may change $\Phi_*$ by an inner automorphism, but does not affect the result.
Since $\Phi_*(\grp{S})$ and $\grp{S'}$ are conjugate and fix a unique vertex in the Bass-Serre tree $\Tilde\Gamma_2$,
these vertices have to be in the same orbit so $F(v)=v'$.

Then $\Phi_*(S)=\phi_v(S)\subset\Gamma_{v'}\subset\pi_1(\Gamma_2,v')$.
By assumption, there exists $g\in \pi_1(\Gamma_2,v_2)$ such that $\phi_v(S)=S'^g$. We claim that $g\in \Gamma_{v'}$, which concludes the proof. 
Indeed, $\phi_v(S)$ and $S'$ both fix the the base point of $\Tilde \Gamma_2$, and no other point.
It follows that $g$ fixes the base point of $\Tilde \Gamma_2$, \ie $g\in \Gamma_{v'}$.
\end{proof}

In what follows, $\calv$ is a class of groups $(G,\calp_u)$ with unmarked peripheral structures.
 By abuse of notation, we also say that a group with \emph{marked} peripheral structure $(G,\calp_m)$
lies in $\calv$ if the induced unmarked peripheral structure does.
We assume that one can solve the \emph{marked isomorphism problem} in $\calv$,
\ie that one can decide whether two groups with marked peripheral structures
 $(G;\calp_m)$,  $(G';\calp'_m)$ in $\calv$ are isomorphic.

\begin{prop}\label{prop_IP_gog}
  Assume that $\calV$ is a class of finitely presented groups with unmarked peripheral structures for which
  the marked isomorphism problem is solvable.
          
  Then, given an input consisting of 
  \begin{itemize*}
      \item   $\Gamma_1,\Gamma_2$ two finite graphs of groups, 
        with marked peripheral structures $\calq_1,\calq_2$ for $\pi_1(\Gamma_1)$ and $\pi_1(\Gamma_2)$
such that 
\begin{itemize*}
\item peripheral subgroups are conjugate in a vertex group but not in an edge group,
\item for each vertex $v\in V(\Gamma_i)$, $\Gamma_v$ endowed with its unmarked peripheral structure induced by $\Gamma_i$ and by $\calq_i$ 
(as in Lemma \ref{lem_gog_periph}), lies in $\calv$.
\end{itemize*}
      \item a graph isomorphism $F: \Gamma_1 \to \Gamma_2$,
      \item a collection of isomorphisms $\phi_e: \Gamma_e \to \Gamma_{F(e)}$ for   $e\in E(\Gamma_1)$,
  \end{itemize*}
  one can decide whether there exists an isomorphism of graph of groups $\Phi$ of the form
  $\Phi=(F,(\phi_e),(\phi_v),(\gamma_e))$ (where $F$ and $(\phi_e)_{e\in E(\Gamma_1)}$ are the ones given as input),
  and such that $\Phi_*$ maps $\calq_1$ to $\calq_2$,
  and produces one if it exists.
\end{prop}

\begin{proof}
Assume first that $\Gamma_i$ have no peripheral structure.
  Choose an ordering of the oriented edges of $\Gamma_1$, and
  for each edge $e \in E(\Gamma_1)$, let $S_{e}$ be a tuple generating the group $G_e$.
  Transport this ordering and these markings to $\Gamma_2$ using $F$ and $\phi_e$.
  This way, each vertex group of $\Gamma_1$ and $\Gamma_2$ inherits an induced marked peripheral structure.
  The existence of $\Phi$ is equivalent to the existence of isomorphisms $\phi_v:\G_v \to \G_{F(v)}$
  and elements $\gamma_e$ making the diagram of Definition \ref{dfn_Phi} commute.
  This is equivalent to the fact that $\phi_v$ preserves the induced marked peripheral structures of the vertex groups.
  The proposition follows.

In presence of peripheral structures $\calq_i$, the same argument applies thanks to Lemma \ref{lem_gog_periph}.
\end{proof}

We can now state a solution to the isomorphism problem for certain graphs of groups.

\begin{coro}\label{coro_IP_gog}
  Assume that $\calE$ is either the class of finite groups, or of $\Z$-groups,
  and that $\calV$ is a class of  finitely presented  
groups with unmarked peripheral structures whose peripheral subgroups are in $\cale$, 
  for which the marked isomorphism problem is decidable.

Consider $\calc$ the following class of  graphs of groups $\Gamma$, maybe together with a marked peripheral structure $\calq$ of $\pi_1(\Gamma)$:
  \begin{itemize*}
  \item  edge groups of $\Gamma$ are in $\cale$
  \item each peripheral subgroup of $\calq$ lies in $\cale$, is conjugate into a vertex group, but not into an edge group
  \item each vertex group $\Gamma_v$ of $\Gamma$, endowed with the unmarked peripheral structure induced by $\Gamma$ and by $\calq$, lies in $\calV$. 
  \end{itemize*}

Then the   isomorphism problem for graph of groups in $\calc$ (with peripheral structure)
 is effectively solvable. 
\end{coro}

\begin{proof}
If two groups in $\calE$ are isomorphic, we can, with our assumption, list the whole set of isomorphisms
between them. In particular, one can list all graph isomorphisms $F:\Gamma_1\ra \Gamma_2$ such
that $\Gamma_e\simeq\Gamma_{F(e)}$, and list all possible isomorphisms $\phi_e:\Gamma_e\ra\Gamma_{F(e)}$.
For any such choice, we apply Proposition \ref{prop_IP_gog}, which allows us to decide the existence of a graph of groups isomorphism $\Phi$
between $\Gamma_1$ and $\Gamma_2$ preserving the peripheral structures.
\end{proof}

\begin{coro}\label{coro_gene_Aut_gog}  
Consider $\calE$ and $\calV$, 
and $\calc$   
classes of graph of groups $\Gamma$ (maybe with a marked peripheral structure $\calq$) defined in Corollary \ref{coro_IP_gog}.

Assume moreover that there is an effective algorithm that computes a generating set
of $\Autm(H;\calp_m)$ for all group with marked peripheral structure $(H;\calp_m)\in \calv$,
and an algorithm computing a generating set of the centraliser $Z_H(P)$ of a peripheral subgroup $P$ of $(H;\calp)\in \calv$.

Then given $\Gamma$ (maybe with a marked peripheral structure $\calq$)
in $\calc$,
one can compute a system of generators for $\Out_{\Tilde\Gamma}(\pi_1(\Gamma))$
(or of $\Out_{m,\Tilde\Gamma}(\pi_1(\Gamma),\calq)$
of outer automorphisms of $\pi_1(\Gamma)$
preserving both $\Tilde \Gamma$ and $\calq$).

\end{coro}

\begin{proof}
We first treat the absolute case (where no $\calq$ is given).
 We will compute the four sets mentioned in Lemma \ref{lem_gen_set_deltaAut}.

Let $q_V:\ker(q_E)\subset \delta Aut(\Gamma) \to \prod_{v\in V(\Gamma)} \Aut (\G_v)$, be the map defined before Lemma \ref{lem_ext}. 
Denote by $\calp_v$ a marking of the peripheral structure of a vertex group $\Gamma_v$ defined by incident edge groups.
By  Lemma \ref{lem_ext}, 
the image of $q_V$ is precisely $\prod_v \Autm(\G_v;\calp_v)$.

First, a generating set for
  $\prod_{e} Z_{\G_{t(e)}} (\G_e)$ is computable by assumption.
  Second, by assumption, we can compute a generating set for $\Im q_V = \prod_v \Autm(\G_v;\calp_v)$.  
 We can compute a preimage under $q_V$ of each element in
  this generating set by finding correct values for $\gamma_{e}$ for every
  oriented edge so that the diagram of Definition \ref{dfn_Phi} commutes.  This can be done by enumeration. 
  Third, one can list all  maps $F\in\Homeo(\Gamma)$,  and all families of isomorphisms  $(\phi_e)\in\prod_e \Isom(\G_e, \G_{F(e)})$.
  For every such $F$ and $(\phi_e)$, we need to check whether it induces an automorphism of $\Gamma$.
  This is exactly done by Proposition \ref{prop_IP_gog}.

  According to Lemma \ref{lem_gen_set_deltaAut}, 
  we have computed a set of generators for $\delta Aut (\Gamma)$,
  and hence of $\Out_{\Tilde\Gamma}(\pi_1(\Gamma))$.

Assume now that we are given a peripheral structure $\calq$.  
Instead of  $\partial \Aut(\Gamma)$, we have to work with the subgroup $\partial \Aut(\Gamma,\calq)$ 
consisting of automorphisms $\Phi=(F,(\phi_e),(\phi_v),(\gamma_e))$
such that $\Phi_*:\pi_1(\Gamma)\ra\pi_1(\Gamma)$ preserve $\calq$.
Denote by $\calq_v$ the marked peripheral structure induced by $\calq$ on $\Gamma_v$ (as in Lemma \ref{lem_gog_periph}).
By Lemma \ref{lem_gog_periph}, 
$\Phi_*$  is compatible  with $\calq$
  if and only if $F$ fixes all vertices containing a peripheral subgroup in $\calq$, 
and each $\phi_v$ is compatible with $\calq_v$.
The first step is the same since $\ker q_V$ automatically preserves $\calq$, so the image of $q_D$ restricted to $\partial \Aut(\Gamma,\calq)$
is still  $\prod_{e} Z_{\G_{t(e)}} (\G_e)$.
For step 2, Lemma \ref{lem_gog_periph} says that $q_V(\ker q_E\cap \partial \Aut(\Gamma,\calq)) = \prod_v \Autm(\G_v;\calp_v\cup\calq_v)$,
and we can compute a generating set by hypothesis.
Finally, one can list all maps  $F\in\Homeo(\Gamma)$ fixing vertices involved in $\calq$, 
and all families of isomorphisms $(\phi_e)\in\prod_e \Isom(\G_e, \G_{F(e)})$.
Proposition \ref{prop_IP_gog} allows us to check whether 
such $F$ and $(\phi_e)$, extend to an automorphism of $\Gamma$ preserving $\calq$.
This concludes the proof.
\end{proof}

\subsection{Deformation spaces of trees}  
\label{sec_deformation}

    Given two $G$-trees $T_1,T_2$  we say that $T_1$ \emph{dominates} $T_2$ if  there is an equivariant 
    (but otherwise arbitrary) map $T_1\to T_2$. 
    Equivalently, $T_1$ dominates $T_2$ if every vertex stabiliser of $T_1$ is elliptic in $T_2$.
    We say that two $G$-trees $T_1, T_2$ are in the same \emph{deformation space} if they dominate each other, or
    equivalently, if they have the same elliptic subgroups.
    This is clearly an equivalence relation. The group of outer automorphisms $\Out(G)$ acts on the set of $G$-trees by precomposition 
    (recall that we identify equivariantly isomorphic $G$-trees). This induces an action on each invariant deformation space (see \cite{GL2} for more details).

    One says that a $G$-tree $T$ is \emph{reduced} if no orbit of edges of $T$ can be collapsed so that $T$
    and the collapsed tree lie in the same deformation space.  
Equivalently, $T$ is reduced if for all oriented edge $e$ of the graph of groups $T/G$ 
        such that the edge morphism $i_e$ is onto, $e$ is a loop (\ie $o(e)=t(e)$).

    In all $G$-trees we will consider, $G$ will be hyperbolic, and edge stabilisers will be finite or virtually cyclic.
    In particular, no edge stabiliser is properly contained in a conjugate of itself. 
    It follows that all deformation spaces $\cald$ that we will consider are \emph{non-ascending}, a technical condition
    asking that no graph of groups of a tree in $\cald$ has an edge $e$ with both endpoints at $v$, and such that
    $i_e(G_e)=G_v$ and $i_{\ol e}(G_e)\subsetneq G_v$ \cite[Prop. 7.1]{GL2}.

\paragraph{The Stallings-Dunwoody deformation space.}   
The \emph{Stallings-Dunwoody deformation space} of a group $G$ 
is the set of $G$-trees with finite edge stabilisers and finite or one-ended vertex groups. 
The fact that this is indeed a deformation space follows from the fact that vertex groups of
such a $G$-tree do not split over finite groups by Stallings theorem. Existence of such $G$-trees, for a finitely presented group $G$, 
is Dunwoody's original accessibility.

A relative version also exists for $(G,\calp)$ a group with peripheral structure, where we consider only $G$-trees
in which peripheral subgroups are elliptic.

    \paragraph{Slide moves.} Consider a $G$-tree $T$, and 
    a pair of adjacent edges $e_1=[a,b],e_2=[b,c]$ not in the same orbit, such that $G_{e_1}\subset G_{e_2}$.
    Let $T'$ be the $G$-tree having the same vertex set as $T$, and whose
edges are obtained from the edges of $T$ by replacing each edge
$g.e_1=[g.a,g.b]$ by an edge joining $g.a$ to $g.c$.
This definition does not depend on the choice of $g$ because $G_{e_1}\subset G_{e_2}$.
We say that $T'$ is obtained by  \emph{sliding $e_1$ across $e_2$}.

By \cite[Th. 7.2]{GL2}, in a non-ascending deformation space,
performing a slide move on a reduced $G$-tree yields
a reduced $G$-tree, and all reduced trees of $\cald$ are connected by slide moves.

Let us describe a slide move at the level of  the graph of groups. To avoid confusion,  edges in graphs of groups are written between brackets.
At the level of the graph of groups $T/\G$, a slide move is determined by two different oriented edges $[e_1]\neq [e_2]$ with same final  vertex $v$,
and an element $g\in G_v$ such that $i_{[e_1]}(G_{[e_1]})\subset i_{[e_2]}(G_{[e_2]})^g$.
After the slide, the endpoint of $[e_1]$ is attached to the origin of $[e_2]$, and the new monomorphism is
$ i_{[e_1']} = i_{ [\ol e_2]}\circ i_{[e_2]}\m \circ \inn_g\circ   i_{[e_1]}$.

\begin{lemma}\label{lem;slide_autom}
Let $T$ be a $G$-tree, and consider  a pair of adjacent edges $e_1=[a,b],e_2=[b,c]$ not in the same orbit (as non-oriented edges), 
such that $G_{ e_1}\subset G_{e_2}$. Let $e_3 =ge_2$ for some element $g$ in the centraliser of $G_{e_2}$ in $G_v$.  
Let $T'$ (resp. $T''$) be the $G$-tree obtained by sliding $e_1$ across  $e_2$ (resp. across $e_3$). 
Then, $T'$ and $T''$ are in the same orbit under $\Out(G)$. 
\end{lemma}

\begin{proof}
This is best seen at the level of graph of groups.  Let $[e'_1]$ and $[e''_1]$ be the slid edge obtained from $[e_1]$ in   the graphs of groups $T'/G$ and $T''/G$. These graphs of groups 
 differ only in that the morphism $i_{[e_1']} $ is, in general, not equal to $i_{[e''_1]}$. 
Indeed,  $i_{[e_1']} = i_{[\ol e_2]}\circ i_{[e_2]}\m \circ \inn_{g_1}\circ   i_{[e_1]}$ for some $g_1$, 
and   $i_{[e_1'']} = i_{[\ol e_2]}\circ i_{[e_2]}\m \circ \inn_{gg_1}\circ   i_{[e_1]}$. 
This means that $i_{[e_1'']} = \inn_{i_{[\ol e_2]}(g)} \circ  i_{[e_1']}$, and there is an isomorphism $\Phi$ between the graph of groups 
(Definition \ref{dfn_Phi}). By {\cite[\S 2.3]{Bass_covering}} (recalled in Lemma \ref{lem_Phi2tree} above), there is a $\Phi_*$-equivariant isomorphism
$\Tilde\Phi$ between $T'$ and $T''$.
\end{proof}

\subsection{Tree of cylinders}

 Consider $G$ a hyperbolic group and recall that a $\Z$-tree is a $G$-tree whose edge stabilisers are in the class $\calz$ 
of virtually cyclic subgroups with infinite centre.
Let $\sim$ be the commensurability relation on $\Z$:
given $A,B\in\Z$, $A\sim B$ if $A\cap B$ has finite index in $A$ and $B$.

Given $T$ a $\Z$-tree, $\sim$ induces an equivalence relation on edges of $T$ defined by $e\sim e'$ if $G_e\sim G_{e'}$.
A \emph{cylinder} $Y\subset T$ is the union of the edges in an equivalence class. Each cylinder is a subtree of $T$ \cite{GL4}.
The \emph{tree of cylinders} $T_c$ of $T$ is the tree dual to the covering of $T$ by its cylinders in the following sense.
Let $V_1(T_c)$ be the set of cylinders of $T$, and $V_0(T_c)$ be the set of vertices $v\in T$ which lie in at least $2$ cylinders.
The tree $T_c$ is the bipartite $G$-tree whose vertex set is $V(T_c)=V_0(T_c)\dunion V_1(T_c)$,
and where $x\in V_0(T_c)$ and $Y\in V_1(T_c)$ are connected by an edge $(x,Y)$ if and only if $x\in Y$ (see \cite{GL4} for details).

The stabiliser $G_Y$ of a cylinder $Y$ is the commensurator of some edge group $G_e\in\Z$.
Since $G$ is hyperbolic,  $G_Y=VC(G_e)$ is the maximal virtually cyclic group containing $G_e$ ($G_Y$ might have finite centre).
Also note for future use that since $G_e$ has finite index in $G_Y$, $G_Y$ is elliptic in $T$.
The stabiliser of an edge $(x,Y)$ of $T_c$ is virtually cyclic: it is infinite by \cite[Rem 4.4]{GL4},
and virtually cyclic because contained in $G_Y$.
Because edge stabilisers of $T_c$ may fail to be in $\Z$ (they may have finite centre), one defines
\emph{collapsed tree of cylinders} $T_c^*$ obtained from $T_c$ by collapsing all edges whose stabiliser are not in $\Z$.

\begin{lem}\label{lem_same_D}
  Let $G$ be hyperbolic, and $T$ be a $\Z$-tree.

Then $T$, $T_c$ and $T_c^*$ lie in the same deformation space. 
 Every edge stabiliser of $T_c$ and $T_c^*$ contains an edge stabiliser of $T$ with finite index.
If edge stabilisers of $T$ are $\Zmax$, the so are edge stabilisers of $T_c^*$.
\end{lem}

\begin{proof}
The fact that $T$ dominates $T_c$ which dominates $T_c^*$ is clear and general.
We saw that the stabiliser $G_Y$ of each cylinder is elliptic in $T$.
Thus, every vertex stabiliser of $T_c$ is elliptic in $T$, so $T$ and $T_c$ lie in the same deformation space \cite[Prop. 5.2]{GL4}.
By \cite[Cor. 5.10]{GL4}, $T_c^*$ lies in the same deformation space as $T$.

 By \cite[Rem 4.4]{GL4}, the stabiliser of any edge $\eps=(x,Y)$ of $T_c$ contains an edge stabiliser $G_e$ of $T$.
Since $G_Y$ is virtually cyclic, $G_e\subset G_\eps\subset G_Y$ are commensurable to each other.
The last two assertion follow.
\end{proof}

Recall that we identify two $G$-trees when there is a $G$-equivariant isomorphism between them,
which defines equality in the following statement:

\begin{prop}[{\cite[Cor.\ 4.10]{GL4}}]\label{prop;Tc_inv}
  Let $T,T'$ be two $\Z$-trees in the same deformation space.

Then $T_{c}=T'_{c}$ and $T_{c}^*=T_{c}'^*$.
\end{prop}

On the algorithmic side, we have:
\begin{lem}\label{lem_calcul_Tc} 
Let $G$ be a hyperbolic group.
  Given a $\Z$-tree $T$ (represented by a graph of groups decomposition of $G$),
  one can compute the graph of group decompositions corresponding to  $T_c$ and $T_c^*$.
\end{lem}

\begin{proof} Since it is possible to decide whether a given subgroup of $G$ is in $\Z$, it is enough to compute
  effectively a graph of group decompositions corresponding to $T_c$.

  With the notations above, the set of vertices of $T_c/G$ consists of two subsets: $V_0(T_c)/G$, which is the set of vertices of
  $T/G$ with non-elementary vertex groups, and $V_1(T_c)/G$, which is the set of orbits of cylinders in $T$.  Note that this set
  of orbits is in natural correspondence with the equivalence classes of edge groups of the graph of groups $T/G$, under the
  relation ($G_e\sim_G G_{e'}$ if $\exists g\in G, \, G_e \sim G_{e'}^g $) where $\sim$ is the commensurability relation as
  above. 
Since $G$ is hyperbolic, $G_e \sim G_{e'}^g $ if and only if $\Zmax(G_e)$ (the maximal $\Z$-group containing $G_e$)
  is conjugate to $\Zmax(G_{e'})$. One can list all automorphisms $\Zmax(G_e)\ra\Zmax(G_{e'})$ by Lemma \ref{lem_autZ},
  and use the simultaneous conjugacy problem to decide whether these groups are conjugate.
 The stabiliser of the $V_1$ vertex corresponding to $G_e$ is $G_Y=VC(G_e)$.
Therefore, we can compute the set of vertices of $T_c/G$, and their vertex groups.

  Let us compute the set of edges incident to a certain vertex $v_0$ whose group is non-elementary. 
  Fix some representative $\Tilde v_0\in T$.
  Edges of $T_c$ incident on $\Tilde v_0$ correspond to cylinders containing $\Tilde v_0$
  so edges of $T_c/G$ incident to  $v_0$ are in correspondence 
  with equivalence classes of the edges of $T$ incident to $\Tilde v_0$ for the relation
  ($G_e\sim_{G_{\Tilde v_0}} G_{e'}$ if $\exists g\in G_{\Tilde v_0}, \, G_e \sim G_{e'}^g $) where $\sim$ is the
  commensurability relation. As above, knowing conjugacy classes of groups of edges incident to ${\Tilde v_0}$, we can compute
  these equivalence classes  (here, we are using the fact that $G_{v_0}$ is itself a hyperbolic group, see for instance \cite{Bo_cut}). 
  The edge groups are the commensurators in $G_{v_0}$, and the edge monomorphism
  is given by inclusion.
  The vertex in $V_1(T_c)$ of $T_c/G$ on which the edge is attached (recall that $T_c/G$ is bipartite), 
  is the conjugacy class of its commensurator in $G$,
  and the edge monomorphism is inclusion, followed by some conjugation to arrive in the right conjugate.
  This allows us to compute the graph of groups $T_c/G$.
\end{proof}

\section{Isomorphism problem for rigid hyperbolic groups}\label{sec_rigid} 

In this section we will formulate and prove a rigidity criterion extending the one that was mentioned the introduction (Proposition \ref{alt}). 
We will also give an algorithm that determines whether a given hyperbolic group (with marked peripheral structure) satisfies or not the criterion,
and we will give a solution of the isomorphism problem for groups satisfying the rigidity criterion.

\subsection{A rigidity criterion}

Recall that a homomorphism $f:(G,\calp)\ra (H,\calq)$ where $\calp=(S_1,\dots,S_n)$ and $\calq=(S'_1,\dots,S'_n)$ 
is a homomorphism $f:G\ra H$ sending each peripheral tuple $S_i$ to a conjugate of $S'_i$.
We say that two such homomorphisms $f,g$ are \emph{post-conjugate} if there exists $h\in H$ such that
$g=\inn_h\circ f$.
We say that $G$ is one-ended relative to $\calp$ (or that $(G,\calp)$ is one-ended) if
$G$ has no non-trivial splitting over a finite group relative to $\calp$ (\ie in which peripheral subgroups are elliptic).

\begin{prop}\label{prop_alt}
  Let $(G;\calp)$ be a  hyperbolic group with a marked peripheral structure. Assume that $G$ is one-ended relative to $\calp$.  
  Then the following are equivalent:
  \begin{enumerate}
\renewcommand{\theenumi}{\roman{enumi}}  
\renewcommand{\labelenumi}{(\theenumi)}  
  \item $(G;\calp)$ splits non-trivially over a maximal $\Z$-subgroup  (relative to $\calp$)
 \item $\Outm(G;\calp)$ is infinite 
 \item for all $R>0$, there are infinitely many  post-conjugacy
classes of endomorphisms of $(G;\calp)$ injective on a ball of radius $R$
 \item there is $(H,\calq)$ a one-ended hyperbolic group with a marked peripheral structure, such that for all $R>0$,  
there are infinitely many  post-conjugacy
classes of morphisms  $(G;\calp) \to (H;\calq)$ (respecting the marked peripheral structures) that are injective on the ball of radius $R$ of $G$ 

\item[$(v)$] $G$ admits a non-trivial 
 action on an $\R$-tree $T$, 
  with finite or $\Zmax$ (pointwise) arc
           stabilisers,  with finite tripod stabilisers, and such that every peripheral subgroup of $\calp$ fixes a point in $T$. 
  \end{enumerate}
\end{prop}

We see this result as providing an alternative: either $(G,\calp)$ is rigid (in the sense of the negation of $(ii)$ or better, but more technical, of the negation of $(iv)$), or it admits a splitting over a $\Zmax$-subgroup. The fifth point of the equivalence will be auxiliary. Let us remark that  the peripheral subgroups in the statement can be arbitrary finitely generated subgroups. 

\begin{proof}
Clearly, $(ii) \implies (iii) \implies (iv)$. 
Let us prove that $(i)\implies (ii)$.

 Assume that we have a non-trivial decomposition of $G$ as an amalgam $G=A*_C B$ 
or an HNN extension $G=A*_C$  over a $\Zmax$-subgroup $C$. 
       The splitting being non-trivial, in the case of an amalgamation, 
       $A$ and $B$ strictly contain $C$, hence have finite centre. 
       In the case of an HNN extension,  $C$ is a proper subgroup of $A$ (under both embeddings actually) since $G$ is hyperbolic,
       so $A$ has finite centre.
   It follows from \cite[Prop. 3.1]{Lev_automorphisms} that  the group of Dehn twists on the edge of the    splitting is infinite. This proves $(ii)$.

It remains to prove the most important part, namely $(iv) \implies (v) \implies (i)$. We begin by the first implication, a careful adaptation of Bestvina and Paulin's argument.

Fix generating sets $S$, $S'$ of $G$ and $H$ respectively, and consider the corresponding Cayley graphs $\Cay G$, $\Cay H$. 
We measure lengths in $G$ and $H$ with respect to the word metrics.
        Consider a sequence of 
        morphisms $\phi_n: (G;\calp) \to (H;\calq)$,  in distinct   
post-conjugacy classes, 
        such that  $\phi_n$ is injective 
        on the $n$-ball of $G$. Let us choose them in minimal position for a given set of 
        generators $S$ of $G$: each of them realises 
        the minimum of $m_n= \max\{d(1, \phi_n(s)), s\in S\}$ over their 
        post-conjugacy class. 
        Since the morphisms are not post-conjugate to each other, $m_n $
        goes to  infinity. We can apply Bestvina-Paulin's
        argument.

        Let us consider the   renormalised space $(\Cay H, d_n)$ where $d_n$ is 
        the word metric divided by $m_n$. It is endowed with an action of $G$ induced by $\phi_n$. 
        This sequence of actions converges to an isometric action of $G$ on $\R$-tree $T$
        in the equivariant Gromov-Hausdorff topology.
        Up to replacing $T$ by a subtree, we can assume that the action of $G$ is minimal.
         This action satisfies the following properties:
        \begin{enumerate}
                \item $G$ has no global fixed point, and every peripheral subgroup in  $\calp$ fixes a point in $T$

                \item tripod stabilisers are finite
                \item stabilisers of non-degenerate segments of $T$ are finite or virtually cyclic

        \end{enumerate}
        
        These properties are rather classical, and we refer the reader to
        standard adaptations of arguments of Bestvina and Paulin, 
        \cite{Be_degenerations,Pau_topologie}.  
        The fact that each peripheral subgroup $P$ of $\calp$ is elliptic in $T$ follows from
        the fact that the translation length in $\Cay H$ of any element of $P$ for the 
        action given by $\phi_n$ 
        is independent of $n$, and therefore tends to zero in the rescaled metric.
        There remains to prove that arc stabilisers are finite or $\Zmax$.

          Let $C$ be the stabiliser of some arc $[x,y]\subset T$.  
          Assuming $C$ is infinite, we already know that $C$ is virtually cyclic, and
          let's prove that its centre is infinite.  
          Using the classification of virtually cyclic groups, $C$ is either
          cyclic-by-finite, or dihedral-by-finite. 

          We proceed by contradiction, and we assume that it is an extension $1\ra F\ra C\ra
          D_\infty\ra 1$.  We claim that there exists $\sigma,\tau\in C$ 
          with $\tau$ of infinite order such that $(\tau)^\sigma =
          \tau^{-1}$.  There are such elements in $D_\infty$, 
          so consider $\sigma,\tau \in C$ such that $(\tau)^\sigma =
          \tau^{-1}f$ for some $f\in F$.  Replacing $\tau$ by a power, we can assume 
          that $\tau$ centralises $F$. For each $k$,
          write $(\tau^k)^\sigma = \tau^{-k}f_k$ with $f_k\in F$. Since $F$ is finite, 
          there is $k\neq j$ with $f_k=f_j$, which
          implies $(\tau^{k-j})^\sigma=\tau^{j-k}$.

          Let us also consider $\gamma=\tau^k$, where $k$ is greater than 
          $10\delta/t_0$, and $t_0>0$ is the smallest stable
          translation length of all hyperbolic elements of $H$ (which is positive, 
          see \cite[Prop. 3.1(iii)]{Delzant_sous-groupes}).

          Consider $D=d(x,y)$, and $\epsilon < D/200$.  Let us then choose $N$ a sufficiently large integer such that in $(\Cay H,
          d_N)$, $x$ and $y$ are approximated by two points $x',y'$ (with $d_N(x',y')\geq D-\eps$), such that, for all $\alpha\in
          \{\phi_N(\gamma), \phi_N(\sigma)\}$, $d_N( \alpha x',x') <\epsilon$ and $d_N( \alpha y',y') <\epsilon$.  Moreover one
          can choose $N$ such that the hyperbolicity constant $\delta_N=\delta/m_N$ of $(\Cay H, d_N)$ is at most $\epsilon$, and
          also such that $\phi_N$ is injective enough so that $\phi_N(\tau)$ has not finite order (recall that there is an
          explicit bound on the order of torsion elements in $H$).  Note that $N$ is now fixed, so $\phi_N$ is a fixed morphism,
          and $d_N$ a fixed metric.  We use the notations $\sigma'=\phi_N(\sigma)$ and $\gamma'=\phi_N(\gamma)$.

          By definition of $\gamma$, $\gamma'$ has stable translation length at least $1000\delta_N$ (we measure all lengths in
          the metric $d_N$).  Since $\gamma'$ does not move $x'$ and $y'$ by more than $\epsilon$, its stable norm is at most
          $\eps$ and its translation axis must pass within $\epsilon$ from both $x'$ and $y'$.  It follows that there exists an
          integer $n$ such that $d_N(\gamma'^n x', y') \leq 3\epsilon$.

          We claim that $\gamma'^n$ almost sends $y'$ to $x'$.  
          Indeed, 
          $$\begin{array}{ccc} 
            d_N(\gamma'^{n} y',x') & =& d_N(y',\gamma'^{-n}    x') \\ 
            & = & d_N(y',\sigma' \gamma'^{n}\sigma'^{-1} x') \\ 
            & = & d_N(\sigma'^{-1} y',\gamma'^n\sigma'^{-1} x'). \end{array}$$  
          Since $\sigma'^{-1}$
          moves $x'$ and $y'$ by at most $\eps$, we get the claim, namely:
          $$   d_N(\gamma'^{n} y',x') \,\leq \,  d_N(y',\gamma'^n x')+ 2\eps 
          \,\leq \, 5\eps.$$  
          It
          follows that $d(x',\gamma'^{2n}x')\leq 8\eps$.

          Now consider $x''$ a point such that $d(x',x'')\leq \eps$, and $d(x'',\gamma' x'')=\min_{z\in H} d(z,\gamma'z)$ (existnece is guaranteed by the fact that $x'$ is at distance at most $\epsilon$ from the axis of $\gamma'$).  By
          \cite[Prop. 3.1(ii)]{Delzant_sous-groupes}, 
          $$ n||\gamma'||\geq d(x'',\gamma'^n x'')-120\delta_N\geq D-3\eps-120\eps.$$  
          It
          follows that 
$$d(x',\gamma'^{2n}x')\geq 2n||\gamma'||\geq 2(D-123\eps)$$ which is greater than $8\epsilon$,  a contradiction. This proves that $C$ has
          infinite centre.

Let's prove that if the stabiliser $C$ of an arc $[x,y]$ is infinite, then it is a $\Zmax$-subgroup. 
Let $\Hat C\supset C$ be the $\Zmax$-subgroup of $G$ containing $C$.
We argue by contradiction, and we can assume for instance that 
that for some $\sigma\in\Hat C$, $x\neq \sigma x$.
 Since $T$ has finite tripod stabilisers and $C\cap C^\sigma$ is infinite, the convex hull $I$ of $[x,y]\cup[\sigma x,\sigma y]$ contains no tripod.
Since $\sigma$ is elliptic in $T$, it fixes a point $z\in I$. 
Changing $C$ to the pointwise stabiliser of $I$ (a finite index subgroup),
we can assume that $C$ fixes $[x,z]$, and $\sigma\in \Hat C$ fixes $z$ but not $x$.

Consider an element $\gamma$ of infinite order in the centre of $\Hat C$.
Up to changing $\gamma$ by a power, we can assume that $\gamma\in C$, and that $\phi_N(\gamma)$ has stable norm
at least $1000\delta_N$ as before. 
Since $\gamma\in C$, it fixes $x,z$.
Define $D_2=d(x,\sigma x)$,
and $\eps< D_2/100$.

We take $N$ large enough so that 
in $(\Cay H, d_N)$, $x$ and $z$ are approximated by two
points $x',z'$ (with $d_N(x',z')\geq d(x,z)-\eps$), such that, 
$\phi_N(\gamma)$ has infinite order,
$\phi_N(\gamma)$ moves $x'$ and $z'$ by at most $\eps$ (for the metric $d_N$),
$\phi_N(\sigma)$ moves $z'$ by at most $\eps$,
 and $d_N(\phi_N(\sigma)x',x')\geq D_2-\eps$.
We fix such an $N$, and we define $\gamma'=\phi_N(\gamma)$, $\sigma'=\phi_N(\sigma)$ as above.
As before, the axis of $\gamma'$ passes close to $x'$ and $z'$, and there exists $n$ such that
$d_N(\gamma'^n z',x')<3\eps$.
Then,   since $\sigma$ commutes with $\gamma$.
$$d_N(\sigma' x',x')\, \leq\,  d_N(\sigma' \gamma'^n z',x')+3\eps \, =\, d_N(\gamma'^n\sigma' z',x')+3\eps. $$ 
Since $\sigma'$ almost fixes $z'$, we get 
$$d_N(\sigma' x',x')\, \leq \,d_N(\gamma'^n z',x')+4\eps \, \leq \, 7\eps.$$
This contradicts $d_N(\sigma'x',x')\geq D_2-\eps$.
Thus $C$ is a $\Zmax$-subgroup.

This shows 
$(iv)\implies (v)$.  It remains to show that $(v)\implies (i)$.

  We apply a result of Rips Theory obtained by the second author (Main Theorem of \cite{Gui_actions}).
Recall that an arc $I$ is \emph{unstable} if some subarc $J\subset I$ has strictly larger stabiliser than $I$.
Since arc stabilisers of $T$ are either finite or $\Zmax$, unstable arc stabilisers are finite.
Recall that a \emph{graph of action}  on $\bbR$-trees is an $\bbR$-tree obtained by gluing equivariantly a family $Y_v$ of
$\bbR$-trees (called vertex trees) along points.
The combinatorics of the gluing is given by a simplicial tree $S$ endowed with an action of $G$:
to each vertex $v\in S$ corresponds an $\bbR$-tree $Y_v$ endowed with an action of $G_v$,
and to each edge $e=[u,v]$ of $S$ corresponds a pair of points in $Y_u,Y_v$ to be glued together.
By \cite[Theorem 5.1]{Gui_actions}, either $G$ splits over a finite group relative to $\calp$, or
$T$ has a decomposition into a graph of actions where each vertex action is either
\begin{itemize}
\item simplicial: $Y_v$ is a simplicial tree;
\item of Seifert type: $G_v\actson Y_v$ has finite kernel $N_v$, $G_v/N_v$ is the fundamental
  group of a conical $2$-orbifold $\Sigma_v$ with boundary, 
  and $\pi_1(\Sigma_v)\actson Y_v$ is dual to an arational measured foliation on this orbifold;
\item axial: $Y_v$ is a line, and $G_v$, its stabiliser,
  is finitely generated with dense orbits in $Y_v$ .
\end{itemize}

The finiteness of the kernel of Seifert type vertex actions follows from the finiteness of tripod stabilisers.
If some $Y_v$ is a non-degenerate simplicial tree, then collapsing all the vertex trees not in the orbit of $Y_v$
provides a simplicial tree $T_0$ with an action of $G$. Since arc stabilisers of $T$ are finite or $\Zmax$,
this gives a splitting of $G$ over a finite or a $\Zmax$-subgroup.
There is no axial components since $G_v$ would have some  quotient isomorphic to   a subgroup of $\Isom(\bbR)$ containing $\bbZ^2$
 with  finite or  $\Zmax$  kernel. 
This is impossible since such a $G_v$ would not contain a free group, so would be virtually cyclic,
and could not map onto such a group.

Finally, assume that some $Y_v$ is of Seifert type, hence a hanging orbifold vertex of $S$.
Since $\Sigma_v$ admits an arational measured foliation, it contains a two-sided simple closed curve, not parallel to the boundary in $\Sigma_v$,
and not bounding a M\"obius strip or a disk with at most one cone point.
Then this curve defines a splitting of $\pi_1(\Sigma_v)$ over  
a group $C\simeq\bbZ$ which is $\Zmax$ in $\pi_1(\Sigma_v)$.  
Let $\Hat C\subset G_v$ be the preimage of $C$ so that $G_v$ splits
over $\Hat C$. Refining $S$ using this splitting of $G_v$ defines a splitting of $G$ over $\Hat C$.
Since $\Hat C$ fixes no edge of $S$ and is a $\Zmax$-subgroup of $G_v$, it is a $\Zmax$-subgroup of $G$.

\end{proof}

\subsection{Satisfiability of the criterion, and isomorphism problem}\label{subsubsec_resolving}

In the rest of the section, our main goal is the following result, that allows us to decide  whether 
 a given group satisfies the rigidity criterion, and to solve the isomorphism problem for such groups.

\begin{thm}\label{thm_IP_rigid}
There exists an algorithm which takes as input 
 two  non-elementary hyperbolic groups with marked peripheral structures $(G;\calp)$, $(H;\calq)$,
which always stops, and outputs  either
\begin{enumerate*}\renewcommand{\theenumi}{\alph{enumi}}  
\renewcommand{\labelenumi}{(\theenumi)}  
  \item \label{it_split} a non-trivial splitting of $(G;\calp)$ over a finite or $\Zmax$-subgroup,
  \item \label{it_list} or  a finite list $\calf$ of morphisms $(G;\calp)\ra (H;\calq)$, such that
any monomorphism $(G;\calp)\ra (H;\calq)$ is  post-conjugate to an element in $\calf$.
\end{enumerate*}
\end{thm}

\begin{rem}
  \begin{itemize}\item
    The splittings and morphisms in the statement are compatible with
    the peripheral structure.
    There is no assumption on the peripheral subgroups defined by
    $\calp$ and $\calq$ (they are not assumed to be elementary, or
    quasiconvex). 
  \item  The assumption that the groups are not elementary is superfluous, as one could easily extend the algorithm in this case.
  \end{itemize}
\end{rem}

Given $(G;\calp)$ and $(H;\calq)$, there may exist both a splitting as in (\ref{it_split}) and a finite list as in (\ref{it_list}).
In this case, we cannot guess which output the algorithm will provide. 
However, if $(G;\calp)=(H;\calq)$, and if $(G;\calp)$ is one-ended, the equivalence $(i)\Leftrightarrow (iii)$ in the rigidity criterion \ref{prop_alt}
ensures that only one output is possible. 
 Therefore, one can decide whether a one-ended $(G;\calp)$ satisfies the rigidity criterion  by running the algorithm of Theorem \ref{thm_IP_rigid}
with $(H;\calq)=(G;\calp)$: this occurs if and only if the output is a list of morphisms as in (\ref{it_list}).
If $G$ is virtually cyclic and one-ended relative to $\calp$, then some peripheral subgroup has finite index in $G$
 so $(G,\calp)$ always satisfies the rigidity criterion.
One deduces immediately the following corollary.

\begin{cor}
\label{cor_decision_rigid}
Given a hyperbolic group with marked peripheral structure $(G;\calp)$ that does not split over a finite group relative to $\calp$,
 one can decide whether $(G;\calp)$  satisfies the rigidity criterion \ref{prop_alt}.
\end{cor}

One also 
deduces the following:

\begin{cor}\label{coro_list}
The extended isomorphism problem is solvable for  one-ended rigid hyperbolic groups with marked peripheral structure.  More precisely:

There is an algorithm which takes as input two hyperbolic groups with marked peripheral structures
$(G;\calp)$ and $(H;\calq)$ which don't split over a finite or a $\Zmax$-subgroup (relative to their peripheral structure),
and which decides whether $(G;\calp)$ is isomorphic to $(H;\calq)$ (respecting the peripheral structures), and computes the finite group
  $\Outm(G;\calp)$. 
\end{cor}

\begin{proof}[Proof of the corollary]
 One can decide whether $G$ or $H$ is virtually cyclic by Lemma \ref{lem_VC2}.
If $G$ or $H$ is virtually cyclic, the corollary follows easily from Lemma \ref{lem_autZ}.

Otherwise,  apply the algorithm of Theorem \ref{thm_IP_rigid} to $((G;\calp),(H;\calq))$.
  Because $(G;\calp)$  does not split over finite or $\Zmax$-subgroups,
 we get a list $\calf$ of morphisms $(G;\calp)\ra (H;\calq)$ containing a post-conjugate of any monomorphism.
Applying the same algorithm to $((H;\calq),(G;\calp))$, we get a similar list $\calf'$ of morphisms
$(H;\calq)\ra(G;\calp)$. 
Since $\calf$ and $\calf'$ contain a representative of the post-conjugacy class of any monomorphism,
$(G;\calp)$ and $(H;\calq)$ are isomorphic if and only if there exists
$\phi\in\calf$ and $\phi'\in\calf'$ such that $\phi\circ \phi'$ and $\phi'\circ \phi$ are  inner automorphisms.
This can be tested using the simultaneous conjugacy problem.

In order to compute $\Outm(G;\calp)$, we apply the algorithm of Theorem \ref{thm_IP_rigid} for $(G;\calp)$, and $(H;\calq) = (G;\calp)$. 
Since $G$ does not split over a finite or $\Zmax$-subgroup, we obtain a list of morphisms $\calf$. 
 As above, a morphism $\phi\in\calf$ is an automorphism if and only if there exists  $\psi\in\calf$
such that $\psi\circ\phi$ and $\phi\circ\psi$ are inner automorphisms.  
We can thus sort out which morphisms are automorphisms, and we can effectively check whether two of them are post-conjugate, 
and therefore we find a section  of $\Outm(G;\calp)$ in $\calf\subset\Autm(G;\calp)$.  
Using again a solution to the simultaneous conjugacy problem in $G$, one can compute the multiplication table of $\Outm(G;\calp)$ on these representatives.
\end{proof}

\subsubsection{Short morphisms}

We introduce the notion of \emph{short} morphism, which provide good representatives of a post-conjugacy class of morphism.
We follow \cite{DaGr_isomorphism}.

Let $H$ be a hyperbolic group, and   $\delta$  a hyperbolicity constant for a given set of generators. Let us define  
 for all $h\in H$, 
 the set $\calL_h \subset H$ of all elements $g$ at distance at least $1000\delta$ from $1$, and such that some geodesic segment 
 $[1,g]$ passes at distance at most   $20\delta$ from both $h$ and $gh$.
 
 \begin{lemma}\label{lem;qiers}
   For all $h$, $\calL_h \subset H$ is a \qiers (see Definition \ref{dfn_qiers}).
 \end{lemma}

 \begin{proof}
   Consider the ball of radius $20\delta$ around
   $h$, and for each element $x$ in it, the set of geodesic
   words representing $x$. Let $W(h)$ be this collection of
   words. This is finite and computable.  The language of all
   geodesics in $H$ is a regular language, and $\calL_h$ is the
   image of the geodesics words that have a prefix in $W(h)$
   and a suffix in $W(h^{-1})$. This subset of geodesic words
   is clearly regular, and this makes $\calL_h \subset H$ a
   \qiers. 
 \end{proof}

 Let $G$ be a group and choose two elements $a,b\in G$.
 \begin{dfn}
   Let us say that a morphism $\varphi:G \to H$ is \emph{long} (for
   the given choice of $a,b$) if there is $h$ of length $200\delta$
   such that (at least) one of the three situations occurs:
   \begin{itemize}
   \item $\varphi(a)$ and $\varphi(b)$ are in $\calL_h$, or 
   \item $\varphi(b)$, $\varphi(ab)$ and $\varphi(a^{-1}b)$ are
     in $\calL_h$, or
   \item $\varphi(a)$, $\varphi(ba)$ and $\varphi(b^{-1}a)$ are in $\calL_h$.  
   \end{itemize}
 \end{dfn}
 Let us say that a morphism $\varphi:G \to H$ is \emph{short} if it is not long.

\begin{prop}\label{prop;shortness}
  In any post-conjugacy class of morphisms $\phi:G\ra H$ that sends $a,b,ab^{\pm1}$ to infinite order elements generating a non-elementary subgroup of $H$,
  there is at least one, and at most finitely many short morphisms.
\end{prop}

The proposition follows form Propositions \ref{prop;finiteness_result}, and \ref{prop;existence_result} below.
Our argument is essentially extracted from \cite[Corollary 4.14 and 4.18]{DaGr_isomorphism}. 

For $g\in H$, we denote by $Z(g)$ its centraliser.

\begin{lemma} (Compare with \cite[Lemma 4.13]{DaGr_isomorphism})\label{lem_Kg}
  For all $g\in H$, there is a constant $K(g)$ such that if $d(h\m,Z(g))\geq K(g)$,
  and if $h_0\in [1,h]$ is the point such that $|h_0|=200\delta$, then $hgh^{-1}$ is in $\calL_{h_0}$.
\end{lemma}

\begin{proof}
Note that $d(h\m,Z(g))=d(1,hZ(g))$.
If $hgh^{-1} = h'gh'^{-1}$ then $h'^{-1} h \in Z(g)$ and so $hZ(g) = h'Z(g)$.  Therefore there is $K(g)$
such that if $d(1,hZ(g))\geq K(g)$, then $d(1,hgh^{-1})> 460\delta + d(1,g)$.  
Consider $h_0$ as defined by the lemma.

\begin{figure}[htbp]
  \centering
  \includegraphics[width=\textwidth]{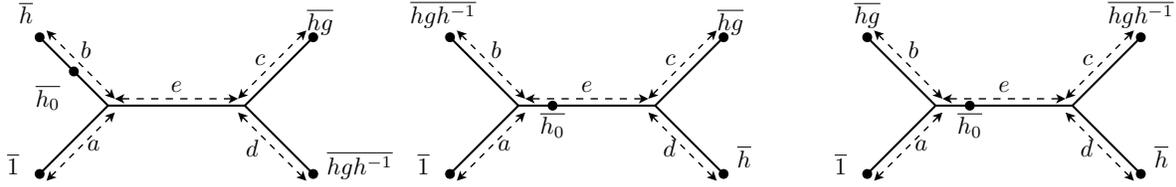}
  \caption{Three cases in Lemma \ref{lem_Kg}}
  \label{fig_Kg}
\end{figure}

We claim that if $h_0$ is at distance at least $20\delta$ from $[1,hgh\m]$, then $|hgh\m|\leq450\delta+|g|$. 
The Lemma will follow by applying the same argument to $g\m$ in place of $g$.
We work in an approximation tree $T$ for the quadrilateral $(1,h,hg,hgh^{-1})$, and we denote by $\ol x$ the point of $T$,
corresponding to some $x$ in this quadrilateral. 
There are three combinatorial possibilities for $T$ as shown on Figure \ref{fig_Kg},
and we denote by $a,b,c,d,e$ the lengths as shown on Figure \ref{fig_Kg}.
Note that in all cases $a\leq d(1,\ol{h_0})$
since otherwise, $\ol{h_0}\in [\ol 1,\ol{ghg\m}]$ so $d(h_0,[1,hgh\m]\leq 10\delta$, a contradiction.
Thus, $a\leq |h_0|+10\delta=210\delta$.
In the first case, we have $d(\ol 1,\ol{ghg\m})=a+e+d\leq a+e+|h|+10\delta\leq a+e+a+b+20\delta
= 2a+(e+b)+20\delta
\leq 420\delta +|g|+    
30\delta$, so $|ghg\m|\leq 460\delta+|g|$.
In the two other cases, since $d(1,h)= d(hg,hgh\m)$, we have $a+d=b+c$ up to an error of at most $10\delta$.
Thus 
in case 2, $d(\ol 1,\ol{hgh\m})=a+b\leq a+(a+d-c)+10\delta  \leq 2a+d+10\delta \leq 420\delta +|g|+20\delta$,
and $|hgh\m|\leq 450\delta+|g|$.
Similarly, in case 3, $d(\ol 1,\ol{hgh\m})=a+e+c\leq a+e+(a+b-d)+10\delta \leq 2a+(b+e)+10\delta \leq 420\delta +|g|+20\delta$.
\end{proof}

\begin{prop}\label{prop;finiteness_result} (Compare with \cite[Corollary 4.14]{DaGr_isomorphism})
  Consider a morphism $\varphi:G \to H$ that sends $a,b,ab^{\pm1}$ to torsion free elements generating a non-elementary subgroup of $H$.
Then only finitely many post-conjugates of $\varphi$ are short.
\end{prop}

\begin{proof}
   Let $K> \max\{ K(a^{\pm 1}), K(b^{\pm 1}), K(a^{\pm1} b^{\pm1}), K(b^{\pm1} a^{\pm1})\}$.  For all $x\in G$, write $N_x=
\{h\in H, d(h\m, Z(\varphi(x)))  \leq K\}$. 
Since, by assumption, $\langle \varphi(a) \rangle$ and $\langle\varphi(b) \rangle$ are infinite cyclic subgroups not in the same
elementary subgroup of $H$, the $K$-neighbourhoods of their centralisers have finite intersection, therefore $N_a \cap N_b$ is
finite. Similarly, $N_a\cap N_{ab}$ and $N_a\cap N_{a^{-1}b}$ are finite. 

Assume that, for an infinite family of different
elements $h_n$, the morphism $h_n\varphi(\cdot)h_n^{-1}$ is short. By the
previous lemma, for all $n$, $h_n \in N_a \cup N_b$ (falsification of the first point of the definition of long). Thus we can
assume that for all $n$, $h_n \in (N_a\setminus N_b)$, or that for all $n$, $h_n \in (N_b\setminus N_a)$. 
In the first case, $\varphi(b^{\pm 1} ) \in \calL_{(h_n)_0}$ (with notation
as in the previous lemma).  By falsification of the second point of the definition of long (and the previous lemma), $h_n \in
N_{ab} \cup N_{a^{-1} b}$ for all $n$. But both $N_a\cap N_{ab}$ and $N_a\cap N_{a^{-1}b}$ are finite (as we noted earlier). This
is a contradiction.
The second case is proved similarly using the third point of the definition of long.
\end{proof}

Define, for a morphism $\varphi$, the quantity $Q(\varphi) = \max\{|\varphi(a)|, |\varphi(b)| \}$.

\begin{lemma}(Compare with \cite[Lemma 4.17]{DaGr_isomorphism})\label{lem_raccourcir}
  Let $\varphi:G \to H$ be a long morphism, and let $h$ be as in the definition of long. Then $Q(h^{-1} \varphi(\cdot)h) <
  Q(\varphi)$.
  
\end{lemma}

\begin{proof}  First, it is obvious that if $g\in
  \calL_h$, then $| h^{-1} g h| <|g|$. This already indicates
  that if $\varphi$ satisfies the first point of the definition of
  long, the statement is true.  

By symmetry of the argument, we can assume
   that the second point of the definition holds, so $\varphi(b)$, $\varphi(ab)$ and $\varphi(a^{-1}b)$ are
  in $\calL_h$. As we said, $|h^{-1} \varphi(b) h|  <|\varphi(b)|$. 
  We need to show that $| h^{-1} \varphi(a) h|  < \max(|\phi(a)|,|\phi(b)|)$.
  Recall that since $\phi(b)\in\call_h$, $|\phi(b)|\geq 1000\delta$. 
  In particular, if $|\phi(a)|\leq 470\delta$, then 
  $|h\m\phi( a)h|\leq 2|h|+|\phi(a)|\leq 870\delta <|\phi(b)|$ and we are done.

 The discussion will now take place in the triangle $(1, \varphi(a), \varphi(ab))$.
We work in an approximation triangle $T$, and we denote by $\ol x$ a point of $T$ 
corresponding to a point $x$ of this triangle.
Consider $h_{b}\in [1,\phi(b)]$ (resp.\ $h_{ab}\in [1,\phi(ab)]$) at distance at most $20\delta$ from $h$,
and $v_a=\phi(a).h_b\in [\phi(a),\phi(b)]$.
If the segments $[\ol 1,\ol{h_{ab}}]$ and $[\ol{\phi(a)},\ol{v_a}]$ have a common point in $T$,
then $|\phi(a)|\leq  d(\ol 1,\ol{h_{ab}})+d(\ol{\phi(a)},\ol{v_a})+10\delta \leq 2|h|+70\delta\leq 470\delta$,
a case already treated.
If $\ol{h_{ab}}$ and $\ol{v_{a}}$ both lie in $[\ol 1,\ol{\phi(a)}]$,
then $|h\m\phi(a)h|<|\phi(a)|$ and we are done. 

\begin{figure}[htbp]
  \centering
  \includegraphics{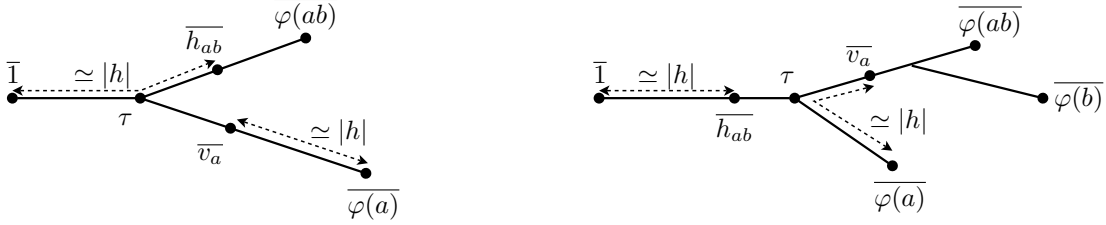}
  \caption{Two cases in Lemma \ref{lem_raccourcir}}
  \label{fig_raccourcir}
\end{figure}

 Let $\tau$ be the centre of $T$.
There are two remaining cases: either $\ol{h_{ab}}\in [\tau,\ol{\phi(ab)}]$ and $\ol{v_{a}}\in [\tau, \ol{\phi(a)}]$,
or $\ol{h_{ab}}\in [\ol 1,\tau]$ and $\ol{v_{a}}\in [\tau, \ol{\phi(ab)}]$ (see Figure \ref{fig_raccourcir}).
In the first case, we have $|h\phi(a)h\m|=d(h,\phi(a)h)\leq 40\delta + d(h_{ab},v_a)\leq 50\delta+ d(\ol{h_{ab}},\ol{v_a})
\leq 50\delta +d(\ol{\phi(a)},\ol{\phi(ab)})-d(\ol{\phi(a)},\ol{v_a})\leq 70\delta+|\phi(b)|-d(\phi(a),v_a)\leq 110\delta +|\phi(b)|-|h|
=|\phi(b)|-90\delta$ (recall that $|h|=200\delta$) and we are done.
In the second case, we consider an approximation tree of the quadrilateral $(1,\phi(a),\phi(ab),\phi(b))$.
If $d(\ol{v_a},\tau)\leq 50 \delta$, then $|h\m\phi(a)h|<|\phi(a)|$ and we are done. 
Let $\tau'$ be the centre of $(\ol{\phi(a)},\ol{\phi(ab)},\ol{\phi(b)})$.
Since $\phi(a\m b)\in\call_h$, $d(\ol{v_a},[\ol{\phi(a)},\ol{\phi(b)}])\leq 50\delta$.
In particular, either $\tau'\in [\ol{v_a},\ol{\phi(ab)}]$ (as on Figure \ref{fig_raccourcir}),
or $d(\tau',\ol{v_a})\leq 50\delta$.
In both cases, $d(\ol{h_{ab}},\ol{v_a})\leq d(\ol{h_{ab}},\ol{\phi(b)})+50\delta= d(\ol 1,\ol{\phi(b)})-d(\ol 1,\ol{h_{ab}})+50\delta
\leq |\phi(b)|-d(1,h_{ab})+70\delta\leq |\phi(b)|-110\delta$.
As above, this implies $|h\m\phi(a)h|<|\phi(b)|$.

\end{proof}

\begin{prop}\label{prop;existence_result}
  Let $\varphi:G \to H$ be a morphism, then some post-conjugate of $\varphi$ is short.
\end{prop}

\begin{proof}  
Choose a post-conjugate $\phi'$ of $\phi$ with $Q(\phi')$ minimal.
 By the previous lemma, $\phi'$ is short.
\end{proof}

Proposition \ref{prop;shortness} follows from Propositions \ref{prop;finiteness_result} and \ref{prop;existence_result}. 

\subsubsection{Listing morphisms}

Consider $(G;\calp)$ and $(H;\calq)$ two groups with marked peripheral structure, $a,b\in G$ generating a free subgroup, and some radius $R\geq 0$. 
We assume that we have a finite generating set for $G$, so that we can talk of $B_G(R)$, its ball of radius $R$ for the word metric.

Say that a morphism $\phi:(G;\calp)\ra(H;\calq)$ is \emph{almost injective with respect to $a,b,R$} if
$\phi(a),\phi(b),\phi(ab^{\pm1})$  have infinite order 
and generate a non-elementary subgroup of $H$, and $h$ is injective in restriction to $B_G(R)$.
Clearly, if $a,b$ generate a non-abelian free subgroup of $G$, 
then any monomorphism is almost injective with respect to $a,b,R$ for any $R\geq 0$.

From \cite{Delzant_sous-groupes} we know that there is a computable constant $N$ such that, for any two  elements 
$x,y$ of $H$, $x^N$, and $y^N$ generate a free group (which may be trivial or cyclic). 
Therefore, for $a,b\in G$ and any morphism $\phi:G\ra H$  such that $\phi(a)$ and $\phi(b)$ have infinite order,
$\grp{\phi(a),\phi(b)}$ is non-elementary if and only if
$[\phi(a)^N,\phi(b)^N]\neq 1$.
Also note that there exists a computable constant $K$ such that $x\in H$ has infinite order if and only if it has order greater than $K$.

\begin{prop} \label{prop;list1}
  There is an algorithm as follows.
Its input is a finitely presented group with a marked 
peripheral structure
 $(G;\calp)$, a solution of the word problem in $G$, two elements $a,b \in G$, 
and a hyperbolic group $(H;\calq)$, with a marked peripheral structure.

This algorithm terminates if and only if 
there exists $R_0>0$ such that there are only finitely many post-conjugacy classes of morphisms $(G;\calp)\ra (H;\calq)$ 
that are almost injective with respect to $a,b,R_0$.

If the algorithm stops, its output is a finite list of morphisms $(G;\calp)\ra (H;\calq)$,  
containing a post-conjugate every such almost injective morphism, and in particular, of any monomorphism $(G;\calp)\ra (H;\calq)$.
\end{prop}

The proof is similar to that of \cite[Theorem 4.4]{DaGr_isomorphism}, which treats the case of torsion free hyperbolic groups.

\begin{proof}
  Consider the following problem: given $R\geq 0$,
  and a family $\calF$ of morphisms $G\to H$, does there exist a
   morphism $\varphi : G\to H$ which is short, not in $\calf$,
sends the peripheral structure $\calp$ to $\calq$,
and is almost injective with respect to $a,b,R$.

If $\grp{X|\calr}$ is a presentation of $G$, morphisms $\phi:G\ra H$ correspond to solutions in $H$
of the system of equations $\calr$ over the set of variables $X$.
Since we have a solution of the word problem, we can compute the ball $B_G(R)$.
The injectivity on $B_G(R)$ can then be encoded by inequations, 
 the fact that $\phi(a),\phi(b),\phi(ab^{\pm1})$ have infinite order is encoded by inequations saying that they have order $>K$,
and the fact that $\grp{\phi(a),\phi(b)}$ non-elementary is encoded by
an inequation saying $[\phi(a)^N,\phi(b)^N]\neq 1$.
The fact that $\phi$ maps the marked peripheral structure $\calp=(P_1,\dots,P_k)$ to $\calq=(Q_1,\dots,Q_k)$ 
can be encoded using new variables $z_1,\dots z_k$ and equations saying that
 $z_i$ conjugates the tuple $\phi(P_i)$ to $Q_i$.
The fact that $\phi\notin\calf$ can be encoded by a system  of disjunctions of inequations.
Shortness of $\phi$ consists of a Boolean combination of conditions of the form $\phi(g)\in\call_h$
where $g\in\{a,b,a^{\pm1}b^{\pm1}\}$, and $h$ varies among all elements of $H$ of length $200\delta$.

Thus, the existence of $\phi$ in the problem above is equivalent to the existence
of a solution for a Boolean combination of systems of equations and inequations with rational constraints,
which can be rewritten as a disjunction of systems of equations and inequations with rational constraints.
Since $\calL_h$ is a \qiers (Lemma \ref{lem;qiers}), this is solvable by  Theorem \ref{theo;eq}. 

The algorithm starts with the empty list $\calf$, and with a radius $R=0$.
If there is a short morphism $\varphi : (G;\calp)\to (H;\calq)$, not in $\calf$,
and is almost injective with respect to $a,b,R$, it adds it to the list $\calf$, and increments $R$ by $1$.
It repeats this operation until there is no such $\phi$, in which case the algorithm stops and outputs $\calf$.
  
By Proposition \ref{prop;shortness} any post-conjugacy class of morphism  almost injective with respect to $a,b,R$ has a short representative.
Therefore, if the algorithm stops  at some value of $R$, every  post-conjugacy class of morphisms $(G;\calp)\ra(H;\calq)$ 
almost injective with respect to $a,b,R$ has a representative in $\calf$.

Conversely, assume that the set $C$
of post-conjugacy classes of morphisms $(G;\calp)\ra(H;\calq)$
almost injective with respect to $a,b, R_0$ is finite.
When $R>R_0$, the post-conjugacy class of any morphism added to $\calf$ by the algorithm lies in $C$.
Since there are only finitely many
short morphisms in each post-conjugacy class of morphism mapping $a,b$ to a non-elementary subgroup,
the algorithm will stop.
\end{proof}

\subsubsection{Splittings of $G$}

\begin{prop}\label{enumerate_ess_splittings} 
  There is an algorithm that, given a  hyperbolic group $(G;\calp)$ with a  peripheral 
  structure, enumerates all the non-trivial splittings of $(G;\calp)$ over $\Zmax$-subgroups.

There is a similar algorithm enumerating all the non-trivial splittings of $(G;\calp)$ over finite subgroups. 
\end{prop}

\begin{rem}
  In this statement, it makes no difference whether $\calp$ is a marked or unmarked peripheral structure.
\end{rem}

\begin{proof} 
List all presentations of $G$ using Tietze transformations.
Keep only those exhibiting  some non-trivial splitting over $\Zmax$-subgroups relative to $\calp$ (in which peripheral subgroups appear as explicit subgroups of vertex groups)  
This is possible since by Lemma \ref{lem_VC2}, it is possible to algorithmically check whether the corresponding edge group is $\Zmax$,
and in the case of an amalgam, that the vertex groups are not themselves $\Zmax$.

The case of splittings over finite groups is similar.
\end{proof}

\subsubsection{ Proof of Theorem \ref{thm_IP_rigid}}

We are now ready  to prove Theorem \ref{thm_IP_rigid}.

\begin{proof}[Proof of Theorem \ref{thm_IP_rigid}]

Recall that we are given  two non-elementary hyperbolic groups with marked peripheral structures $(G;\calp)$, $(H;\calq)$, 
and that we are required to algorithmically find a splitting of $G$ over a finite or $\Zmax$-subgroup, 
or a finite list a finite list $\calf$ of morphisms $(G;\calp)\ra (H;\calq)$, such that
any monomorphism $(G;\calp)\ra (H;\calq)$ has a post-conjugate in $\calf$.

Since $G$ is non-elementary, we can find $a,b\in G$ generating a non-abelian free subgroup.
Indeed, according to \cite{Delzant_sous-groupes}, there is a computable constant $N$ such that 
for any two elements $a,b\in G$,  $a^N$ and $b^N$ generate a (maybe cyclic) free subgroup. 
Enumerating all pairs of $N$-powers of elements of $G$ and using a solution of the word problem in $G$,
one can find a pair with $[ a^N, b^N]\neq 1$, so  $\grp{a^N,b^N}$ is free of rank $2$.

Our algorithm runs two machines in parallel. 
The first machine is the algorithm given by Proposition \ref{prop;list1}.
It stops if and only if there exists $R_0> 0$ such that
there are only finitely many post-conjugacy classes of homomorphisms
 $(G;\calp)\ra(H;\calq)$ almost injective with respect to $a,b,R_0$. 
If this machine stops  first, we output a finite list $\calf$ containing a post-conjugate of any
monomorphism  $(G;\calp)\ra(H;\calq)$, and we stop and output $\calf$.

In parallel, we run the  a machine which lists
all non-trivial splittings over a finite or $\Zmax$-subgroups as described in Proposition \ref{enumerate_ess_splittings}.
If this machine produces such a splitting before the first machine stopped, then we stop and output this splitting.

Our algorithm always stops.  Indeed, the rigidity criterion \ref{prop_alt} implies that if the first machine does not stop,
then $(G,\calp)$ has splitting over a finite or $\Zmax$-subgroup, and it will be  found by the second machine.
\end{proof}

\section{JSJ decompositions}\label{sec_JSJ}

Given two $G$-trees $T,T'$, one says that $T$ is \emph{elliptic} with respect to $T'$ if
 every edge stabiliser of $T$ fixes a point in $T'$.
In this case, one can construct a \emph{blowup} $\Hat T$ of $T$ relative to $T'$ as follows.
For each vertex $v\in T$, choose a $G_v$-invariant subtree $Y_v\subset T'$ in a $G$-equivariant way (if $G_v$ fixes no point, one can take the 
minimal $G_v$-invariant subtree),
and for each oriented edge $e$ of $T$, choose a point $p_e\in Y_{t(e)}$ fixed by $G_e$.
Then define $\Hat T$ as the minimal subtree of the refinement of $T$ obtained by replacing each vertex $v\in T$ by $Y_v$, and 
gluing back each edge $e$ on $p_e$.
The tree $\Hat T$ collapses to $T$ and dominates $T'$.

Consider $\cala$ a family  of subgroups of $G$, stable under conjugacy and taking subgroups. A JSJ splitting of $G$ over $\cala$
is an $\cala$-tree which is \emph{$\cala$-universally elliptic} (\ie elliptic with respect to any $\cala$-tree)
and which  dominates all $\cala$-universally elliptic splittings (see \cite{GL3a}).
The set of JSJ splittings is a deformation space called the \emph{JSJ deformation space} over $\cala$. 
Assuming that $\cala$ is invariant under automorphisms (which will always be the case here), 
the JSJ deformation space is invariant under $\Out(G)$.

A vertex group $G_v$ of a JSJ splitting is \emph{flexible} if there exists some $\cala$-tree in which $G_v$ does not fix a point.
Other vertex groups $G_v$ are  called \emph{rigid}. Equivalently, $G_v$ is rigid if and only if $(G_v;\calp_v)$ has no non-trivial splitting 
over $\cala$ relative to the incident edge groups.
 Note that in general, this notion of rigid vertex group is distinct from the notion in the rigidity criterion,
except when $\cala=\Zmax$.

Instead of $G$, one may start with a pair $(G;\calp)$ with $G$ finitely presented and $\calp$ a finite collection of finitely generated peripheral subgroups.
An $\cala$-tree is \emph{relative to $\calp$} if  peripheral subgroups are elliptic.
We say that this is an $(\cala,\calp)$-tree.
Then one can define the JSJ deformation space of $G$ relative to $\calp$ in this context: 
universal ellipticity means ellipticity with respect to $(\cala,\calp)$-trees,
and JSJ splittings are maximal universally elliptic $(\cala,\calp)$-trees.

\subsection{Virtually-cyclic JSJ decomposition}

Consider $G$ a one-ended hyperbolic group, and $\cala$ the class of finite or two-ended subgroups.
Since $G$ is one-ended, all $\cala$-trees have infinite edge stabilisers.
Bowditch constructed a JSJ splitting over $\cala$ from the boundary of $G$ in  \cite[Theorem 0.1]{Bo_cut}.

A \emph{hanging bounded Fuchsian group} of a graph of groups $\Gamma$ over virtually cyclic groups,
is a vertex group $(G_v;\calp_v)$ with its unmarked peripheral structure induced by $\Gamma$,
such that $G_v$ has a structure of bounded Fuchsian group such that
each group of $\calp_v$ is a subgroup of a boundary group of $G_v$.

The flexible vertices $(G_v;\calp_v)$ of Bowditch's splitting are hanging bounded Fuchsian groups, and
more precisely, the peripheral structure $\calp_v$ is precisely given by the boundary groups of $G_v$.
Note that bounded Fuchsian groups may have reflections.

\subsection{$\Z$-JSJ decomposition}

Because splittings over two-ended subgroups with finite centre produce only a finite group
of Dehn twists, we want to study only $\Z$-splittings, i.e.\ splittings over virtually cyclic subgroups
with infinite centre.
Of course, if $G$ has no $2$-torsion, the class of $\Z$-subgroup of $G$ is its class of two-ended subgroups.

Take for $\cala$ the class of all subgroups  of $\Z$-subgroups of $G$,
and note that any splitting of $G$ over $\cala$ is indeed a $\Z$-splitting since $G$ is one-ended.
By \cite{GL3a}, the JSJ deformation space over $\cala$ exists, we denote it by $\cald_\Z$.
We denote by $T_\Z$ the collapsed tree of cylinders of this deformation space.
 By Proposition \ref{prop;Tc_inv}, this is an $\Out(G)$-invariant splitting, and by Lemma \ref{lem_same_D}, this is a $\Z$-JSJ splitting. 
\\

\subsubsection{$\Z$-JSJ decomposition of orbifold groups}

Before describing $T_\Z$ in general, we first consider the case of a hyperbolic orbifold $\Sigma$.
 This might be a non-trivial decomposition if $\Sigma$ has mirrors.

Let's first describe quickly what an orbifold $\Sigma$ with mirrors looks like (see \cite{Scott_geometries} for additional details).
We view $\Sigma$ as $C/G$ where $G\subset\Isom(\bbH^2)$ is a non-elementary discrete subgroup acting cocompactly on 
a closed convex set $C\subset \bbH^2$ with geodesic boundary. 
If $g\in G$ is a reflection, its set of fixed points is an infinite geodesic $\gamma\subset\bbH^2$.
The image in $\Sigma$ of  $C\cap\gamma$ is a \emph{mirror} of $\Sigma$.
A mirror is either a circle or a segment, and is contained in the topological boundary of $\Sigma$ (\ie the boundary when one forgets
about the orbifold structure). 
If it is a segment, its endpoints may be in another mirror, or in the image of a boundary component.
Boundary components of $\calc$ are geodesics whose stabiliser is either infinite cyclic or infinite dihedral.
The corresponding subset of $\Sigma$ is either a circle or a segment whose two endpoints are singular points, lying in mirrors.
We call such a subset a \emph{peripheral} circle or segment. 

Thus, each component of the topological boundary of $\Sigma$ is either a peripheral circle, 
or is a union of mirrors and of peripheral segments.

We now recall the structure of splittings of an orbifold group over virtually cyclic subgroups.

\begin{lem}[{\cite[Theorem III.2.6]{MS_valuationsI}}]\label{lem;splitting_curve}
Let $(G;\calb)$ be a hyperbolic orbifold group with its peripheral structure, and $T$ be a $(G;\calb)$-tree
with virtually cyclic edge stabilisers.

Then $T$ is dual to a finite disjoint union of non-trivial two-sided simple $1$-suborbifold  
of $\Sigma$ 
which don't intersect the peripheral circles and segments.
\end{lem}

A non-trivial simple $1$-suborbifold is the generalisation of a simple closed curve.
It is the image in the orbifold $\Sigma$ of an infinite properly embedded topological line $\gamma\subset C\subset\bbH^2$,
with cocompact stabiliser, 
and disjoint from its translates (\ie $g.\gamma\cap\gamma=\es$ for all $g\in G\setminus \Stab(\gamma)$).
Two sided means that $\Stab(\Gamma)$ does not exchange the two connected components of $\bbH^2\setminus\gamma$.
Note that 
an embedded segment joining two mirrors is a two sided simple $1$-suborbifold.
An embedded segment joining two cone points with cone angle $\pi$ (or a mirror to a cone point of angle $\pi$) 
is a one-sided simple sub-orbifold.
In general, a two-sided simple $1$-suborbifold cannot contain any cone point.

\begin{rem} 
  Morgan-Shalen's result is stated for surfaces, but it easily extends to orbifolds.
Indeed, there is a finite Galois covering $\Sigma_0$ of $\Sigma$ which is a surface, and the action of $\pi_1(\Sigma)$ 
is dual to a family of disjoint simple closed curves.
Up to changing these curves to geodesics with multiplicities in $\Sigma$,  this family of curves is
$\pi_1(\Sigma_0)$-invariant, and the result easily follows.
\end{rem}

Since a two-sided simple $1$-suborbifold which intersects a mirror has infinite dihedral fundamental group, we get:

\begin{lem}\label{lem_curves_orbifold}
 Let $(G,\calb)$ be a hyperbolic orbifold group with its peripheral structure, and $G\actson T$ be a minimal action on a simplicial tree relative to $\calb$.

If edge stabilisers of $T$ are $\Z$-groups, then $T$ is dual to a finite disjoint union of two-sided simple closed curves in $\Sigma$, which 
don't intersect the peripheral subset,  the conical singularities and the mirrors of $\Sigma$.
\end{lem}

Let $N$ be a regular neighbourhood of the union of mirrors and peripheral segments of $\Sigma$, avoiding
cone singularities. Then the curves of $\partial N$ define a $\Z$-splitting of $G$ which we call
the \emph{mirrors} splitting of $G$.
This splitting might be trivial if $N$ is connected and $\Sigma\setminus N$ is an annulus, or a disk with at most one cone point. 

\begin{prop}\label{prop_mirrors}
The mirrors splitting of an orbifold with mirrors $(G;\calb)$ 
is a $\Z$-JSJ splitting of $(G;\calb)$.
\end{prop}

\begin{proof}
  By  
 Lemma \ref{lem;splitting_curve} any $\Z$-splitting of $G$ relative to $\calb$ is dual to a curve in $\Sigma\setminus N$,
so the mirrors splitting is clearly $\Z$-universally elliptic.
Let's prove maximality. Consider $T$ another $\Z$-splitting of $\calb$, dual to some disjoint union of  curves $c_i\subset \Sigma\setminus N$.
If some $c_i$ is not parallel to the boundary of $N$, then by \cite[Lemma 5.3]{Gui_reading}, there exists another 
non-peripheral simple closed curve $c'\subset \Sigma\setminus N$
which intersects $c_i$ in an essential way, so $T$ is not $\Z$-universally elliptic.
 It follows that all curves $c_i$ are parallel to the boundary of $N$, and $T$ is dominated by the mirrors splitting.
\end{proof}

Lemma \ref{lem_curves_orbifold} and Proposition \ref{prop_mirrors} extend obviously to a bounded Fuchsian group $G$
(\ie an extension of an orbifold group by a finite group $F$, see Section \ref{sec_bfg}).
Indeed, any minimal action of $G$ on a tree factors through $G/F$ as 
the fixed subtree of $F$ is non-empty and $G$-invariant.
\\

\subsubsection{$\Z$-JSJ decomposition of hyperbolic groups}

We now extend this description to the $\Z$-JSJ decomposition of any one-ended hyperbolic group.
 Recall a group of  \emph{dihedral type} is a virtually cyclic group with finite centre.

\begin{prop}\label{prop_ZJSJ}
  Let $G$ be a one-ended hyperbolic group, and $T$ be a JSJ splitting over two-ended subgroups.
Then one can obtain a $\Z$-JSJ splitting for $G$ by
\begin{itemize}
\item first refining $T$ using the mirrors splitting of all its hanging bounded Fuchsian groups;
\item then collapsing all edges whose stabiliser is of dihedral type.
\end{itemize}

In particular, flexible groups of the $\Z$-JSJ splitting of $G$ are hanging bounded Fuchsian groups
without reflection.
\end{prop}

\begin{rem}
  The same statement holds for a one-ended finitely presented group in general, with essentially the same proof.
\end{rem}

\begin{proof}
  By definition, $T$ is elliptic in any splitting of $G$ over two-ended subgroups,
and in particular, in any $\Z$-splitting of $G$.
Let $\Hat T$ be the refinement of $T$ obtained using the mirrors splitting of its hanging bounded Fuchsian groups.
By Proposition \ref{prop_mirrors}, $\Hat T$ is $\Z$-universally elliptic since in any $\Z$-splitting of $G$,
the boundary subgroups of the hanging bounded Fuchsian groups of $T$ are elliptic.
It follows that the $G$-tree $T_0$ obtained from $\Hat T$ by collapsing all edges whose stabiliser is not in $\Z$
is $\Z$-universally elliptic.

Let $T'$ be any  $\Z$-universally elliptic $\Z$-splitting.
Since $T'$ is a splitting over two-ended subgroups, rigid subgroups of $T$ are elliptic in $T'$.
Fix a hanging bounded Fuchsian group $G_v$ of $T$ with its peripheral structure $\calb$.
Every group of $\calb$ is elliptic in $T'$ by universal ellipticity of $T$.
Universal ellipticity of $T'$ says that the action of $(G_v,\calb)$ on $T'$ 
is elliptic with respect to any $\Z$-action
of $G_v$ relative to $\calb$. 
By Proposition \ref{prop_mirrors}, the mirrors splitting
of $G_v$ dominates $G_v\actson T'$. It follows that $\Hat T$ dominates $T'$.
Let $f:\Hat T\ra T'$ be an equivariant map.
Since  edge stabilisers of $T'$ are in $\Z$, $f$ collapses any edge whose stabiliser is of dihedral type,
and hence factors through $T_0$.

We prove that flexible vertices are hanging bounded Fuchsian groups
without reflection.

Vertex groups of $\Hat T$ are of three kinds: 
those coming from rigid vertex groups of $T$, which fix a point in any virtually cyclic splitting,
groups corresponding to the regular neighbourhoods of the mirrors and peripheral segments 
in the mirror splitting of a hanging bounded Fuchsian group,
and hanging  bounded Fuchsian groups
without reflection coming from the mirrors splitting.
Vertex groups of the second kind are elliptic in any $\Z$-tree by Proposition \ref{prop_mirrors}.
All edges groups of $\Hat T$ incident on a vertex of the third kind are $\Z$-groups.
Therefore, there are not collapsed in $T_0$.
It follows that a vertex group $G_{v_0}$ of $T_0$ that is not a hanging bounded Fuchsian group
without reflection is obtained by amalgamating along  subgroups of dihedral type, 
one or several vertices of the first or second kind.

If two such  vertices $u,v$  in $\hat{T}$ are joined by an edge $e$ with $G_e$ of dihedral type,
then $G_u$ and $G_v$ have  a unique common fix point in any $\Z$-tree $S$ as $G_e$ fixes no edge in $S$.
It follows that $G_{v_0}$ itself fixes a point in any $\Z$-tree. Hence it is not flexible.
\end{proof}

The following corollary applies for instance to an orbifold group with mirrors.
\begin{SauveCompteurs}{aut_hyp}
\begin{cor}[see also \cite{Fujiwara_outer}]\label{cor_aut_hyp}
  Let $G$ be a one-ended hyperbolic group. 
Then there is a finite index subgroup $\Out_f(G)$ of $\Out(G)$ which is an extension 
$$1\ra \mathit{Ab}  \ra \Out_f(G_v)\ra \prod_{i=1}^n PMCG_f^*(S_i)\ra 1$$
where  $\mathit{Ab}$ is virtually abelian, and
$PMCG_f^*(S_i)$ is a finite index subgroup  of $PMCG^*(S_i)$, the pure extended mapping class group of a surface with boundary
$S_i$ as defined in Section \ref{sec_orbi_mcg}.
\end{cor}
\end{SauveCompteurs}

\begin{proof}
  Let $T_\calz$ be the tree of cylinders of the $\calz$-JSJ deformation space.
Then $T_\calz$ is $\Out(G)$-invariant. Let $\Out'_0(G)$ be the finite index
subgroup of $\Out(G)$ acting as the identity on  $\Gamma=T_\calz/G$.
We apply \cite{Lev_automorphisms}.
Let $\rho:\Out(G)\ra \prod_{v\in V(\Gamma)} \Out(G_v)$.
For each vertex $v\in \Gamma$, let $\calp_v$ be a marking of the peripheral structure
of $G_v$ induced by the incident edge groups.
Then $\Outm(G_v,\calp_v)$ is denoted by $PMCG(G_v)$ in \cite{Lev_automorphisms}.
By \cite[Prop 2.3]{Lev_automorphisms}, the kernel of $\rho$ is virtually abelian,
and its image contains $\prod_{v\in\Gamma} \Outm(G_v;\calp_v)$ with finite index.

For each rigid vertex $v$, $\Outm(G_v;\calp_v)$ is finite by the rigidity criterion \ref{prop_alt}.
For each flexible vertex $v$, $(G_v;\calp_v)$ is isomorphic to a finite extension 
$1\ra F\ra G_v\ra \pi_1(\Sigma_v)\ra 1$
where $\Sigma_v$ is a $2$-orbifold without mirror.

Denote by $\ol\calp_v$ the image of $\calp$ in $\Sigma_v$.
As recalled in Section \ref{sec_orbi_mcg},
 $\Outm(\pi_1(\Sigma_v),\ol\calp_v)\simeq PMCG^*(S_v)$
where $S_v$ is the surface with weighted marked points underlying $\Sigma_v$.

A finite index subgroup $O_v$ 
of $\Outm(G_v;\calp_v)$ acts as the identity on $F$ and maps with finite kernel 
to $PMCG^*(S_v)$.
Since the set of equivalence classes of extensions of $\pi_1(\Sigma_v)$ by $F$ is finite,
the image of $O_v$
in $PMCG^*(S_v)$ is a subgroup $PMCG_f^*(S_v)$ of finite index.
Define $\Out_f(G)$ as the preimage  under $\rho$ of the product of the $O_v$'s
where $v$ ranges over flexible subgroups of $\Gamma$.
The kernel $Ab$ of the  
morphism $\Out_f(G)\onto \prod_v PMCG_f^*(S_v)$ is a finite extension of $\ker \rho$,
hence is virtually abelian.
\end{proof}

\subsection{Splittings over $\Zmax$-subgroups}

Recall that the class of maximal $\Z$-subgroups of a hyperbolic group $G$ is denoted by $\Zmax$.
We are interested in splittings of a hyperbolic group $G$ over $\Zmax$-subgroups
because of the rigidity criterion (Proposition \ref{prop_alt}) witnessing their better behaviour with respect to automorphisms.
We start by generalities about $\Zmax$-trees (\ie actions of $G$ on trees whose edge stabilisers are $\Zmax$).

\begin{lem}\label{lem_inter_Zmax}
Let $G$ be hyperbolic.
  Let $T$ be a $\Zmax$-tree, and $v$ a vertex. For any $\Z$-subgroup $H$ of $G$,
either $H\subset G_v$ or $H\cap G_v$ is  finite.

In particular, a subgroup of $G_v$ is $\Zmax$ as subgroup of $G_v$ if and only if 
it is $\Zmax$ as a subgroup of $G$.
\end{lem}

\begin{proof}
Assume $H\cap G_v$ is infinite. Then $H$ is elliptic so $H\subset G_u$ for some vertex $u$. 
If $u=v$ we are done.
Otherwise, $H\cap G_v$ fixes the arc $[u,v]$ and so does $H$ since edge stabilisers are $\Zmax$.
Thus $H\subset G_v$.
\end{proof}

\begin{lem}\label{lem_sup_Zmax}
Let $G$ be a one-ended hyperbolic group.
  Let $T,T'$ be $\Zmax$-trees, such that edge stabilisers of $T$ are elliptic in $T'$.

Then for each vertex $v\in T$, the action of $G_v$ on its minimal subtree in $T'$
has $\Zmax$ edge stabilisers,
and
the blowup $\Hat T$ of $T$ relative to $T'$ (as constructed in the beginning of Section \ref{sec_JSJ}) is a $\Zmax$-tree.
\end{lem}

\begin{proof}
Denote by $Y_v\subset T'$ the minimal $G_v$-invariant subtree.
Since $G$ is one-ended, 
  Edge stabiliser of $\Hat T$, and therefore of $G_v\actson Y_v$ are infinite.
By construction, edge stabilisers of $\Hat T$ and of $G_v\actson Y_v$ are of the form $G_v\cap G_{e'}$ where $v$ is a vertex of $T$ and $e'$ an edge of $T'$.
Lemma \ref{lem_inter_Zmax} concludes.
\end{proof}

Now we introduce the \emph{$\Zmax$-fold} of a $\Z$-tree $T$.
Given an edge $e$ of $T$, denote by $\Hat G_e$ the $\Zmax$-subgroup of $G$ containing $G_e$.
We define the $\Zmax$-fold  of $T$ as the  
minimal subtree of the quotient $T/\sim$
where we define $e \sim e'$ if $e=g.e'$ for some $g\in \Hat G_e$.

To prove that $T/\sim$ is indeed a tree, we can construct it using sequence of folds. Consider and edge $e$ such
that $G_e\subsetneq \Hat G_e$. Since $\Hat G_e$ is elliptic, one can find another such edge $e$
such that $o(e)$ is fixed by $\Hat G_e$. The quotient of $T$ by the equivariant equivalence relation generated by
$e\sim \Hat G_e.e$ is a fold (of type II in the terminology of \cite{BF_bounding}), and is in particular a $G$-tree. This new $G$-tree has 
strictly fewer orbits of edges with $G_e\subsetneq \Hat G_e$, and we proceed by induction.
Note however that $T/\sim$ may fail to be minimal, and it may be a trivial splitting even if $T$ is not:
this happens for example if $T$ is dual to a splitting $G=A*_C \Hat C$ where $\Hat C$ is the $\Zmax$-subgroup containing $C$.

The following lemma describes vertex stabilisers  of the $\Zmax$-fold, and their peripheral structure.

\begin{lem}\label{lem_Zfold}
  Let $T$ be a $\Z$-tree, and $T'$ its $\Zmax$-fold.
Consider a vertex $v'\in T'$ with non-elementary stabiliser, and $\{[G_{e'_1}],\dots,[G_{e'_k}]\}$
the peripheral structure of $G_{v'}$ induced by $T'$.

Then there exists a vertex $v\in T$ with peripheral structure
$\{[G_{e_1}],\dots,[G_{e_n}]\}$ for some $n\geq k$, such that,
\begin{enumerate}
\item $G_{e'_i}=\Hat G_{e_i}$ for $i\leq k$
\item $G_{e_i}$ is properly contained in $\Hat G_{e_i}$ for $k<i\leq n$
\item $G_{v'}=G_{v}(*_{G_{e_i}} \Hat G_{e_i})_{i=1}^{n}$
\end{enumerate}
\end{lem}

\begin{proof}
We view $T/\sim$ as obtained by a finite sequence of folds as above.
Each fold amounts to change an edge of the graph of groups labelled $G_u*_{G_e} G_v$ to $G_u*_{\Hat G_e} G_{v'}$, where
$\Hat G_e$ is assumed to be contained in $G_u$, 
and $G_{v'}=\Hat G_e *_{G_e} G_v$.
It is clear from this construction that for each vertex $v\in T$ with peripheral structure 
$\{[G_{e_1}],\dots,[G_{e_n}]\}$, the stabiliser of its image $v'\in T/\sim$
is $G_{v'}=G_{v}(*_{G_{e_i}} \Hat G_{e_i})_{i=1}^{n}$, its peripheral structure 
is $\{[\Hat G_{e_1}],\dots,[\Hat G_{e_n}]\}$.
Now going from $T/\sim$ to the minimal subtree $T'$, some edges attached on $v'$ may disappear.

We have to check that for any such edge $e'$ and any preimage $e$ in $T$, $G_{e}\subsetneq G_{e'}$.
Since $e'$ does not lie in the minimal subtree of $T/\sim$, there is some vertex $u'\in T/\sim$ which has valence one in $(T/\sim)/G$,
and with $G_{u'}\subset G_e'$.
Consider $u$ a preimage of $u'$ in $T$, and note that $G_{u}\subset G_{u'}$ is a $\Z$-group. 
As the map $T\ra T/\sim$ induces the identity on the quotient graphs, 
$u$ has valence one in $T/G$, so $G_{u}$ fixes no other vertex that $u$ by minimality of $T$.
The group $ G_{e'}=\Hat G_e$ is elliptic in $T$, and fixes no other vertex than $u$ as it contains $G_{u}$.
So $G_{u}= G_{e'}$, and $G_{e'}$ fixes no edge of $T$. Therefore, $G_{e}\subsetneq G_u= G_{e'}$.
\end{proof}

\subsection{JSJ decomposition over $\Zmax$-subgroups: the tree $\TZmax$}\label{sec_TZmax}

We want to construct a JSJ splitting for  $\Zmax$-splittings.
In the torsion-free case, this is related to what Sela calls the essential JSJ decomposition of $G$. 
For simplicity reasons, we will 
 rather use a variant of the $\Zmax$-JSJ decomposition:
a $\Z$-universally elliptic $\Zmax$-decomposition maximal for domination.

\begin{prop}\label{prop_JSJZmax}
Let $G$ be a hyperbolic group.
Consider the set $X$ of all $\Z$-universally elliptic $\Zmax$-trees (up to equivariant isomorphism). 

Then the set of largest elements of $X$ for domination is a (non-empty) deformation space denoted by $\cald_{\Zmax}$.
Moreover, $\cald_{\Zmax}$ contains the $\Zmax$-fold of any $\Z$-JSJ splitting.

We define $\TZmax$ as the collapsed tree of cylinders $T_c^*$ 
of any $T\in\cald_{\Zmax}$.
\end{prop}

 When there is a risk of ambiguity, we denote $\TZmax=\TZmax(G)$.

\begin{rem}
  By Lemma \ref{lem_same_D}, $\TZmax$ is itself a $\Zmax$-decomposition in $\cald_\Zmax$.
\end{rem}

We will describe flexible vertices of $\TZmax$ as \emph{orbisockets}, see section \ref{sec_orbi}.

\begin{rem}
A few (technical) words about the difference between a $\Zmax$-JSJ decomposition and the one considered above.
  Consider $(G,\calb)$ is the fundamental group of a once punctured Klein bottle $\Sigma$.
Splittings over $\Zmax$-subgroups of $(G,\calb)$ correspond to two-sided simple closed curves not bounding a M\"obius band.
But $\Sigma$ has a unique homotopy class of two-sided simple closed curve not bounding a M\"obius band.
This means that the $\Zmax$-JSJ splitting of $(G,\calb)$ is the splitting corresponding to this curve,
which differs from $\TZmax$ which is a trivial splitting. 
In other words, while it is true that any simple closed curve $c$ carrying a $\Zmax$-subgroup (\ie not bounding a M\"obius band)
has positive intersection number with some other simple closed curve $c'$ (\cite[th. 3.5]{Gui_reading}), 
there may not exist such curve $c'$ carrying a $\Zmax$-subgroup. 
This example extends to the case of a Klein bottle with a single cone point, and 
these are the only examples among  orbifolds of conical type.

\end{rem}

\begin{proof} 
Consider $T,T'$ two $\Z$-universally elliptic $\Zmax$-trees.
In particular, $T$ is elliptic with respect to $T'$. The blowup of $T$ relative to $T'$ is $\Z$-universally elliptic
 (because the stabiliser of any of its edge is contained in an edge stabiliser of $T$ or $T'$),
and is a $\Zmax$-tree by Lemma \ref{lem_sup_Zmax}.
This clearly implies that if nonempty, 
the set $\cald_\Zmax$ of maximal $\Z$-universally elliptic $\Zmax$-trees is a deformation space.

We now prove the second assertion (which implies that $\cald_\Zmax$ is non-empty).
Let $T_\Zmax$ the $\Zmax$-fold of a $\Z$-JSJ splitting $T_\Z$.
Since every edge group of $T_\Zmax$ contains an edge group of $T_\Z$ with finite index,
$T_{\Zmax}$ is $\Z$-universally elliptic.

We need to prove that $T_{\Zmax}$ dominates any other $\Z$-universally elliptic $\Zmax$-splitting $T'$.
By maximality of the $\Z$-JSJ splitting, 
there exists an equivariant map $f:T_\Z\ra T'$. 
Up to subdividing $T_\Z$, we can assume that $f$ maps each edge to an edge or a vertex. 
We claim that $f$ factors through $T_\Zmax$.
Consider two equivalent edges of $T_\Z$, so that $e=g.e'$ for some $g\in \Hat G_e$. 
If $e$ is collapsed under $f$, so is $e'$.
If not, $\Hat G_e$ stabilises $f(e)$ because $T'$ is a $\Zmax$-splitting, so 
$f(e')=f(e)$.
Thus $f$ factors through the $\Zmax$-fold of $T_\Z$ and the proposition follows.
\end{proof}

\begin{prop}\label{prop_TZmax_is_canonical}
Let $G_1,G_2$ be two one-ended hyperbolic groups.
Let $\Gamma_i$ be the quotient graph of groups of $T_{\Zmax}(G_i)$, and identify $G_i$ with $\pi_1(\Gamma_i,v)$.

Then any isomorphism $\phi:G_1\ra G_2$ is induced by an isomorphism of graphs of groups $\Phi:\Gamma_1 \to \Gamma_2$ (in the sense of definition \ref{dfn_Phi}).
\end{prop}
                        
\begin{proof}
Given any action $G_1\actson T$, denote by $\phi^* T$ the action $G_2\actson T$ induced by precomposition by $\phi\m$.
The definition of the deformation space $\cald_\Zmax$ being canonical,
for any $T_1\in \cald_\Zmax(G_1)$, the action $\phi^*T_1$ lies in $\cald_\Zmax(G_2)$.
Going to the collapsed tree of cylinders, we get that $\phi^* T_\Zmax(G_1)=T_\Zmax(G_2)$,
in other words, that there is a $\phi$-equivariant isomorphism between 
$T_{\Zmax}(G_1)$ and $T_{\Zmax}(G_2)$.
By Lemma \ref{lem_tree2Phi}, there exists an isomorphism of graphs of groups $\Gamma_1 \to \Gamma_2$. 
\end{proof}

\subsection{Orbisockets as flexible vertices}\label{subsec_orbi}

        We now study \emph{orbisockets}, which, as we will see, occur as flexible vertices of $\TZmax$.
        The most basic example is the fundamental group of a surface with boundary, to which one has added 
        some roots to the elements representing the boundary components.
        In \cite{Sela_structure}, Sela calls these groups ``sockets'' in the context of 
        vertex groups of a JSJ decomposition.  
        We see it as a deformation of the English word ``socks'' and of the french word ``socquettes'' 
        (meaning  low socks): 
        a pair of pants comes with  socks, that are attached to  boundary components, see Figure \ref{fig_chaussette}.
        Unfortunately, as we will see, \emph{les socquettes ne sont pas toujours propres}.

        \begin{figure}
          \centering
          \includegraphics{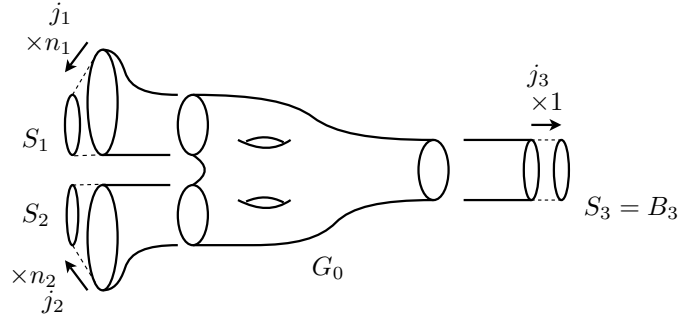}
          \caption{An orbisocket with two proper sockets ($n_1, n_2 \neq 1$), and an improper one.}
          \label{fig_chaussette}
        \end{figure}

        Let $G_0$ be a bounded Fuchsian group without reflection, with its peripheral structure $\calb=\{[B_1],\dots,[B_n]\}$.
        Let $k\in\{0,\dots,n\}$, and consider some $\Z$-groups $S_1,\dots,S_k$  with monomorphisms
        $j_i:B_i\ra S_i$ for $i=1,..,k$. We assume that no $j_i$ is onto (see Figure \ref{fig_chaussette}, where $n=3$ and $k=2$). 
        Let $\calo$ be the multiple amalgam $\calo=G_0(*_{B_i} S_i)_{i=1}^k$ 
        (\ie a tree of groups with a central vertex $G_0$ and one edge for each $i\in\{1,\dots,k\}$).
        For notational convenience, we define $S_i=B_i$ and $j_i=\id$ for $i\geq k+1$; all the groups $S_1,\dots,S_n$ and their conjugates in $\calo$ are called \emph{sockets}.
        The \emph{proper sockets} are the conjugate of the groups $S_1,\dots,S_k$ (characterised by $B_i\neq S_i$), the other ones are \emph{improper}.

        \begin{dfn}[Orbisocket]\label{dfn_orbisocket}
          An \emph{orbisocket} is a group $(\calo;\calp_\calo)$ with unmarked peripheral structure
          obtained from a bounded Fuchsian group without reflection $G_0$ as a tree of groups $\calo=G_0(*_{B_i} S_i)_{i=1}^k$ as described above,
          and whose peripheral structure $\calp_\calo$ consists of all improper sockets, possibly with additional proper sockets.

          A decomposition of $\calo$ as a tree of groups as above is \emph{a socket decomposition} of $\calo$.

         A \emph{hanging orbisocket} in a graph of groups $\Gamma$ is 
          a vertex $v\in\Gamma$ such that the vertex group $G_v$ together with its peripheral structure $\calp_v$ induced by $\Gamma$
          is isomorphic to some orbisocket $(\calo;\calp_\calo)$.
         \end{dfn}

         We also say that $(G_v,\calp_v)$ is a hanging orbisocket.
         
        Starting with an orbifold group and a boundary subgroup $\grp{b}$, the amalgam consisting in adding a square root to $b$ is equivalent
        to gluing a M\"obius band to the orbifold. In general, we say that $S_i$ is  a  \emph{M\"obius socket} if 
        \begin{itemize}
        \item 
          $B_i$ has index $2$ in $S_i$, and $B_i$ and $S_i$ have the same maximal normal finite subgroup $F$ ($F$ is also the
          maximal finite normal subgroup of $G$)
        \item $S_i$ does not lie in the peripheral structure of $\calo$. 
        \end{itemize}
        In this case, $G_0'=G_0*_{B_i} S_i$ is itself a bounded Fuchsian group without reflection, and one can define a new orbisocket with same group
        and same peripheral structure, but with $G'_0$ as the Fuchsian group, and $S_1,\dots, S_{i-1},S_{i+1},\dots,S_k$ as sockets.
        Doing this for all M\"obius sockets, one can transform an orbisocket into an orbisocket without M\"obius socket.
        This shows that the socket decomposition of $\calo$ is not determined by $\calo$ and its peripheral structure.
        However, we will see in the next section that once M\"obius sockets have been removed, the socket decomposition is unique.

\begin{dfn}\label{dfn_flex}
A vertex $v$ of $\TZmax$ is \emph{flexible} if there exists a $\Zmax$-tree in which $G_v$ is not elliptic.   
\end{dfn}

\begin{prop}
Let $G$ be a one-ended hyperbolic group. 

Then, any flexible vertex of $\TZmax$ is a hanging orbisocket. 
\end{prop}

\begin{proof}
Let $T_\Z$ be a $\Z$-JSJ tree, $T$ the $\Zmax$-fold of $T_\Z$ (see section \ref{sec_TZmax}).

We first prove that if $v$ is a flexible vertex of $T$, then $v$ is a hanging orbisocket. \
Write $G_v=G_{\Tilde v}(*_{G_{\Tilde e_i}} \Hat G_{\Tilde e_i})_{i=1}^{n}$ for some $\Tilde v\in T_\Z$ as in Lemma \ref{lem_Zfold},
and consider a $\Zmax$-tree $S$ in which $G_v$ is not elliptic.
Since flexible vertices of $T_\Z$ are hanging bounded Fuchsian group without reflection (Proposition \ref{prop_ZJSJ}), 
we need only to prove that $\Tilde v$ is flexible.
Assuming by contradiction that $G_{\Tilde v}$ is rigid, then $G_{\Tilde v}$ fixes a point $u\in S$,
and groups $\Hat G_{\Tilde e_i}$ fix some $u_i\in S$. 
If $u_i\neq u$, then $G_{\Tilde e_i}$ fixes the arc $[u_i,u]$, and so does $\Hat G_{\Tilde e_i}$ 
since $S$ is a $\Zmax$-tree.
It follows that $G_v$ fixes $u$, a contradiction.

Now recall that  the tree $\TZmax$ is, by definition,  
the collapsed tree of cylinders $T_c^*$ of $T$.
Since $T$ and $T_c^*=\TZmax$ are $\Zmax$-trees in the same deformation space (Lemma \ref{lem_same_D}),
they have the same non-elementary vertex stabilisers (see \cite[Cor. 4.5]{GL2} for this general fact).
Moreover, as they are two $\Zmax$-trees in the same deformation space, 
$T$ and $T_c^*$ induce the same peripheral structure on such a vertex stabiliser $G_v$.
Indeed, the stabilisers of edges incident on $v$
are characterised as infinite groups of the form $G_v\cap G_u$ with $G_u\neq G_v$ non-elementary vertex stabilisers
(see \cite[Prop. 4.10 and Def. 4.11]{GL2} for a general statement or more details).
Since flexible vertices of $T$ are hanging orbisockets, so are flexible vertices of $T_c^*$.

\end{proof}

\begin{rem}
  In fact, some hanging Fuchsian groups of $T_\Z$ may fail to split over $\Zmax$-subgroups although they split over some $\Z$-group 
(this is for instance the case of a twice punctured projective plane). They give rigid vertices of $\TZmax$.
\end{rem}

\begin{lem}\label{lem_orbi_are_ell}
  Let $\Gamma$ be a $\Zmax$-splitting, and $v\in\Gamma$ be a hanging orbisocket.
Then $G_v$ is elliptic in $\TZmax$.
\end{lem}

\begin{proof}
  Let $\Gamma'$ be obtained from $\Gamma$ by refining $v$ using its socket decomposition.
Let $v'\in\Gamma'$ be the corresponding hanging Fuchsian vertex. By \cite[Prop.~7.16]{GL3a}, 
$G_v'$ is elliptic in any $\Z$-universally elliptic splitting, hence fixes some vertex $u\in \TZmax$.
Each socket group $S_i$ of $v'$ fixes a vertex $u_i\in\TZmax$ as some finite index subgroup fixes $u$.
If $u_i\neq u$, then $S_i$ fixes $[u_i,u]$ as edge stabilisers of $\TZmax$ are $\Zmax$.
It follows that $G_v$ fixes $u$.
\end{proof}

\begin{lem}\label{lem_courbes_orbisocket}
          Let $(\calO;\calp_\calo)$ be a hanging orbisocket of $\TZmax$. 
Consider a socket decomposition $\calO=G_0(*_{B_i}S_i)_{i=1}^p$ of $(\calO;\calp_\calo)$ without M\"obius sockets.   
Let $T$ be a $\Zmax$-decomposition of $G$ in which $\calo$ is not elliptic.

Then the minimal $\calo$-invariant subtree $Y_\calo$ of $T$ is dual to a finite set of curves in the following sense:
it is obtained by first refining the socket decomposition of $\calo$ using a splitting of 
$G_0$ dual to a finite set of disjoint non-peripheral  $2$-sided simple closed curves in the orbifold underlying $G_0$,
and then by collapsing all the edges of the socket decomposition.
\end{lem}

\begin{proof}

Viewing $\TZmax$ as the $\Zmax$-fold of a $\Z$-JSJ splitting $T_\Z$,
Lemma \ref{lem_Zfold} implies that all (proper and improper) sockets $S_i$ of $\calo$ have a finite index subgroup
 fixing an edge of $T_\Z$ and are therefore universally elliptic, hence elliptic in $T$.
In particular, boundary subgroups $B_i$ of $G_0$ are elliptic in $T$.
          Let $Y_0\subset Y_\calo\subset T$ be the $G_0$-minimal subtree. 
          By Lemma \ref{lem_curves_orbifold}, $G_0\actson Y_0$ is dual to a finite set $c_1,\dots, c_n$ 
          of non-peripheral disjoint simple closed curves 
          in the orbifold $\Sigma_0$ underlying $G_0$. 
          Because each of these curves is non-peripheral, the $\Zmax$-subgroup of $G$ containing its stabiliser 
          is contained in $G_0$. This implies that two distinct $\calo$-translates of $Y_0$ have no edge in common.

          Let $S'$ be the refinement of the socket decomposition dual to the curves $c_1,\dots, c_n$,
          and $S$ be obtained from $S'$ by collapsing all edges of the socket decomposition.
          Since each $B_i$ fixes a point $u_i\in Y_0$, the corresponding socket $S_i$ fixes some point, hence fixes $u_i$
          since edge stabilisers of $T$ are $\Zmax$. This means that the inclusion $Y_0\ra Y_\calo$ extends to
          an equivariant map $f':S'\ra Y_\calo$ which is constant on edges of the socket decomposition. 
          This map thus factors into a map $f:S\ra Y_\calo$.
          Since $f$ is injective on $Y_0$ and since distinct translates of $Y_0$ have no edge in common,
          $f$ locally injective, hence an isomorphism.
\end{proof}

\subsection{The orbisocket decomposition is canonical}
\label{sec_orbi_canonical}

In this section, we prove that the socket decomposition of an orbisocket is canonical as long as there is no M\"obius socket.
We prove this by showing that this decomposition is actually a $\Z$-JSJ decomposition.
We first prove one-endedness.

 One says that $G$ is one-ended relative to a peripheral structure $\calp$ if
 $(G;\calp)$ has no non-trivial splitting over a finite group (relative to $\calp$).

        \begin{prop}\label{prop_orbi_ends}
                Let $\calO$ be an orbisocket, with its peripheral structure $\calp_\calo$.
                Then $\calO$ is one-ended relative to $\calp_\calo$.
        \end{prop}

        In particular, if $\calo$ has no improper socket, then $\calo$ is one-ended (absolutely).

        The following lemma will be proved below.
        \begin{lemma}\label{lem_one_end_rel}
          Consider a non-trivial decomposition of $G$ into an amalgam  $G=A*_C B$ or an HNN extension $G=A*_C$
          with $C$ virtually cyclic.
          If $G$ is one-ended relative to $C,C_1,\dots, C_n$, then it is one-ended relative to $C_1,\dots,C_n$.

          In particular, if $A$ and $B$ are one-ended relative to the conjugates of $C,C_1,\dots C_n$ contained in 
          $A$ or $B$, then $G$ is one-ended relative to $C_1,\dots,C_n$..
        \end{lemma}

        \begin{proof}[Proof of Proposition \ref{prop_orbi_ends}]
          Let $\calo$ be an orbisocket groups, $G_0<\calo$ be the corresponding bounded Fuchsian group, 
          $B_1,\dots,B_n$ be the boundary subgroups of $G_0$,  
	  and $S_{1},\dots,S_k$ be the improper sockets of $\calo$. 
          We know that $G_0$ is one-ended relative to $B_1,\dots B_n$.
          Lemma \ref{lem_one_end_rel} says that $G_0*_{B_{n}} S_{n}$ is one-ended relative to $B_1,\dots B_{n-1}$,
          and repeating this argument, that $\calo$ is one-ended relative to $B_1,\dots B_{k}$.
          Since all improper sockets $B_1,\dots ,B_k$ appear in $\calp_\calo$, the lemma follows.
        \end{proof}

        \begin{proof}[Proof of Lemma \ref{lem_one_end_rel}]
          We prove the contraposition of the first assertion.
          Assume that $G$ is not one-ended relative to $C_1,\dots,C_n$.
          Let $G\actson S$ be a minimal Bass-Serre tree of a Stallings-Dunwoody decomposition of $G$ relative to $C_1,\dots,C_n$,
          \ie a splitting over finite groups relative to $C_1,\dots,C_n$, maximal for domination (existence follows from  relative Dunwoody accessibility
          see for instance \cite[Th. 4.8]{GL3a}).

         Let $T$ be the Bass-Serre tree of the decomposition $G=A*_C B$ (or $G=A*_C$).         
         Assume first that some vertex stabiliser $G_v$ of $S$ is not elliptic in $T$.
         Let $M\subset T$ be the minimal subtree, and consider $e$ an edge of $M$.
         By one-endedness of $G_v$  relative to the conjugate of $C_1,\dots,C_n$ it contains, 
         $G_v\cap G_e$ is infinite. Since $G_e$ is conjugate to $C$, and thus virtually cyclic,
         a finite index subgroup of $C$ is contained in $G_v$, so $C$ is elliptic in $S$, so $G$ is not one-ended relative to $C,C_1,\dots,C_n$.
         
         Assume now that every vertex stabiliser of $S$ is elliptic in $T$, so $S$ dominates $T$.
         Then one may factor $S\ra T$ into a collapse map $S\ra S_0$ 
	 followed by a sequence of folds $S_0,S_1,\dots, S_n=T$ as in \cite{BF_complexity}.
         Each $S_i$ is non-trivial because $T$ is non-trivial.
         Let $i_0$ be the last index such all edge stabilisers of $S_i$ are finite.

         Then there is an edge $e$ of $S_{i_0}$, and some $g\in G_{o(e)}$ of infinite order which fixes
         the image of $e$ in $S_{i_0+1}$.
         Then $\grp{g}$ fixes an edge 
         in $T$, so it has finite index in some conjugate of $C$.
         Since $g$ is elliptic in $S_{i_0}$, so is $C$, and $S_{i_0}$ is a splitting of $G$ with finite edge stabilisers
         relative to $C,C_1,\dots,C_n$.
        \end{proof}

Since an orbisocket is one-ended (relative to its peripheral structure), 
we now consider its $\Z$-JSJ decomposition.

\begin{lem}\label{lem_scindt_flexible}
 Let $(G,\calP)$ be a hyperbolic group with a peripheral structure in
 $\Z$.   Assume that there is a splitting  of  $(G,\calp)$ over a $\Z$-group $G_e$,  
whose Bass-Serre tree $S$ is not $\Z$-universally elliptic. 

Then, either $G_e$ is in $\Zmax$, or is of index $2$ in its maximal virtually cyclic subgroup, with same maximal finite normal  subgroup.
\end{lem}

\begin{proof}
Let $T$ be a virtually-cyclic JSJ decomposition of $G$  relative to $\calp$.
By \cite[Th. 7.33(2)]{GL3a}, 
 its flexible vertices are hanging bounded Fuchsian groups (maybe with reflections).
Since $S$ is not universally elliptic, some flexible subgroup $G_v$ of $T$ is not elliptic in $S$.
Since $G_v$ is a finite extension of an orbifold group $\pi_1(\Sigma)$,
the action of $G_v$ on its minimal subtree in $S$ is dual to a non-peripheral simple two-sided $1$-suborbifold $\gamma$ contained in $\Sigma$.
Let $C$ be the preimage in $G_v$ of its fundamental group.
Up to conjugacy, one can assume $C\subset G_e$. Thus, $C$ has infinite centre so $\gamma$ intersects no mirror.

 Since $C$ is not universally elliptic, 
$\gamma$ does not bound an annulus with a circular mirror (neither does it bound a regular neighbourhood of the mirrors and peripheral segments of $\Sigma$).
If $\gamma$ does not bound a  M\"obius band or a disk with two cone points of angle $\pi$,
then $C$ is $\Zmax$ (even maximal virtually cyclic).
Otherwise, $C$ has index $2$ in its maximal virtually cyclic subgroup,
with same maximal normal finite subgroup $\Hat C$ (and $\Hat C$ has infinite centre only in the case of a M\"obius band).
The lemma follows.
\end{proof}

        \begin{prop} \label{prop_JSJ_orbi}
                Let $(\calO;\calp_\calo)$ be an orbisocket. Let $\calo\actson T$ be the Bass-Serre tree 
                of a socket decomposition of $\calo$  without M\"obius socket.
                
                Then $T$ is a $\Z$-JSJ splitting of $\calo$ relative to $\calp_\calo$,
                and it is its own tree of cylinders.
        \end{prop}

        \begin{proof}
          As discussed above, $\calO$ is one-ended relative to $\calp_\calo$.
        Consider $\calo\actson T$ the Bass-Serre tree of the sockets decomposition of $\calo$, and $\calo\actson T_{JSJ}$ a tree in the 
        $\Z$-JSJ deformation space of $(\calo;\calp_\calo)$. We use the same notations as in Definition \ref{dfn_orbisocket}: 
        $G_0$ is the Fuchsian group with boundary subgroups $B_i$, and $S_i$ are the sockets.

        We claim that $T$ is universally elliptic (for the class of $\Z$-splittings of $\calo$ relative to $\calp_\calo$).
        Assume on the contrary that the edge group $B_i$ is not $\Z$-universally elliptic, 
        and note that $S_i$ is the maximal $\Z$-subgroup of $\calo$ containing $B_i$.
        Then by Lemma \ref{lem_scindt_flexible},
        either  the inclusion $B_i\subset S_i$ is of M\"obius type in contradiction with the hypothesis, or $B_i=S_i$.
        If $B_i=S_i$, $[S_i]$ lies in the peripheral structure $\calp_\calo$, 
        and is universally elliptic since we consider splittings relative to $\calp_\calo$.
        Thus $T$ is $\Z$-universally elliptic, and therefore dominated by $T_{JSJ}$.

        By \cite[Prop. 7.4]{GL3a}, 
	the hanging Fuchsian group $(G;\calp)$ fixes a point in $T_{JSJ}$.
        Since $B_i\subset G_v$ and $S_i$ contains $B_i$ with finite index, $S_i$ also fixes a point in $T_{JSJ}$.
        Thus, $T$ dominates $T_{JSJ}$, and $T$ is a $\Z$-JSJ splitting of $(\calo;\calp_\calo)$.
        
        Since a boundary subgroup of an orbifold group is malnormal in it, 
        each cylinder of $T$ is the ball of radius $1$ around a point stabilised by a conjugate of $S_i$.
        It immediately follows that $T$ is its own tree of cylinders.

        \end{proof}

\section{Isomorphism problem for orbisockets}\label{sec_orbi}

        \subsection{Recognition of  basic orbisockets}

        Consider an orbisocket $(\calO,\calp_\calo)$ with its peripheral structure. 

        \begin{dfn}\label{dfn_basic}
          We say that $(\calo,\calp_\calo)$ is a \emph{basic} orbisocket if it admits no non-trivial $\Zmax$ relative splitting, 
          and if $\calo$ has at least one improper socket.
        \end{dfn}

    Let us remark that if an orbisocket admits no non-trivial $\Zmax$ relative splitting, and has only proper sockets, then it is a rigid    one-ended hyperbolic group (see Proposition \ref{prop_orbi_ends}).

        Basic orbisockets are the hanging orbisockets one gets in a maximal $\Zmax$-splitting of a non-rigid one-ended hyperbolic group.
        The goal of this section is to describe all possible basic orbisockets, and to provide an algorithm recognising them.

        \begin{thm}\label{thm_recon_orbisocket}
          There exists an algorithm which takes as input  a hyperbolic group $(G;\calp)$ with an unmarked peripheral structure (given
          by a presentation and generating sets of peripheral subgroups)
          and which decides whether $(G;\calp)$ is isomorphic to a basic orbisocket.
          
          If this is the case, the algorithm provides a socket decomposition of $(G;\calp)$. 
        \end{thm}

        \begin{rem}
          In particular, when $(G;\calp)$ is an orbisocket, the theorem allows one to know which sockets are proper.
        \end{rem}
        
        Several cases will be considered, and Theorem \ref{thm_recon_orbisocket} 
        will follow from Propositions \ref{prop_recon_rigid}, \ref{prop_recon_3b} and \ref{prop_recon3c}.

     \begin{rem}
       Although $(G;\calp)$ has a finite group of outer automorphisms, one cannot use directly the solution of the isomorphism problem
       for rigid hyperbolic pairs (Theorem \ref{thm_IP_rigid}) to answer this question because there are infinitely many candidate orbisockets  as
       we have no a priori bound on the index of the sockets.
     \end{rem}

        \subsubsection{ The underlying orbifold}

        \begin{lem}
	  For every  orbisocket $(\calo,\calp_\calo)$, 
	there is a socket decomposition (possibly with M\"obius sockets) such that   the orbifold underlying the decomposition
          is orientable.

        \end{lem}

        \begin{proof}
          If the underlying orbifold $\Sigma$ of $\calo$ is non-orientable, then $\Sigma$ contains an embedded M\"obius strip.
          Cutting $\Sigma$ along its boundary gives a new representation of $\calo$ as a socket decomposition with one additional M\"obius socket.
          Since this operation decreases the complexity 
          of the underlying orbifold (measured for instance by the opposite of its Euler characteristic relative to the boundary), 
          this operation can be performed only finitely many times.
        \end{proof}
        
        From now on, we only consider orbisocket decompositions of $(\calo,\calp_\calo)$ whose underlying orbifold $\Sigma$ is orientable.

        \begin{lemma}\label{lem_class_orbi_1}
          Consider a basic orbisocket $(\calo,\calp_\calo)$.
	  Then it has an orbisocket decomposition whose underlying 2-orbifold $\Sigma$ is either 
          \begin{itemize}
          \item a disk with two cone points (type 1),
          \item a twice punctured sphere with one cone point (type 2),
          \item or a thrice punctured sphere (type 3).
          \end{itemize}
          Moreover it has at most three, and at least one, peripheral subgroups. 
        \end{lemma}
        
        \begin{proof}
          An orientable orbifold of conical type with positive genus, or containing more than 
          three boundary components or conical singularities contains a non-peripheral simple closed curve.
          This defines a $\Zmax$-splitting of $(\calo,\calp_\calo)$, contrary to the definition of a basic orbisocket.
          Since $\Sigma$ has at most three boundary components, 
          the orbisocket has at most three peripheral subgroups. By definition of a basic orbisocket, at least one socket is improper,
          so there is at least one peripheral subgroup.
        \end{proof}
  
        \subsubsection{Algebraic characterisation}

        A basic orbisocket $(\calo,\calp_\calo)$ is a virtually free group with a particular peripheral structure.
        We aim to write $\calo$ as a graph of finite groups, and to read its peripheral structure there.

        \begin{lemma} \label{lem_type1_CNS} 
          Let $(G;\calp)$ be a hyperbolic group with an unmarked $\Z$-peripheral structure.
          
          Then $(G;\calp)$ is isomorphic to a basic orbisocket of type 1 (as defined in Lemma \ref{lem_class_orbi_1})
          if and only if  
          \begin{itemize}
          \item  $G$ splits as a graph of finite groups (not relative to $\calp$) $G =\pi_1 \Big(
            \xymatrix{ \ar@{{*}{-}{*}}[r]_F^{ \langle \sigma_1,
                F \rangle  \quad \langle F,\sigma_2 \rangle } & } \Big)
            $, 
            with $F$  normal in $G$, $\sigma_i \notin F$, with $\sigma_1$ and $\sigma_2$ not both of order 2 in $G/F$.
          \item $\calp=\{[\langle \sigma_1 \sigma_2,F \rangle]\}$
          \end{itemize}
        \end{lemma}

        \begin{proof}
          If $(\calo,\calp_\calo)$ is a basic orbisocket of type $1$ over the bounded Fuchsian group $G_0$,
          then it has at most one socket, and by definition of a basic orbisocket, it has at least one improper socket.
          Therefore, the socket decomposition of $\calo$ is the trivial decomposition $\calo=G_0$, 
          and its peripheral structure consists in the boundary subgroups of $G_0$.
          Since $G_0/F\simeq \bbZ/r_1 *\bbZ/r_2$, the two assertions follows.

          Conversely, assume that $(G;\calp)$ satisfies both points of the lemma.
          Then $G/F$ is the fundamental group of a disk with two cone points of angles
          $2\pi/r_1$, $2\pi/r_2$ where $r_i$ is the order of $\sigma_i$ in $G/F$ and the boundary subgroup is conjugate to $\grp{\sigma_1\sigma_2}$.
          It follows that $G$ is a basic orbisocket with one improper socket $\grp{\sigma_1\sigma_2,F}$ and no proper socket.
        \end{proof}

        \begin{lemma} \label{lem_type2_CNS} 
          Let $(G;\calp)$ be a hyperbolic group with an unmarked $\Z$-peripheral structure.
          
          Then $(G;\calp)$ is isomorphic to a basic orbisocket of type 2 if and only if  
          \begin{itemize}

          \item  $G$  splits as a graph of finite groups (not relative to $\calp$)
            $$\displaystyle G=\pi_1 \Big( \input{HNN1.pst} \Big )$$
            with             $\sigma \notin F$ and $\sigma$ normalises $F$.  
          \item Denote by $H_1=F_1*_{F_1}$ the HNN extension on the right. Then there exists $b\in H_1$,
            with translation length $n>0$ in the corresponding Bass-Serre tree,
            such that $b$ normalises $F$, and 
            \begin{itemize}
            \item either $\calp=\{ [\grp{b \sigma,F}],[H_1] \}$
            \item             or $\calp=\{ [\grp{b \sigma,F}] \}$.
            \end{itemize}
            If $F=F_1$ and $n=1$, then it is the first possibility which occurs.
          \end{itemize}
          If these conditions are satisfied, the decomposition 
          $\calo\simeq  \grp{b,\sigma,F} *_{\grp{b,F}} H_1$ is an orbisocket decomposition, the sockets being $H_1$ and $\grp{b\sigma,F}$.
        \end{lemma}

        \begin{rem}
          To define $\sigma$ and $b$, we identify $G$ with $\pi_1(\Gamma,\tau)$ where $\tau$ is the maximal subtree of 
          the graph of finite groups $\Gamma$ above. Once this is done, this identifies $\grp{\sigma,F}$ and $F_1*_{F_1}$
          with well defined subgroups of $G$ (not just independent conjugacy classes).
        \end{rem}

        \begin{proof}  
          Assume that $\calo$ is a basic orbisocket of type 2.
          The corresponding Fuchsian group $G_0$ can be written as $G_0= \grp{\sigma,F} *_F \grp{b,F}$
          with $F$ normal in $G_0$, $\sigma$ of finite order corresponding to the cone point,
          and with boundary subgroups conjugate to $\grp{b,F}$ and $\grp{\sigma b,F}$.
          By definition, a basic orbisocket has at least one improper socket, say at $b\sigma$. 
          Thus, it has at most one proper socket, and the socket decomposition of $\calo$ can be written as 
          $\calo \simeq G_0 *_{\grp{b,F}} S$.
          Since $S$ is a $\Z$-group, it can be written as an HNN extension $S=F_1*_{F_1}$ where $F_1$ is its maximal finite subgroup.
          It follows that $\calo\simeq \grp{\sigma,F} *_F S\simeq  \grp{\sigma,F} *_F (F_1 *_{F_1})$.
           The peripheral structure consists of the improper socket $\grp{\sigma b}$, maybe together
          with the socket $S=H_1$.
                    The assertions of the lemma follow.

          Assume conversely that $G$ satisfies all the points of the  lemma.  
          Using the trivial splitting
          $H_1 \simeq \grp{b,F} *_{\grp{b,F}} H_1$, one can write 
          $$G\simeq G_0  *_{\grp{b,F}} H_1 \text{ with } G_0=\grp{\sigma,F} *_F \grp{b,F}.$$ 
          Now $F$ is a normal subgroup of $G_0$ since it is normalised by $b$ and
          $\sigma$.  
          Thus, $G_0/F \simeq \grp{\bar{\sigma}} * \langle\bar{b} \rangle \simeq
          \mathbb{Z}/_{r\mathbb{Z}} *\mathbb{Z}$
          is the fundamental group of 
          twice punctured sphere with one cone point of angle $2\pi/r$ 
          where $r$ is the order of the image $\bar\sigma$ of $\sigma$ in $\grp{\sigma,F}/F$,
          and the boundary subgroups are conjugate to $\grp{b}$ and $\grp{b\sigma}$.
          Thus $G_0$ is a bounded Fuchsian group whose boundary subgroups are
          $\langle b ,F \rangle $ and $\langle b\sigma,F \rangle $. 
          The amalgam $G\simeq G_0  *_{\grp{b,F}} H_1$ is therefore an orbisocket decomposition for $G$, with a socket corresponding to $H_1$,
          and an improper socket corresponding to $\grp{b\sigma,F}$. 
          The first socket is improper if and only if $H_1=\grp{b,F}$ \ie if $n=1$ and $F=F_1$ (recall that $F$ is normalised by $b$).
          The hypothesis on $\calp$ guarantees that each peripheral subgroup is a socket group and that
          every improper socket lies in $\calp$. It follows that $(G;\calp)$ is isomorphic to a basic orbisocket of type 2.
        \end{proof}

        \begin{lemma} \label{lem_type3_CNS} 
          Let $(G;\calp)$ be a hyperbolic group with an unmarked $\Z$-peripheral structure.
          
          Then $(G;\calp)$ is isomorphic to a basic orbisocket of type 3 if and only if  
          \begin{itemize}
          \item  $G$ splits as a graph of finite groups
 $$G = \pi_1 \Big(  \xymatrix{ 
              \ar@{{-}{-}{-}}@(ul,dl)_{F_1} \ar@{{-}{-}{-}}@(dl,ul)  \ar@{{*}{-}{*}}[r]_F^{F_1 \quad F_2} &    
              \ar@{{-}{-}{-}}@(ur,dr)^{F_2} \ar@{{-}{-}{-}}@(dr,ur)
            }  \Big)$$  

          \item Consider $H_1=F_1*_{F_1}$ and $H_2=F_2*_{F_2}$ the left and right HNN extensions.
            There exist elements $b_1\in H_1$ and $b_2\in H_2$ normalising $F$,
            with positive translation length $n_1,n_2$ such that
             $\calp$ consists of $[\grp{b_1b_2,F}]$, maybe together with $[H_1]$ or $[H_2]$ (or both);
            if $F_i=F$ and $n_i=1$, then necessarily $[H_i]\in \calp$.
          \end{itemize}
          If these conditions are satisfied, on can get an orbisocket decomposition as
          $\calo\simeq H_1*_\grp{b_1,F_1} \grp{b,\sigma,F}*_\grp{b_2,F_2} H_2$, the sockets being $H_1,H_2$ and $\grp{b_1b_2,F}$.
        \end{lemma}

        \begin{dfn}\label{dfn_type3}
                    We say that $\calo$ is of type $3a$ if $F\neq F_1, F_2$, $3b$ if $F=F_2\subsetneq F_1$,
          and $3c$ if $F=F_1=F_2$.
        \end{dfn}
  
        \begin{rem}
          The translation length $n_i$ is also the index of $\grp{b_i,F_i}$ in $H_i$.
        \end{rem}

          The proof is similar to the previous cases,
          we leave it to the reader.

          \subsubsection{The rigid case}

Recall that the Stallings-Dunwoody deformation space of $G$ 
is the set of $G$-trees with finite edge stabilisers and finite or one-ended vertex groups 
Recall that a graph of groups $\Gamma$ is reduced if for all oriented edge $e$ 
        such that the edge morphism $i_e$ is onto, $e$ is a loop (\ie $o(e)=t(e)$).
        We say that a deformation space is \emph{rigid} if it contains a unique reduced tree.
        If $T$ is this unique reduced tree, we also say that $T$ is rigid.
        We will use Levitt's characterisation of rigid deformation spaces, stated here under an additional hypothesis.

        \begin{thm}[{\cite[Theorem 1]{Lev_rigid}}]\label{thm_levitt_rigid}
          Let $G\actson T$ be a reduced $G$-tree and $\Gamma$ the corresponding graph of groups.
          Assume that no edge group is strictly contained in a conjugate of itself, 
          (this holds if $G$ is hyperbolic, and edge groups are finite or virtually cyclic).

          Then the deformation space of $T$ is rigid if and only if
                  for each vertex $v\in\Gamma$, and for each pair of distinct oriented edges $e_1,e_2$ incident on $v$
        such that $i_{e_1}(G_{e_1})$ is contained in a conjugate of $i_{e_2}(G_{e_2})$, 
        \begin{itemize}
        \item either $e_1=\ol e_2$ (\ie $e_1$ and $e_2$ are the two orientations of a loop)
        \item or $e_2$ is a loop and $G_v=i_{e_2}(G_{e_2})=i_{\ol e_2}(G_{e_2})$, and the only oriented edges incident on $v$ are $e_1,e_2,\ol e_2$.
        \end{itemize}
        \end{thm}

        When a deformation space is rigid, 
        the unique reduced tree it contains is  invariant under automorphisms of $G$ that preserve  
this deformation space.

        \begin{lem}\label{lem_rigid}
          Let $(\calo;\calp_\calo)$ be a basic orbisocket. 
          Then the  absolute (\ie not relative to $\calp_\calo$) Stallings-Dunwoody deformation space of $\calo$
          is rigid if and only if $\calo$ is of type $1$, $2$, or $3a$ as defined in Lemma \ref{lem_class_orbi_1} and Definition \ref{dfn_type3}.
        \end{lem}

        \begin{proof}
          The splittings described in Lemmas \ref{lem_type1_CNS}, \ref{lem_type2_CNS},
          and \ref{lem_type3_CNS} are graphs of finite groups so are in the Stallings-Dunwoody deformation space.
          One checks that under our hypotheses, they are reduced except for type 2 orbisockets
          when $F=F_1$. In this case, by collapsing an edge, we get the following reduced Stallings-Dunwoody decomposition
          $\calo= \grp{\sigma,F} *_{F}$. 
          Applying Levitt's criterion, we immediately see that in each case, the Stallings-Dunwoody deformation space is rigid.

        If $\calo$ is of type 3b or 3c, $F_2=F\subset F_1$, and by collapsing the middle edge of the graph of groups of Lemma \ref{lem_type3_CNS}, 
        one gets the following reduced Stallings-Dunwoody decomposition
        $\stackrel{\xymatrix{ 
            \ar@{{-}{-}{-}}@(ul,dl)_{F_1} \ar@{{-}{-}{-}}@(dl,ul) \bullet 
            \ar@{{-}{-}{-}}@(ur,dr)^{F} \ar@{{-}{-}{-}}@(dr,ur) }}{_{F_1}}$
        which is not rigid by Levitt's result.
        \end{proof}

     \begin{prop}\label{prop_recon_rigid}
       There exists an algorithm which takes as input a hyperbolic group with unmarked peripheral structure $(G;\calp)$       
       and which decides if $(G;\calp)$ is isomorphic to a basic orbisocket
        of type 1, 2, or 3a.
       If it is, it finds a socket decomposition for $(G;\calp)$.
     \end{prop}

     \begin{proof} First, if $(G,\calP)$ is an basic orbisocket, then its peripheral structure consists of at most three non conjugate $\Zmax$-subgroups. 
       This can be checked by Lemma \ref{lem_VC2}, thus we can assume that $\calp$ contains at most $3$ conjugacy classes of peripheral subgroups.

       Using Gerasimov's algorithm, we can compute $\Gamma$ a Stallings-Dunwoody decomposition $G$. 
        We can check whether it is reduced, and make it reduced if it is not.
       Using Levitt's criterion (Lemma \ref{thm_levitt_rigid}), one can check whether it is rigid.
       We can assume it is rigid by Lemma \ref{lem_rigid}.
       By rigidity, any isomorphism $G\ra \calo$ to a basic orbisocket group should map the computed decomposition
       to the unique reduced Stallings-Dunwoody decomposition of $\calo$, described in Lemma \ref{lem_rigid}.
       So if the reduced Stallings-Dunwoody decomposition of $G$ does not have the required form, we are done.
       Furthermore, the topology of the graph of groups tells us which type of orbisocket we should look for.

       Assume first that $\Gamma=A_1*_F A_2$ is an amalgam of finite groups, so that $G$ has to be of type 1.
       If $\calp$ has more that $1$ conjugacy class of peripheral subgroups, 
       if $F$ is not normal in one of the two vertex groups, or if $A_1/F$ or $A_2/F$ is not cyclic,
       we know that $(G;\calp)$ is not an orbisocket of type 1, and not an orbisocket at all. 
       So rewrite $\Gamma=\grp{\sigma_1,F}*_F\grp{\sigma_2,F}$.
       If both $\sigma_1$ and $\sigma_2$ have order 2 modulo $F$, $G$ is virtually cyclic, and is not an orbisocket.
       We have finitely many choices for $\sigma_1$ and $\sigma_2$, and we can check whether
       for some choice, $\calp=\{[\grp{\sigma_1\sigma_2,F}]\}$. If this is not the case, 
       then we know by Lemma \ref{lem_type1_CNS} that $(G;\calp)$ is not an orbisocket
       since by rigidity of the deformation space,  the groups $\grp{\sigma_i,F}$ and $F$ are determined up to simultaneous
       conjugation.  
      If this is the case, then the two assertions of Lemma \ref{lem_type1_CNS} hold, and $G$ is an orbisocket of type 1.

       Assume now that $\Gamma=A*_F (F_1*_{F_1})$ or $\Gamma=A*_{F}$, so that $G$ has to be of type $2$. 
       If $\calp$ has more that $2$ conjugacy class of peripheral subgroups, 
       if $F$ is not normal in $A$, or if $A/F$ is not cyclic,
       $(G;\calp)$ is not an orbisocket. 

       We first consider the case where $\Gamma=A*_F (F_1*_{F_1})$.
       We claim that we can determine the translation length $n$ of the element $b$ (that we still don't know) appearing in Lemma \ref{lem_type2_CNS}.
       Consider $\ol G\simeq \bbZ$ the abelianization of $G$ modulo its torsion.
       If $(G;\calp)$ is an orbisocket of type $2$, then the image in $\ol G$ of one of the peripheral subgroups has index $n$, 
       and the image of the other (if any) has index $1$.
       Since these indices can be computed, this allows us to find $n$. 
       One can list all the finitely many elements $b$ of translation length $n$ in $F_1*_{F_1}$,  
       and list all possibilities for the choice of $\sigma$.
       Then for each such choice, we check whether $\calp$ is conjugate to $\{[\grp{b\sigma,F}]\}$ or to $\{[\grp{b\sigma,F}],[\grp{b,F}]\}$.
       Then by Lemma \ref{lem_type2_CNS}, $(G;\calp)$ is a socket of type 2 if and only if there was one successful choice.
       
       When $\Gamma=A*_F$, we will compute all possible $\Hat\Gamma=A*_F (F_1*_{F_1})$ that give $\Gamma$ when made reduced. 
       When this is done, we proceed as above for each possible $\Hat \Gamma$.
       First, if $\Hat \Gamma$ exists, then $F_1=F$ since otherwise $\Hat\Gamma$ would be reduced. 
       It follows that $F$ should be normal in $G$.
       This fact can be checked so we can assume that $F \normal G$.
       Let $t_0$ be a stable letter of the HNN extension $\Gamma$. 
       Given another stable letter $t=at_0 $ for some $a\in A$,
       we can write $G=A*_F \grp{F,t}$, hence produce the blow up $\Hat \Gamma_t=A*_F (F*_F)$ 
       where $t$ is a stable letter of the HNN extension $F*_F$. 
       One can easily check that every possible blowup $\Hat \Gamma$ of $\Gamma$ coincides with such a $\Gamma_t$,
       so all possible blowups can be listed.

       Assume finally that 
$\Gamma=  \xymatrix{ 
         \ar@{{-}{-}{-}}@(ul,dl)_{F_1} \ar@{{-}{-}{-}}@(dl,ul)  \ar@{{*}{-}{*}}[r]_F^{F_1 \quad F_2} &    
         \ar@{{-}{-}{-}}@(ur,dr)^{F_2} \ar@{{-}{-}{-}}@(dr,ur)
       } $  
 with $F\subsetneq F_1,F_2$,  so that $(G,\calp)$ is candidate to be a type 3a orbisocket.
       Let $\ol G\simeq \bbZ^2$ the abelianization of $G$ modulo its torsion.
       Consider $h_1,h_2$ two elements of translation length $1$ in $F_1*_{F_1}$ and $F_2*_{F_2}$ respectively.
       Their image $\ol h_1,\ol h_2$ in $\ol G$ is a basis of $\ol G$, which does not depend on the choice of $h_i$ up to sign.
       By rigidity of the Stallings-Dunwoody deformation space, 
       $F_1*_{F_1}$ and $F_2*_{F_2}$ don't depend on $\Gamma$ (up to simultaneous conjugation), 
       so  $\ol h_i$ does not depend on $\Gamma$ up to sign.
       On the other hand, Lemma \ref{lem_type3_CNS} says that if $(G;\calp)$ is an orbisocket of type 3,
       then $\calp$ consists of $[\grp{b_1b_2}]$, maybe together with $[F_1*_{F_1}]$ and/or $[F_2*_{F_2}]$.
       The groups conjugate to $[F_i*_{F_i}]$ project in $\ol G$ to $\grp{\ol h_i}$,
       and $\grp{b_1b_2}$ projects to $\grp{\ol h_1^{\pm n_1}\ol h_2^{\pm n_2}}$
       where $n_i$ is the translation length of $b_i$ in $F_i*_{F_i}$.
       This means that from the peripheral structure $\calp$ of $G$, we can read $n_1$ and $n_2$.
       Now, there are only finitely many possible candidates for $b_1$ and $b_2$, namely
       $b_i=f h_i^{\pm n_i}$ for some $f\in F_i$. For each possible choice of $b_1$ and $b_2$, there remains to check whether
       the peripheral structure has the form required by \ref{lem_type3_CNS}.
     \end{proof}

\subsubsection{ The semi-rigid case}

\begin{prop}\label{prop_recon_3b}
  There exists an algorithm which takes as input a hyperbolic group with unmarked peripheral structure $(G;\calp)$,
  and which decides if $(G;\calp)$ is isomorphic to a basic orbisocket of type $3b$ and in this case, gives a socket decomposition
of $(G;\calp)$.
\end{prop}

\begin{proof}
A basic orbisocket of type $3b$ has a decomposition as a graph of finite groups $\Gamma$ of the form
       $$\calo = \pi_1 \Big(  \xymatrix{ 
         \ar@{{-}{-}{-}}@(ul,dl)_{F_1} \ar@{{-}{-}{-}}@(dl,ul)  \ar@{{*}{-}{*}}[r]_F^{F_1 \quad F} &    
         \ar@{{-}{-}{-}}@(ur,dr)^{F} \ar@{{-}{-}{-}}@(dr,ur)
       }  \Big)$$  
with $F\subsetneq F_1$. For lack of a better name, we call such a decomposition a \emph{nice}
 decomposition.
 
We note that $\calo$ has infinitely many distinct such nice decompositions (distinct as actions on trees).
Moreover, the fact that $\calp$ satisfies the characterisations for being orbisockets of Lemma \ref{lem_type3_CNS}
depends on the choice of the nice decomposition.

Let $T$ be the Bass-Serre tree of $\Gamma$, and $T_0$ be obtained by collapsing the orbit of the left and middle edges.
It is dual to the HNN extension $\calo=H_1\xymatrix{\bullet\ar@{{-}{-}{-}}@(ur,dr)^{F}}$.
We claim that $T_0$ does not depend of the nice decomposition used to define it.
Indeed, $F_1$ is the unique maximal finite subgroup up to conjugacy,
so $H_1$ is well defined up to conjugacy as the normaliser of $F_1$.
Since the elliptic subgroups of $T_0$ are the conjugate of the subgroups of $H_1$, the deformation space of $T_0$
is independent of choices. By Levitt's rigidity criterion, $T_0$ is the unique reduced tree in its deformation space, which proves the claim.  

\begin{lem}
  One can decide if a hyperbolic group $G$ has a nice decomposition, and produce one if it does.
\end{lem}

\begin{proof}
 If $G$ has a nice decomposition, then it has a reduced Stallings-Dunwoody
  decomposition $\Gamma_{red}$ of the form $\stackrel{\xymatrix{ \ar@{{-}{-}{-}}@(ul,dl)_{F_1} \ar@{{-}{-}{-}}@(dl,ul) \bullet
      \ar@{{-}{-}{-}}@(ur,dr)^{F} \ar@{{-}{-}{-}}@(dr,ur) }}{_{F_1}}$ with $F\subsetneq F_1$, 
obtained by collapsing the middle edge. We claim that all
  reduced Stallings-Dunwoody decompositions of $G$ are of this form. This is because any two such decompositions are related by slide moves
  (see section \ref{sec_deformation}), and the only possible slide moves (which consist in sliding the edge labelled $F$ along the
  edge labelled $F_1$), give a new decomposition of the same form.  Using Gerasimov's algorithm \cite[Theorem 1.3]{DaGr_detecting}, one can compute a
  reduced Stallings-Dunwoody decomposition of $G$, and check whether it has the same form as $\Gamma_{red}$.  

But this is not enough to conclude that $G$ has a nice decomposition.  

The argument above shows that the tree $T_0$ obtained by collapsing the edge
  labelled $F_1$ in $\Gamma_{red}$ does not depend on choices.  The decomposition $T_0/G$ can be computed, and has the form
  $H_1\xymatrix{\bullet\ar@{{-}{-}{-}}@(ur,dr)^{F}}$ with $H_1$ a $\Zmax$-subgroup with maximal normal finite subgroup $F_1$.
   We claim that $G$ has a nice decomposition if and only if the two images of $F$ in $H_1$ are conjugate
  in $H_1$.  Since this condition is easy to check, one can decide if $G$ has a nice decomposition (and produce one if it does).  
  The claim is
  clear if $T_0$ comes from a nice decomposition.  Conversely, if the two images of $F$ in $H_1$ are conjugate in $H_1$, one can
  assume that they are equal up to post-conjugating the edge morphism, which allows us to blow up $T_0$ into $ \xymatrix{
    \ar@{{*}{-}{*}}[r]_F^{H_1 \quad F} & \ar@{{-}{-}{-}}@(ur,dr)^{F} \ar@{{-}{-}{-}}@(dr,ur) } $ and then into a nice
  decomposition.
\end{proof}

Now we can assume that $G$ has a nice decomposition $\Gamma$.
Denote by $T$ its Bass-Serre tree.

Say that $(h_1,h_2)$ is a \emph{basis} of $\Gamma$ if
the translation length of $h_1$ and $h_2$ is $1$, and 
there exists an edge $e=v_1v_2$ in the preimage of the middle edge of $\Gamma$
such that the axis of $h_i$ contains $v_i$
(in other words, $h_1,h_2$ are simultaneously stable letters of the two HNN extensions of $\Gamma$).

Such a basis $(h_1,h_2)$ defines $H_i$ for $i=1,2$ as the $\Zmax$-subgroup containing $h_i$, $F_1$ (resp. $F$) as the maximal finite
normal subgroup of $H_1$ (resp. $H_2$). In particular, $H_1=\grp{h_1,F_1}$ and $H_2=\grp{h_2,F}$.
 
\begin{lem}\label{lem_basis}
  Consider $(h_1,h_2)$ a basis coming from a nice decomposition of $G$ as above.

Then $(h'_1,h'_2)$ is another basis coming from some nice decomposition $\Gamma'$
if and only if 
$(h'_1,h'_2)^g=(h_1^{\pm1}f,h_2^{\pm1}k_1)$
for some $g\in G$, $k_1\in N_{H_1}(F)$ and $f\in F_1$  where $N_{H_1}(F)$ denotes the normaliser of $F$ in $H_1$.
\end{lem}

\begin{proof}
  Assume that $(h'_1,h'_2)$ is another basis. Denote by $H'_1,H'_2,F'_1,F'$ the subgroups of $G$ defined analogously using $(h'_1,h'_2)$
(namely, $H'_i$ is the $\Zmax$-subgroup containing $h'_i$, etc.).
Up to conjugating $(h'_1,h'_2)$  by some $g\in G$, we may assume that $H'_1=H_1$, and in particular $F'_1=F_1$.
Since $h_1$ and $h'_1$ give two generators of the cyclic group $H_1/F_1$, 
$h'_1=h_1^{\pm1}f$ for some $f\in F_1$.
The elements $h'_2$ and $h_2$ both have translation length $1$ in $T_0$, and their axis in $T_0$ contain the vertex $u_1\in T_0$ fixed by $H_1$.
Conjugating $(h'_1,h'_2)$ by some element of $H_1$, we may assume that the axes of $h_2$ and $h'_2$ share an edge incident on $u_1$
so in particular, $F=F'$.
It follows that $h'_2=h_2^{\pm 1}k_1$ for some $k_1\in H_1$. 
Since both $h_2$ and $h'_2$ normalise $F$, so does $k_1$. The conclusion follows.

Conversely, assume $(h'_1,h'_2)^g=(h_1^{\pm1}f,h_2^{\pm1}k_1)$ with $f,k_1,g$ as above.
A conjugate of a basis being a basis, we can assume $g=1$.
Since $h_2$ is a stable letter of the HNN extension $T_0/G$, so is $h'_2$.
Since $h_2$ and $k_1$ normalise $F$, so does $h'_2$.
It follows that one can blow up the HNN extension $G\simeq H_1\xymatrix{\bullet\ar@{{-}{-}{-}}@(ur,dr)^{F}}$ corresponding to $T_0$
into 
$ \xymatrix{ 
         \ar@{{*}{-}{*}}[r]_F^{H_1 \quad F} &    
         \ar@{{-}{-}{-}}@(ur,dr)^{F} \ar@{{-}{-}{-}}@(dr,ur)
       } $  
with $h'_2$ a stable letter of the HNN extension on the right.
This last decomposition can be blown up into a nice decomposition
$\Gamma'=  \xymatrix{ 
         \ar@{{-}{-}{-}}@(ul,dl)_{F_1} \ar@{{-}{-}{-}}@(dl,ul)  \ar@{{*}{-}{*}}[r]_F^{F_1 \quad F} &    
         \ar@{{-}{-}{-}}@(ur,dr)^{F} \ar@{{-}{-}{-}}@(dr,ur)
       } $  
of which $(h'_1,h'_2)$ is a basis.
\end{proof}

Recall that we have computed $\Gamma$ a nice decomposition of $(G,\calp)$ and $(h_1,h_2)$ a basis.
To simplify the discussion, assume that $\calp$ consists of exactly three conjugacy classes of subgroups $\{[P_0],[P_1],[P_2]\}$.
Given $(h'_1,h'_2)$ a basis coming from another nice decomposition,
denote by $H'_1,F'_1,H'_2,F'$ the subgroups defined by $(h'_1,h'_2)$ as before the statement of Lemma \ref{lem_basis}.
Writing $(h'_1,h'_2)^g=(h_1^{\pm1}f,h_2^{\pm1}k_1)$, we get that 
$H_1'^g=H_1$, $F_1'^g=F_1$ and $F'^g=F$ (but $H_2'^g\neq H_2$ in general).

Then one can rephrase Lemma \ref{lem_type3_CNS} as follows:
$(G,\calp)$ is a basic orbisocket of type $3b$ if and only if 
there exists a basis $(h'_1,h'_2)$ associated to some nice decomposition $\Gamma'$, 
$b_1\in H'_1$ normalising $F'$,
$b_2\in H'_2$ normalising $F'$ such that
$\calp=\{[\grp{F',b_1b_2}],[\grp{h_1',F'_1}],[\grp{h'_2,F'}]\}$.
Up to re-indexing the $P_i$'s, we may as well require that $[P_0]=[\grp{F',b_1b_2}]$,
$[P_1]=[\grp{h_1',F'_1}]$, $[P_2]=[\grp{h'_2,F'}]$.

Writing $b_2=h_2'^{n_2}f'$ for some $f'\in F'$, we claim the we can determine $n_2$ up to sign.
Indeed, if $(\calo,\calp_\calo)$ is a basic orbisocket of type 3b, write it as a the fundamental group of
a nice decomposition $\Gamma'$. Consider $\ol\calo\simeq\bbZ$  the abelianization of $\calo$ modulo its torsion and
$\ol H_1'\subset \ol\calo$ the image of $H_1'$. 
Since $H'_1$ is conjugate to $H_1$, one can compute the group $\ol\calo/\ol H'_1$ from $\Gamma$.
If $[P_0]=[\grp{b_1b_2,F'}]$, then 
$|n_2|$ is the index of the image of $P_0$ in $\ol\calo/\ol H'_1$, and can be computed.

To decide if $(G,\calp)$ is indeed a basic orbisocket, we have to decide the existence of a basis $(h'_1,h'_2)$, and elements $b_1,b_2$ as above.
We view this as disjunction of systems of equations as follows:
We view the relation $(h'_1,h'_2)^g=(h_1^{\pm1}f,h_2^{\pm1}k_1)$
as a finite family of systems of two equations relating some variables $h'_1,h'_2,k_1,g$ in $G$, parametrised by $f\in F$, and $h_1,h_2$ being constants.
We include additional disjunctions of systems of equations
saying that $k_1$ normalises $F_1$ (\ie lies in $H_1$), 
and another disjunction of systems of equations saying that $k_1$ normalises $F$.

Now introduce two new variables $b_1,b_2$, subject to
a disjunction of systems of equations saying that $b_1^g$ normalises $F_1$ (so that $b_1\in H_1'$),
and a disjunction of equations $b_2=h'_2{}^{\pm n_2}f_2^g$ parametrised by $f_2\in F$, $|n_2|$ having been already computed.
We add disjunctions of systems of equations saying that $b_1^g$ and $b_2^g$ normalise $F$.
Write $P_i=\grp{a_i,C_i}$ where $C_i\normal P_i$ is the maximal finite normal subgroup (one can compute such $a_i,C_i$).
The condition $[P_0]=[\grp{b_1b_2,F'}]$ (resp. $[P_1]=[\grp{h_1',F'_1}]$, $[P_2]=[\grp{h'_2,F'}]$)
 translates into a disjunction of systems of equations saying that $C_0^{g_0}=F^{g\m}$ (resp. $C_1^{g_1}=F_1^{g\m}$, $C_2^{g_2}=F^{g\m}$)
and a disjunction of equations $(b_1b_2)^{g_0}=a_0^{\pm 1} c_0$, (resp. $h'_1{}^{g_1}= a_{1}^{\pm 1} c_1$, $h'_2{}^{g_2}= a_{2}^{\pm 1} c_2$) indexed by $c_0\in C_0$
(resp. $c_1\in C_{1}$, $c_2\in C_2$)
where $g_0$ (resp. $g_1,g_2$) is a new variable.
This way, we get a disjunction of systems of equations, and the existence of a solution in $G$ is equivalent to the fact that
$(G;\calp)$ is an orbisocket of type $3b$.
Since one can solve systems of equations in virtually free groups by \cite{DG1}, one can decide whether $\calo$ is an orbisocket
of type $3b$.
By proposition \ref{lem_type3_CNS}, one can get an orbisocket decomposition of $G$ from a solution of this system of equations:
$G\simeq H_1*_{\grp{b_1,F}} \grp{b_1,b_2,F} *_{\grp{b_2,F}} H_2$.
\end{proof}

\subsubsection{The flexible case}

\begin{prop}\label{prop_recon3c}
  There exists an algorithm which takes as input a hyperbolic group with unmarked peripheral structure $(G;\calp)$,
  and which decides if $(G;\calp)$ is isomorphic to a basic orbisocket of type $3c$ and in this case, gives a socket decomposition
of $(G;\calp)$.
\end{prop}

\begin{proof}
If $\calo$ is an orbisocket of type $3c$, 
 any of its reduced Stallings-Dunwoody decompositions has the form                        $\stackrel{\xymatrix{ 
                            \ar@{{-}{-}{-}}@(ul,dl)_{F}\ar@{{-}{-}{-}}@(dl,ul)  \bullet 
                            \ar@{{-}{-}{-}}@(ur,dr)^{F}\ar@{{-}{-}{-}}@(dr,ur)  }}{_{F}}$,
and $G$ is an extension of the free group $\bbF_2$ by the finite normal subgroup $F$.

  Using Gerasimov's algorithm, compute a reduced Stallings-Dunwoody decomposition of $G$.
  We can check whether this decomposition has the right form, so we can assume this is the case.

  By lemma \ref{lem_type3_CNS}, $(G;\calp)$ is an orbisocket of type $3c$ if and only if there
  exists $h_1,h_2\in G$ whose image in $\bbF_2$ is a free basis, 
  and $b_1\in \grp{h_1,F}$, $b_2\in \grp{h_2,F}$, such that $\calp=\{[\grp{b_1b_2,F}]\}$
maybe together with $[\grp{h_1,F}]$ and/or $[\grp{h_2,F}]$.
To simplify, assume that $\calp=\{[P_0],[P_1],[P_2]\}$ has 3 elements.
One can first check that each $P_i$ contains $F$.
Choose $a_i\in P_i$ with $P_i=\grp{a_i,F}$

 From a reduced Stallings-Dunwoody decomposition, compute $h_1,h_2\in G$ (stable letters of the two HNN extensions) whose image in $\bbF_2$ is a free basis.
Now by Dehn-Magnus-Nielsen theorem  about bases of $\bbF_2$, 
the pair $(h'_1,h'_2)$ projects to a free basis of $\bbF_2$ if and only if $[h'_1,h'_2]^g=([h_1,h_2]^{\pm{1}})f$ for some $g\in G$ and some $f\in F$.
Moreover, given such $h'_1,h'_2\in G$,  $b_i\in \grp{h'_i,F}$ if and only if $[b_i,h'_i]=f_i$ for some $f_i\in F$.
Finally, for such $b_1,b_2$ the peripheral structure has the right form if and only if 
$a_0^{g_0}=(b_1b_2)^{\pm 1}f'_0$, $a_1^{g_1}=h_1^{\pm 1} f'_1$, and $a_2^{g_2}=h_2^{\pm 1} f'_2$, for some $g_0,g_1,g_2\in G$
and $f'_0,f'_1,f'_2\in F$. 
Thus we get a disjunction of systems of equations with unknowns $h_1,h_2,g,b_1,b_2,g_0,g_1,g_2$,
parametrised by the finitely many values of the parameters $f,f_1,f_2,f'_0,f'_1,f'_2$.
It has a solution if and only if $(G,\calp)$ is an orbisocket of type (3c).
By \cite{DG1}, this can be decided algorithmically.
\end{proof}

\subsection{Isomorphism problem for orbisockets} 

In this section, we solve the isomorphism problem for (non-basic) orbisockets,
given orbisocket decompositions as input.
We will reduce the problem to the case of bounded Fuchsian groups and apply Proposition \ref{prop_IP_BFG} which solves
the isomorphism problem for bounded Fuchsian groups.

\begin{thm}\label{theo;IP_orbi}

The extended isomorphism problem is solvable for orbisockets with marked peripheral structure.  

More precisely,  there exists an algorithm which takes as input two orbisocket groups $(\calo_1;\calm_1)$, $(\calo_2;\calm_2)$,
with marked peripheral structure, 
an orbisocket decomposition of both, and decides whether 
$(\calo_1;\calm_1)$ and $(\calo_2;\calm_2)$ are isomorphic (as groups with marked peripheral structure).

Moreover, there is an algorithm which takes as input an orbisocket group $(\calo;\calm)$,
with marked peripheral structure, and an orbisocket decomposition of it, and computes
generators for  $\Outm(\calo;\calm)$ (the group of outer automorphisms preserving the marked peripheral structure).
\end{thm}

\begin{proof}
  Let us denote by $\Gamma_1, \Gamma_2$ the graph of groups encoding the socket decompositions. 
  One can algorithmically remove M\"obius sockets from $\Gamma_i$ as in section \ref{subsec_orbi}.
  By Proposition \ref{prop_JSJ_orbi}, $\Gamma_i$ is a $\Z$-JSJ decomposition of $G_i$ relative to $\calm_i$.
  Moreover, any orbisocket decomposition is its own tree of cylinders, so
  any isomorphism between $(\calo_i;\calm_i)$, should be induced by an isomorphism of the graph of groups $\Gamma_i$.
  By Corollary \ref{coro_IP_gog},  there remains to solve
  the isomorphism problem for  vertex groups with a marking of their peripheral structure.
  This follows from Lemma \ref{lem_autZ} for virtually cyclic group, and from Proposition \ref{prop_IP_BFG}
  for hanging bounded Fuchsian groups.

  For the second assertion, we use Corollary \ref{coro_gene_Aut_gog}.
  Since generating sets of centralisers are computable,
  it is enough to solve the  extended
  isomorphism problem for vertex groups with a marking of their peripheral structures.
  This also follows from Lemma \ref{lem_autZ} and Proposition \ref{prop_IP_BFG}.
\end{proof}

\section{Computation of the $\Zmax$ JSJ decomposition, and isomorphism problem for one-ended hyperbolic groups}\label{sec_1end}

   \subsection{Computation of a maximal $\Zmax$-splitting}

The following proposition computes decomposition of $G$ over $\Zmax$-subgroups which is maximal
for domination.  This is not yet a JSJ decomposition:
 for instance, if $G$ is an orientable surface group, one would get a pants decomposition.

     \begin{prop}    \label{prop_max_ess} 
       There is an algorithm that takes as input a one-ended hyperbolic group $G$,
       and which computes a maximal $\Zmax$-splitting of $G$.
     \end{prop}

\begin{remark} \label{rk_max}
A splitting of $G$  
over $\Zmax$ edge groups is maximal if and only if for each vertex group $G_v$ with the induced peripheral structure $\calp_v$,
$(G_v;\calp_v)$ has no non-trivial $\Zmax$-splitting. 
Indeed, denote by $T$ a maximal $\Zmax$-tree,
and assume that a vertex group $G_v$ has a splitting relative to $\calp_v$ whose edge groups are $\Zmax$ in $G_v$.
Then these edge groups are $\Zmax$ in $G$ by Lemma \ref{lem_inter_Zmax},
The converse is  Lemma \ref{lem_sup_Zmax}.
\end{remark}

   \begin{proof} 
     Starting with the trivial decomposition of $(G;\calp)$, assume that we have already constructed
     some $\Zmax$-decomposition $\Gamma$ of $(G;\calp)$.
     One can check if this decomposition can be refined at a vertex $v$ by applying
     Corollary \ref{cor_decision_rigid} to the vertex group $G_v$ with its peripheral structure $\calp_v$ induced by 
     the incident edge groups. Doing this for each vertex of $\Gamma$, one can check whether $\Gamma$ is maximal, in which case we are done.

     If $\Gamma$ is not maximal, one can produce a $\Zmax$-splitting of some $(G_v;\calp_v)$ by Proposition \ref{enumerate_ess_splittings}, 
     and use it to refine $\Gamma$.
      This refinement a $\Zmax$-splitting by Lemma \ref{lem_inter_Zmax}.

     Since the obtained splittings are reduced in the sense of Bestvina-Feighn \cite{BF_bounding}, this refinement process has to stop:
     a maximal splitting is then produced.
   \end{proof}

     \subsection{Computation of the  JSJ tree $\TZmax$}

     \begin{prop}\label{prop_calcul_TZmax}
       Let $G$ be a one-ended hyperbolic group.
       Then one can compute  $T_\Zmax(G)$ (as defined in section \ref{sec_TZmax}).
       
       Moreover, one can determine which vertices of this decomposition are flexible,
       and find their orbisocket decomposition.
     \end{prop}

     \begin{proof}
       By Lemma \ref{lem_calcul_Tc}, one can compute the collapsed tree of cylinders of any given $\Z$-tree.
       Therefore, we need only to produce some splitting in the deformation space of $\TZmax$.

       Start with a maximal $\Zmax$-tree $T$ as produced by Proposition \ref{prop_max_ess}.
       One can compute $T_c^*$ the collapsed tree of cylinders of $T$.
       By Lemma \ref{lem_same_D}, $T_c^*$ is a $\Zmax$-splitting in the same deformation space as $T$, hence is also maximal,
       so we can assume that $T=T_c^*$.
     
       Note that $T$ dominates $\TZmax$: the blowup $S$ of $\TZmax$ relative to $T$ (as defined at the beginning of Section \ref{sec_JSJ})
       is a $\Zmax$-tree  by Lemma \ref{lem_sup_Zmax}. By maximality of $T$, $S$ lies in the same deformation space as $T$,
       and since $S$ dominates $\TZmax$, so does $T$.

       Denote by $\Gamma=T/G$ the quotient graph of groups.
       For each vertex $v\in \Gamma$, let $\calp_v$ be the unmarked peripheral structure $\Gamma_v$ induced by $\Gamma$.
       Let $V_B\subset \Gamma$ be the set of vertices $v$ such that $(\G_v;\calp_v)$ is a basic orbisocket.
       Since $T$ is a maximal $\Zmax$-splitting, for each $v$, $(\G_v;\calp_v)$ have no non-trivial $\Zmax$-splitting.
       By Definition  \ref{dfn_basic}, it follows that an orbisocket $(\G_v;\calp_v)$ of $\Gamma$ is basic if and only if 
       it has at least one improper socket.
       The  set $V_B$ is computable by Theorem \ref{thm_recon_orbisocket}, and one can compute 
       its socket decomposition of the corresponding orbisockets. In particular, one can decide which elements of $\calp_v$ are proper sockets.
       
\begin{figure}[htbp]
  \centering
  \includegraphics{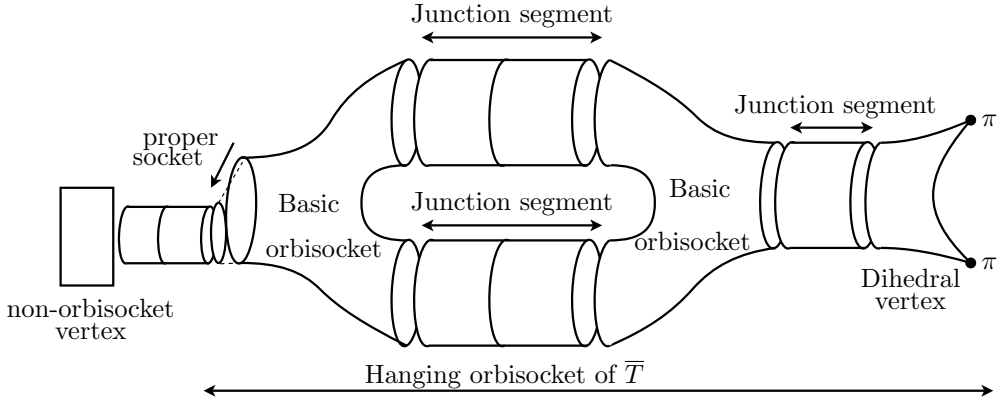}
  \caption{Junction segments}
  \label{fig_junction}
\end{figure}

Note that a disk with two cone points of angle $\pi$ has infinite dihedral fundamental group $D_\infty$, and its boundary subgroup $B$ is
is a cyclic subgroup of index $2$ (the unique $\Zmax$-subgroup in $D_\infty$). 
Corresponding to this situation, we call a pair $(D;B)$ a \emph{dihedral orbisocket}
if $D$ is virtually cyclic of dihedral type, and $B<D$ is its unique $\Z$-subgroup of index $2$. 
A \emph{hanging} dihedral orbisocket $v$ of $\Gamma$ is a terminal vertex 
whose peripheral structure induced by $\Gamma$ makes it a dihedral orbisocket.

A \emph{junction segment} of $\Gamma$ is (see Figure \ref{fig_junction}):
       \begin{itemize}
       \item either a segment $[u,v]=e_1\cup e_2$ of length $2$ (maybe with $u=v$) 
such that $u,v\in V_B$ are hanging basic orbisockets in which $\G_{e_1}$ and $\G_{e_2}$      are improper sockets, such that
the midpoint $w$ of $[u,v]$ has valence $2$ in $\Gamma$ and $\G_{e_1}=\G_w$ and $\G_{e_2}=\G_w$. 
       \item or an edge $e=[u,v]$ where $u\in V_B$ is a hanging basic orbisocket in which $\G_e$ is an improper socket,
         and $v$ is a hanging dihedral orbisocket.
       \end{itemize}

       Since for each $v\in V_B$ we know which edges correspond to proper sockets, we can compute effectively the set of junction segments of $\Gamma$.

       Let $\ol T$ be the $G$-tree obtained from $T$ by collapsing all edges whose image in $\Gamma$ lie in a junction segment.
       One can compute the corresponding  graph of groups decomposition $\ol \G$ of $G$.
       We will prove that $\ol T$ lies in the same deformation space as $\TZmax$.
       The following lemma says that new vertices of  $\ol \G$ are hanging orbisockets.

       \begin{lem}\label{lem_collage}
         Assume that $(G;\{[S_1],\dots,[S_p]\})$ and $(G',\{[S'_1],\dots,[S'_q]\})$ are two orbisockets (maybe not basic)
         in which $S_1$ and $S'_1$ are improper sockets.
         We allow one of them (but not both) to be a dihedral orbisocket as above.

         Assume that $S_1$ is isomorphic to $S'_1$.
         Then for any isomorphism between $S_1$ and $S'_1$,
         $$(G*_{S_1=S'_1} G';\{[S_2],\dots, [S_p],[S'_2],\dots,[S'_q]\})$$ is itself an orbisocket,
          and one can compute an orbisocket decomposition from orbisocket decompositions of $G$ and $G'$.

         Similarly, if $S_1$ and $S_2$ are two improper sockets of $G$ which are isomorphic, then the HNN extension
         $$(G*_{S_1=S_2};\{[S_3],\dots, [S_p]\})$$ is an orbisocket,
          and one can compute an orbisocket decomposition from orbisocket decompositions of $G$ and $G'$.

       \end{lem}

       \begin{proof}
         Consider the socket decompositions 
         $G=G_0(*_{B_i} S_i)_{i=1}^p$  and
         $G'=G'_0(*_{B'_i} S'_i)_{i=1}^q$ of $G,G'$.
         Let $F\normal G_0$ (resp.\ $F'\normal G'_0$) be the maximal normal finite subgroup of $G_0$
         (resp.\ of $G'_0$).
         Since the socket $S_1$ is improper, 
         $S_1$ is an extension of $\bbZ$ by $F$, and
         $F$ is the maximal finite subgroup of $S_1$.
         Similarly, $F'$ is the maximal finite subgroup of $S'_1$,
         so $F$ is necessarily mapped to $F'$ under the identification of $S_1$ with $S'_1$.
         Thus, $F=F'$ is normal in the group $H_0=G_0*_{S_1=S'_1} G'_0$, and $H_0$ is a bounded Fuchsian group
         with peripheral structure $\{[B_2],\dots, [B_p],[B'_2],\dots,[B'_q]\})$.
         It follows that $(G*_{S_1=S'_1} G',\{[S_2],\dots, [S_p],[S'_2],\dots,[S'_q]\})$ is an orbisocket.
         
         The same verification can be done for HNN extensions, or when one of the sockets is dihedral.
       \end{proof}

Applied inductively, Lemma \ref{lem_collage} says that the vertex groups of $\ol T$ which are not elliptic in $T$ are orbisockets.
Since $T$ dominates $\TZmax$, and since by Lemma \ref{lem_orbi_are_ell}, any hanging orbisocket is elliptic in $\TZmax$, 
it follows that $\ol T$ dominates $\TZmax$.

There remains to prove that  $T_\Zmax$ dominates $\ol T$.
Since flexible vertices of $T_\Zmax$ are hanging orbisockets, other vertices of $T_{\Zmax}$ fix a point in $T$,
so we need only to prove that hanging orbisockets of $\TZmax$ fix a point in $\ol T$.
Let $v$ be a hanging orbisocket of $T_{\Zmax}$ such that $\Gamma_v$ not elliptic in $T$, and 
 $Y_v\subset T$ be the minimal $\G_v$-invariant subtree.

By Lemma \ref{lem_courbes_orbisocket},
the action $\G_v\actson Y_v$ is dual to a finite set of non-peripheral disjoint  $2$-sided simple closed curves $c_1,\dots c_n$
in the orbifold $\Sigma_v$ underlying the orbisocket $\G_v$.
In particular, for each edge $e$ of $T_{\Zmax}$ incident on $v$, $G_e$ fixes a unique point $p_e\in Y_v$.
Let $\Hat T$ be the $G$-tree obtained by blowing up each orbisocket vertex $v$ of $T_{\Zmax}$ into $Y_v$, and by 
attaching the edges $e$ incident on $v$ to $p_e$.
We denote by $\Hat Y_v$ the copy of $Y_v$ in $\Hat T$.
Since for every non-orbisocket vertex $w$ of $\TZmax$, $\Gamma_w$ fixes a point in $T$, 
$\Hat T$ dominates $T$. 
Since $\Hat T$ is a $\Zmax$-splitting (Lemma \ref{lem_sup_Zmax}), and $T$ is a maximal one,
if follows that $T$ and $\Hat T $ are in the same deformation space.
Thus, $T_c^*=\Hat T_c^*$, and since $T=T_c^*$, $T=(\Hat T)_c^*$.

Now we will describe $T$ using a description of cylinders of $\Hat T$.
Since $\Gamma_v\actson \Hat Y_v$ is dual to a family of non-peripheral curves in $\Sigma_v$,
no edge stabiliser of $\Gamma_v\actson \Hat Y_v$  is commensurable with a peripheral subgroup of $\Gamma_v$.
It follows that each cylinder of $\Hat T$ is either contained in $\Hat Y_v$, or has no edge in $\Hat Y_v$.

We describe the tree of cylinders of $\Hat Y_v$.
One can assume that no two curves $c_i$ bound an annulus since removing one of them does not change the tree of cylinders $\Hat Y_v$.
Each connected component $U$ of $\Sigma_v\setminus c_1\cup\dots\cup c_n$ is a sub-orbifold.
Let $H$ be the preimage in $\Gamma_v$ of $\pi_1(U)$, a bounded
Fuchsian group if $H$ is not virtually cyclic.
If $\pi_1(U)$ is cyclic, then $U$ is a M\"obius band since there are no annuli,
which is excluded by the fact that edge stabilisers are $\Zmax$.
If $\pi_1(U)$ is infinite dihedral then $U$ is a disc with
two cone points of angle $\pi$. 
Since $\Sigma$ has no mirror, these are the only possibilities.
Finally, one easily checks that the tree of cylinders $(Y_v)_c$ 
is obtained from $Y_v$ by barycentric subdivision of all edges not incident on such a dihedral component $U$.

Using $(Y_v)_c$ instead of $Y_v$ to blow up $\TZmax$ into $\Hat T$, we get a tree in the same deformation space,
so we still have $T=(\Hat T)_c^*$.
On the other hand,  one easily checks that $\Hat T$ is its own collapsed tree of cylinders (see \cite[prop.\ 5.7]{GL4} for instance),
\ie  $\Hat T=(\Hat T)_c^*=T$.

Now it is clear from the description above that all vertices of $\Hat Y_v$ are hanging orbisockets,
dihedral, or subdivision vertices, and that all edges of $\Hat Y_v$ are junction edges.
It follows that the map $T\ra \ol T$ collapses $\Hat Y_v$ to a point, so $\G_v$ is elliptic in $\ol T$, and $\TZmax$ dominates $T$.

This proves that one can compute $\TZmax$. Since each flexible vertex is coming from the collapse of junction edges in $\ol T$,
we can compute an orbisocket decomposition by Lemma \ref{lem_collage} from orbisocket decompositions of basic orbisockets.
\end{proof}

\subsection{Isomorphism problem for one-ended hyperbolic groups}

\begin{theo} \label{theo;IP_one_ended}
  The extended isomorphism problem is solvable for the class of one-ended hyperbolic groups:

  there is an explicit algorithm that, given two one-ended hyperbolic groups, 
  terminates and indicates whether they are isomorphic, and which computes  generating systems of their automorphism groups
\end{theo}

\begin{proof}
  Let $G,G'$ be two one-ended hyperbolic groups.  Compute $\Gamma,\Gamma'$ the graph of groups of their
  $T_\Zmax$ splittings, determine its flexible vertices, 
   and compute an orbisocket decomposition of all flexible vertex groups (Prop. \ref{prop_calcul_TZmax}).  
  Since $T_\Zmax$ is a canonical splitting, any
  isomorphism $G\ra G'$ is induced by an isomorphism between the graphs of groups $\Gamma,\Gamma'$ (Prop. \ref{prop_TZmax_is_canonical}).

  By Corollary \ref{coro_IP_gog}, deciding whether $\Gamma$ and $\Gamma'$ are isomorphic reduces
  to solve the isomorphism problem for vertex groups with a marking of their peripheral structure defined by
  their edge groups.  This is done by Th. \ref{thm_IP_rigid} if $v$ is a rigid vertex, and by
  Th. \ref{theo;IP_orbi} if $v$ is an orbisocket vertex.

  For the second assertion, since $T_\Zmax$ is a canonical splitting  (Prop. \ref{prop_TZmax_is_canonical}), 
  $\Out(G) = Out_{\Gamma}(G)$. 
  Corollary \ref{coro_gene_Aut_gog} allows us to conclude
  since the extended isomorphism problem for  rigid groups and orbisockets 
is solvable by Corollary \ref{coro_list} and Theorem \ref{theo;IP_orbi}. 
\end{proof}

Applying Lemma \ref{lem_periph_fini}, we deduce from Theorem \ref{theo;IP_one_ended} the following corollary:

\begin{coro} \label{coro_CIP_one_ended}
There is an effective algorithm that, given two one-ended hyperbolic groups, with marked peripheral structure consisting of finite subgroups, 
determines whether they are isomorphic. 

There is an effective algorithm that,  given a one-ended hyperbolic group $(G,\calp)$ with marked peripheral structure consisting of finite subgroups, 
determines a generating set of $\Outm(G,\calp)$.
\end{coro}

\section{Groups with several ends}
\label{sec_ends}

 In this section,                     
we explain how to reduce the isomorphism problem for hyperbolic groups  
with several ends to the isomorphism problem between  one-ended hyperbolic groups.

\begin{theo}\label{theo;IP} 
The extended isomorphism problem is solvable for the class of hyperbolic groups:
there is an explicit procedure that, given two  hyperbolic groups (possibly with torsion and several ends), indicates whether they are isomorphic.

 There is an explicit procedure that given a hyperbolic group $G$, computes a generating set for $\Out(G)$.
\end{theo}

Recall that the Stallings-Dunwoody deformation space of $G$ 
is the set of $G$-trees with finite edge stabilisers and finite or one-ended vertex groups. 
We view this deformation space as a graph $\calD$
whose vertex set $V(\cald)$ is the set of reduced Stallings-Dunwoody decompositions, and whose
set of edges $E(\cald)$ is the set of pairs of trees $\{T,T'\}$ related by a slide move. 
This graph is connected as this deformation space is non-ascending \cite[Th. 7.2]{GL2}.
Recall that $\Out(G)$ acts by precomposition on the deformation space, and that this action extends to the graph $\calD$.  

 We are almost going to compute the quotient graph $\cald/\Out(G)$: we are going to compute exactly
its vertex set, and its edge set up to some uncertainty.

\begin{lemma}\label{lem;DF_fini}
  The  action of $\Out(G)$ on $\calD$ 
  has finite quotient.
\end{lemma}  

\begin{proof}
  We view a vertex of $\cald/\Out(G)$ as a graph of groups $\Gamma$ up to graph of groups isomorphisms.  Since a
  slide move doesn't change the number of vertices and edges of a graph of groups, and since $\cald$ is connected
  by slides moves, there are finitely many possibilities for the graph $\Gamma$.  Similarly, the isomorphism types
  of vertex and edge groups is preserved by a slide move, so there are finitely possibilities for the assignments
  $v\mapsto \G_v$ and $e\mapsto \G_e$.  But since there are only finitely many conjugacy classes of finite subgroups
  in the vertex groups, there are only finitely many conjugacy classes of possible edge morphisms $i_e:\G_e\ra
  \G_v$.  If follows that $\cald/\Out(G)$ is finite.
\end{proof}

\begin{lemma}\label{lem_same_orbit}  
  There is an effective procedure that, given two graphs of groups decomposition of two hyperbolic groups,  
  whose Bass-Serre trees are  in  the Stallings-Dunwoody deformation space,  
  indicates whether there is an isomorphism of graph of groups between them. 
                
  Moreover, there is an effective procedure that, given a  graph of groups decomposition of a 
  hyperbolic group in Stallings-Dunwoody deformation space, computes a generating set of its automorphism group.
\end{lemma}

\begin{proof}
  We use Corollary \ref{coro_IP_gog} to prove the first assertion.
  Indeed, the edge groups (which are finite) have finite computable automorphism groups, and the set of possible isomorphisms
  between them is computable. The vertex groups are one-ended hyperbolic, hence by Corollary \ref{coro_CIP_one_ended}, we have a
  solution to the marked isomorphism problem for the class of vertex groups (the peripheral structure is that of the adjacent
  finite edge groups, marked by a tuple of generator).
  Thus Corollary \ref{coro_IP_gog} applies.

  The second assertion is a consequence of Corollary \ref{coro_gene_Aut_gog}, since 
Theorem \ref{theo;IP_one_ended} allows us to compute a generating set of the automorphisms of the vertex groups with their marked peripheral structure, 
and  Lemma \ref{lem_cent_finite} allows us to compute a system of generators of the 
  centraliser of any finite group in a hyperbolic group.
\end{proof}

Given $T\in \cald$ a reduced Stallings-Dunwoody decomposition, we denote by $\Out_T(G)$ its stabiliser in $\Out(G)$,
and by $Lk(T)$ the set of edges of $\cald$ incident on $T$.

\begin{lemma} \label{lem;around_a_vertex}
There is an effective procedure that, given a reduced Stallings-Dunwoody decomposition $T$ of a hyperbolic group $G$, 
computes a finite set $E_T\subset Lk(T)$,
such that $\Out_T(G).E_T=Lk(T)$.
\end{lemma}

\begin{proof}
Denote by $\Gamma=T/G$ a quotient graph of groups.
Recall (see section \ref{sec_deformation}) that  a slide move is determined by two oriented edges $e_1,e_2$ in $\G$,    
terminating at the same vertex   $v$, and by an element $g\in \G_v$ such that  $i_{e_1} (\G_{e_1}) \subset  i_{e_2}(\G_{e_2})^g $.   
If $u$ denotes the other endpoint of $e_2$, after the slide, the edge $e_1$ is replaced by an edge $e'_1$ incident on $u$, and
with monomorphism
$ i_{e_1'} = i_{\ol e_2}\circ i_{e_2}\m \circ \inn_g\circ   i_{e_1}$.

In this case, the group $i_{e_1}(G_{e_1})^{g\m}$ is a subgroup of the finite group $ i_{e_2}(\G_{e_2})$.  One can decide which
subgroups $F\subset i_{e_2}(\G_{e_2})$ are equal to $i_{e_1}(G_{e_1})^{g\m}$ for some $g\in\G_v$. 
If there exists such an element $g_0\in G_v$, the set of all other such elements is  $N_F. g_0$ where $N_F$ is the normaliser
of $F$ in $\G_v$.

By Lemma \ref{lem;slide_autom}, all slide moves with $g\in Z_F. g_0$ where $Z_F$ is the centraliser of $F$ in $\G_v$, are
equivalent under the stabiliser of $T$.  This ensures that, in order to enumerate representatives of possible slide moves, it
suffices to compute a set of representatives for $N_F/Z_F$ (which is finite since $\Aut(F)$ is finite).  This is done by Lemma
\ref{lem_cent_finite}.
\end{proof}

\begin{lemma}\label{lem;stallings_DF}
There is an effective procedure that, given a hyperbolic group $G$, computes a finite set $V_0\subset V(\cald)$
such that $\Out(G).V_0=V(\cald)$, 
and no two points of $V_0$ are in the same orbit under $\Out(G)$.
\end{lemma}

\begin{proof} Using Gerasimov's algorithm \cite{Gerasimov,DaGr_detecting}, one can compute a certain reduced $G$-tree $T_0$ in the
  Stallings-Dunwoody deformation space.  

  Proceed by induction, and assume that one has computed a finite set $S_n\subset V(\cald)$ such that
  $\Out(G).S_n$ contains all vertices of $\cald$ at distance $\leq n$ from $T_0$, and
  no two points of $S_n$  are in the same orbit under $\Out(G)$.

  By Lemma \ref{lem;around_a_vertex}, for each $T\in S_n$, we can compute a finite set $N_T$ of representatives of its neighbours. 
  Since one can check whether two trees are in the same $\Out(G)$-orbit by Lemma \ref{lem_same_orbit},
  one can choose from  $\bigcup_{T\in S_n} N_T$ one representative of each orbit not in the orbit of $S_n$.
  Then we can take $S_{n+1}\supset S_n$ to be the set obtained by adding these elements to $S_n$.

  Since $\cald/\Out(G)$ is finite, the sequence $S_n$ stabilises. We claim that if $S_{n+1}=S_n$, then $\Out(G).S_n=\cald$. Indeed,
  $S_{n'}=S_n$ for all $n'>n$, but by connectedness of $\cald$, any $T\in \cald$ lies in $\Out(G).S_{n'}$ for some $n'$.
  Therefore, one can output $V_0=S_n$.
\end{proof}

\begin{rem}
 The two previous  lemmas can be interpreted as the computation of a finite graph $\cald_0$ 
whose vertex set is $\cald/\Out(G)$, and whose link at $T\in \cald_0$ is given by $E_T$.
The natural map $\cald_0\ra \cald$
is an isomorphism the vertex sets, and is surjective on the edges of $\cald$.
\end{rem}

\begin{lemma} \label{lem;cas_gene}
 Let $G$ act on a  connected graph $K$.  
Let $V_0$ be a set of orbit representatives for the vertices. 
For all $v\in V_0$, let $S_v$ be a generating set of the stabiliser of $v$, 
let $E_v\subset Lk(v)$ be such that $S_v.E_v=Lk(v)$, 
and for all $e=vw\in E_v$, let $\alpha_e\in G$ be such that $\alpha_e.w\in V_0$. 
Denote by $S_{V}=\bigcup_{v\in V_0} S_v$ and $S_E=\{\alpha_e|e\in E_v, v\in V_0\}$.

Then $G$ is generated by $S_V\cup S_E$.
\end{lemma}

This is elementary, but convenient for our purpose. 
\begin{proof}  Let $G'=\grp{S_V\cup S_E}$.
Consider $\gamma\in G$, choose $v_0\in V_0$, and take $v=\gamma v_0$. 
Consider $\gamma'\in G'$ such that $d(\gamma 'v,V_0)$ is minimal.
If $\gamma'v\in V_0$, then $\gamma'v=v_0$, so $\gamma'\gamma\in \grp{S_{v_0}}$ and $\gamma\in G'$ so we are done.
Otherwise, consider a shortest path $v_1,v_2,\dots,v_n=\gamma'v$ from  $V_0$ to $\gamma'v$.
Let $e=v_1v_2$ be its first edge. There exists $e_1\in E_{v_1}$ and 
$\gamma_1 \in \langle S_{v_1}\rangle$ such that $\gamma_1 e_1=e$,
and $\alpha_e\in S_E$ with $\alpha_e.v_2\in V_0$.
Then  $\alpha_e\gamma_1\m .(v_2,\dots,v_n)$ is a path of length $n-1$ joining $V_0$ to a point in $G'v$, a contradiction.
\end{proof}

\begin{proof}[Proof of Theorem \ref{theo;IP}]
   Given $G_1,G_2$ two hyperbolic groups, compute (by Gerasimov's algorithm) a reduced Stallings-Dunwoody decomposition $\Gamma_1$ for $G_1$.
By Lemma \ref{lem;stallings_DF}, compute a finite set $V_0$ of representatives  of reduced Stallings-Dunwoody decompositions 
for $G_2$ modulo $\Out(G_2)$. Then, by Lemma \ref{lem_same_orbit}, one can check whether  $\Gamma_1$ is isomorphic 
(as a graph of group decomposition) to a decomposition $\Gamma_2\in V_0$.
This is enough to conclude whether the two groups are isomorphic or not.

In order to compute a generating set of $\Out(G)$, we look at the action of $\Out(G)$ on $\cald$.  
We use Lemma \ref{lem;stallings_DF}
a finite set $V_0\subset V(\cald)$ of representatives of $\Out(G)$-orbits in $V(\cald)$.
By Lemma \ref{lem_same_orbit}, for each $T\in V_0$, we can compute a finite generating set $S_T$ of the stabiliser of $T$.
For each $T\in V_0$, we use Lemma \ref{lem;around_a_vertex} to compute a finite set $E_T$ of oriented edges originating at $v$,
with $\Out_T(G).E_T=Lk(T)$.
By Lemma \ref{lem_same_orbit}, for each edge $e=\{T,T'\}\in E_T$, 
one can compute an element $\alpha_e\in\Out(G)$ sending $T'$ into $V_0$.
By connectedness of $\cald$, Lemma \ref{lem;cas_gene} gives us a generating set of $\Out(G)$.
\end{proof}

\section{Whitehead problems}\label{sec_Wh}

The goal of this section is to give a solution of the extended isomorphism problem
for hyperbolic groups with marked peripheral structure, and to deduce solutions of some Whitehead problems.

\begin{thm}\label{thm_whitehead_marque}
The extended isomorphism problem is solvable for the class of hyperbolic groups with marked peripheral structure.
More precisely, there is an algorithm that takes as input two hyperbolic groups with marked peripheral structure
$(G,\calp),(G',\calp')$, and decides whether $(G,\calp)\simeq (G',\calp')$.

There is an another algorithm which takes as input a  hyperbolic group with marked peripheral structure
$(G,\calp)$, and gives a generating set of $\Outm(G,\calp)$.
\end{thm}

Restricting to finite or virtually cyclic subgroups, we can as well solve the unmarked isomorphism problem:

\begin{cor}\label{cor;ip_unmarked_hyp}
  The extended isomorphism problem is solvable for the class of hyperbolic groups with unmarked peripheral structures
with virtually cyclic or finite peripheral subgroups.
\end{cor}

\begin{proof}
  This follows from Theorem \ref{thm_whitehead_marque} and Lemma \ref{lem_eqv_um}.
\end{proof}

Before proving the theorem, we first give a few applications. We consider
the following versions of the Whitehead problem:

Given a hyperbolic group $G$ 
and two families of elements $g_1, \dots g_n$, and $h_1, \dots, h_n$, 
\begin{itemize}
\item[(W1)] 
  is there an automorphism $\phi\in \Aut(G)$ such that $\phi(g_i)= h_i$ ?
\item[(W2)]  
  is there an automorphism $\phi\in \Aut(G)$ such that $\phi(g_i)$ is conjugate to  $h_i$ ? 
\item[(W3)]   
  is there an automorphism $\phi\in \Aut(G)$ such that $\langle \phi(g_i) \rangle = \langle h_i \rangle$ ?
\item[(W4)]   
  is there is an automorphism $\phi\in \Aut(G)$ such that $\langle \phi(g_i) \rangle$  
  is conjugate to $\langle h_i \rangle$ ?
\end{itemize}

\begin{cor}\label{cor_whitehead_1234}
  Whitehead problems (W1), (W2), (W3) and (W4) have a uniform solution in  hyperbolic groups.
\end{cor}

The uniformity means that one algorithm works for all hyperbolic groups, taking a presentation
of the group as input.

\begin{proof}
  Consider a hyperbolic group $G$ 
  and two families of elements $g_1, \dots g_n$, and $h_1, \dots, h_n$.

Problem  (W1) asks about the existence of an automorphism $\phi\in \Aut(G)$ such that $\phi(g_i)= h_i$.
Consider the tuples $S=(g_1,\dots,g_n)$, $S'=(h_1,\dots,h_n)$, and $\calp=((S))$, $\calp'=((S'))$
the corresponding marked peripheral structures with one peripheral subgroup.
By the solution of the marked isomorphism problem, we can decide whether there exists $\phi$ and $g$
such that $\phi(g_i)=h_i^g$. This is equivalent to (W1).

Problem  (W2) asks about the existence of an automorphism $\phi\in \Aut(G)$ such that $\phi(g_i)$ is conjugate to $h_i$.
Consider the 1-tuples $S_i=(g_i)$, $S'_i=(h_i)$, and $\calp=(S_1,\dots,S_n)$, $\calp'=(S'_1,\dots,S'_n)$.
Then the existence of $\phi$ as in (W2) is equivalent to the fact that $(G,\calp)\simeq (G,\calp')$.

Problem (W3),  ask about the existence of an automorphism $\phi\in \Aut(G)$ such that $\langle \phi(g_i) \rangle = \langle h_i \rangle$.
For each cyclic group $\grp{g_i}$ (finite or infinite), one can compute all possible $s_i$ with $\grp{s_i}=\grp{g_i}$.
For each such choice, (W1) tells us if there is $\phi$ sending $s_i$ to $h_i$. 
There exists such  a $\phi$ for some choice if and only if (W3) has  a positive answer.

Finally, 
Problem (W4) is just an instance of the unmarked isomorphism problem with cyclic (maybe finite) peripheral subgroups, 
and is solved by application of Corollary \ref{cor;ip_unmarked_hyp}.   
\end{proof}

Let $F_n$ be a free group.
 McCool's theorem  \cite{McCool_fp} gives an algorithm to compute, given a finite subset $X \subset F_n$,  a finite presentation 
 of the stabiliser of $X$ in $Aut(F_n)$.

 Because we have a solution of the extended isomorphism problem, 
 we have the following (generalised but weaker) version of McCool's Theorem \cite{McCool_fp}: 

\begin{cor}\label{cor_mccool}
  There exists an algorithm that takes as input a hyperbolic group $G$,
and $g_1,\dots,g_n\in G$, and which outputs a finite generating set of
\begin{itemize}
\item the subgroup of $\Aut(G)$ fixing the conjugacy classes of $g_1,\dots,g_n$
\item the subgroup of $\Aut(G)$ fixing $g_1,\dots,g_n$
\end{itemize}
\end{cor}

\begin{proof}
  By Theorem \ref{thm_whitehead_marque}, one can compute a generating set of $\Outm(G,((g_1),\dots,(g_n))$.
Choosing representatives in $\Aut(G)$, and adding a generating set of inner automorphisms answers the first problem.

For the second problem, compute a generating set $S$ of $\Outm(G,((g_1,\dots,g_n))$ where the peripheral structure
consists of one peripheral tuple. 
For each $\alpha\in S$, consider $\Tilde \alpha\in \Aut(G)$ representing $\alpha$, and 
modify $\Tilde\alpha$ by an inner automorphism so that $\Tilde \alpha$ fixes $g_1,\dots,g_n$.
Let $\Tilde S$ be the obtained lifts of generators.
Compute $S_Z$ a generating set of the centraliser of $(g_1,\dots,g_n)$; 
if the group they generate is elementary, this follows from Lemmas \ref{lem_cent_finite} and \ref{lem_VC2};
otherwise, this centraliser is finite, and one can compute this centraliser by solving, for each representative $F$ of the conjugacy classes of finite groups,
the system of equations with unknown $g$ saying $F^g$ commutes with $g_1,\dots,g_n$;
the largest finite group $F^g$ obtained is the centraliser.
Viewing $S_Z$ as a set of inner automorphisms,
then $\Tilde S\cup S_Z$ is 
a generating set of the stabiliser of $(g_1,\dots,g_n)$ in $\Aut (G)$.
\end{proof}

We now give the proof of the main result of this section, solving the extended isomorphism problem for hyperbolic groups with marked peripheral structures. 
This will deduced from the absolute case 
using a standard filling method.

\begin{proof}[Proof of Theorem \ref{thm_whitehead_marque}]
Obviously, we can assume that no peripheral subgroup is trivial.
Write $\calp=(S_1,\dots,S_n)$,
and let $T_i$ be the the family of elements of $S_i$ together with the products of all pairs of elements of $S_i$.
The point of introducing $T_i$ is that if all elements of $T_i$ are elliptic in a $G$-tree, 
then by  Serre's Lemma, the group $\grp{S_i}=\grp{T_i}$ fixes a point in this tree.
Similarly, write $\calp'=(S'_1,\dots,S'_n)$ and define $T'_i$ analogously.
This defines enlarged peripheral structures $\calq=(T_1,\dots,T_n)$ and $\calq'=(T'_1,\dots,T'_n)$ of $G$ and $G'$.
Note that $(G,\calp)\simeq(G',\calp')$ if and only if 
$(G,\calq)\simeq(G',\calq')$.

Write $T_i=(t_{i,1},\dots,t_{i,p_i})$. For each $t_{i,j}\in T_i$ of infinite order, choose $R_{i,j}$ a torsion free
one-ended hyperbolic group without cyclic splitting 
and choose $\tau_{i,j}$ a generator
of some maximal cyclic group of $R_{i,j}$  
(we allow the groups $R_{i,j}$ to be isomorphic to each other, so we can give once and for all
a presentation of such a group to our algorithm).
Define $\Hat G$ as the multiple amalgam $G(*_{\grp{t_{i,j}}}R_{i,j})_{i,j}$
where one identifies $t_{i,j}$ with $\tau_{i,j}$.
Denote by $\Lambda$ this decomposition of $\Hat G$.
Consider $R'_{i,j}$ an isomorphic copy of $R_{i,j}$, $\tau'_{i,j}$ the copy of $\tau_{i,j}$,  
and define $\Hat G'=G'(*_{\grp{t'_{i,j}}}R'_{i,j})_{i,j}$ analogously.
By Bestvina-Feighn's combination theorem, $\Hat G$ and $\Hat G'$ are hyperbolic groups \cite{BF_combination}.

Clearly, if $(G,\calq)\simeq (G',\calq')$, then there is an isomorphism $\phi:\Hat G\ra \Hat G'$
sending $R_{i,j}$ to a conjugate of $R'_{i,j}$, and sending each tuple $T_i$ to a conjugate of $T'_i$.
We prove that the converse is true under one-endedness assumptions.

\begin{lem}\label{lem_chapeau}
Assume that $\Hat G$ and $\Hat G'$ are one-ended.
Assume that there is an isomorphism $\phi:\Hat G\ra \Hat G'$
sending $R_{i,j}$ to a conjugate of $R'_{i,j}$, and sending $T_i$ to a conjugate of $T'_i$.  

Then $(G,\calq)\simeq (G',\calq')$.
\end{lem}

\begin{proof}
  Let $G=\pi_1(\Gamma)$ be a $\Z$-JSJ decomposition of $G$ relative to $\calq$.
Let $\Hat \Gamma$ be a decomposition of $\Hat G$ obtained from $\Lambda$ by blowing up its central vertex
using $\Gamma$.
Using the same argument as \cite[Lemma 7.31]{GL3a},  
we see that $\Hat \Gamma$ is a $\Z$-JSJ decomposition of $\Hat G$.

Let $T_c$ be the tree of cylinders of the Bass-Serre tree of $\Hat \Gamma$.
Note that the vertices $R_{i,j}$ are terminal vertices of $T_c/\Gamma$.
Let $\Tilde R\subset T_c$ be the set of points stabilised by a conjugate of some $R_{i,j}$.
Then $T_c\setminus \Tilde R$ has only one orbit of connected components,
and $G$ is the stabiliser of one of them.
The analogous facts hold for the analogous tree $T'_c$.
Since $T_c$ and $T'_c$ are canonical, there exists a $\phi$-equivariant map $f:T_c\ra T'_c$.
Since $\phi$ maps each $R_{i,j}$ to a conjugate of $R'_{i,j}$,
$f$ maps $\Tilde R$ to $\Tilde R'$, so
$\phi(G)$ is conjugate to $G'$, and changing $\phi$ by an inner automorphism, we can assume that $\phi(G)=G'$.

By hypothesis, the tuple $\phi(T_i)\subset G'$ is conjugate in $\Hat G$ to $T'_i\subset G'$.
We claim that $\phi(T_i)$ is conjugate to $T'_i$ in $G$.
Indeed, let $g\in\Hat G$ be such that $\phi(T_i)^g=T'_i$.
Consider $\Tilde \Lambda'$ the Bass-Serre tree of $\Lambda'$, and $u\in\Tilde\Lambda$ the vertex fixed by $G'$.
Then $\grp{\phi(T_i)}$ and $\grp{T'_i}$ both fix the vertex $u\in \Lambda'$,
so $T'_i$ fixes both $u$ and $g\m u$. But $\grp{t_{i,j}}$ being malnormal in $R_{i,j}$, we see that any
segment of length $2$ in $\Tilde\Gamma'$ has trivial stabiliser.
Since we assumed $\grp{T_i}\neq 1$, we get $u=g\m u$, and $g\in G$.
This proves that $\phi(T_i)$ is conjugate to $T'_i$ in $G$, and concludes that
$\phi$ induces an isomorphism between $(G,\calq)$ and $(G',\calq')$.
\end{proof}

\begin{lem}\label{lem_whitehead_1bout}
  The extended isomorphism problem is solvable for groups $(G,\calp)$ with marked peripheral structures, 
relatively one-ended.
\end{lem}

\begin{proof}
  We need to decide the existence of $\phi:\Hat G\ra\Hat G'$ as in
  Lemma \ref{lem_chapeau}.  First, since we have a solution to the isomorphism problem
  (without peripheral structure), we can assume that $\Hat G\simeq
  \Hat G'$.
  Because $R_{i,j}$ is the stabiliser of a rigid vertex of a $\Z$-JSJ
  decomposition of $\Hat G$,  the orbit of the conjugacy class $[R_{i,j}]$ of $R_{i,j}$ 
  under the action of $\Out(\Hat G)$ is finite.  

  We claim that
  the orbit of $\Out(\Hat G)$ on the conjugacy class $[T_i]$ of $T_i$ is
  finite. Indeed,  all elements of $T_i$ are elliptic in any $\Z$-splitting of
  $\Hat G$, so $\grp{T_i}$ is $\Z$-universally elliptic.  In
  particular, it fixes a point $v$ in $\Hat \Gamma$.  If $v$ is
  flexible, then it is a hanging bounded Fuchsian group without
  reflection, and since $\grp{T_i}$ is $\Z$-universally elliptic, it
  is either finite or conjugate into a boundary subgroup of $G_v$
  (\cite{GL3a}).  In each case, the claim follows easily.  If $v$ is
  rigid, all automorphisms in some finite index subgroup of $\Out(\Hat
  G)$ coincide with an inner automorphism on $G_v$, so the claim
  follows in this case too.  We have proved that the $\Out(\Hat
  G)$-orbits of all $[R_{i,j}]$ and $[T_i]$ is
  finite. 

   One can decide whether two quasiconvex subgroups are
  conjugate: one can compute an automaton representing them as \qie rational subsets by \cite[Proposition 1]{Kapovich_detecting},
and one can decide whether there exists $g\in G$ such that $g^{-1}S_1g \subset  \langle S_2\rangle$ and $gS_2g^{-1} \subset  \langle S_1\rangle$
using solvability of systems of \qie rational constraints (Theorem \ref{theo;eq}).  

  One can also decide  whether two tuples of elements are conjugate,
  therefore one can apply Lemma \ref{lem_probleme_orbite} and decide whether
  there exists   $\alpha\in\Out(\Hat G)$ such that $f\circ\alpha
  ([R_{i,j}])=[R'_{i,j}]$ and $f\circ\alpha([T_i])=[T'_i]$.  
  By the claim above, this allows us to decide whether
  $(G,\calq)\simeq (G',\calq')$.

  Lemma \ref{lem_probleme_orbite} also gives a set of generators for
  the subgroup $\Out_0(\Hat G)$ that preserves the conjugacy class of
  each $R_{i,j}$ and of each $T_i$.  We saw that restriction to $G$
  defines an epimorphism $\Out_0(\Hat G)\onto \Outm(G,\calq)$.
  This gives a generating set for  $\Outm(G,\calq)$, as desired.\\
\end{proof}

Now we treat the case where $\Hat G$ is not one-ended.
  By Lemma \ref{lem_periph_fini}, we can assume that all peripheral subgroups are infinite.

\begin{lem}\label{lem_SDrel}
  One can compute a relative Stallings-Dunwoody decomposition of $(G,\calp)$.
\end{lem}

\begin{proof}
  We note that $G$ is one-ended relative to $\calp$ if and only if $\Hat G$
is one-ended. Using Gerasimov's algorithm, this decide allows us to decide whether
$(G,\calp)$ is relatively one-ended.
If it is, we are done. 
If it is not, one can enumerate by Tietze transformations 
all presentations of $G$. We need to recognise among them which correspond to splittings of $G$
over a finite group relative to $\calp$. One can proceed as follows. One can first easily recognise presentations on which one can read a splitting over 
 finite groups. 

 Since vertex groups are quasiconvex, one can
compute an automaton representing them as \qie rational subsets by \cite[Proposition 1]{Kapovich_detecting}.
Then, one can check whether the tuples of  $\calp$ can each be conjugated into a vertex group, 
 by solving an explicit system of equations with \qie constraints  (Theorem \ref{theo;eq}). 
This determines whether the considered presentation  corresponds to a splitting of $G$
over a finite group relative to $\calp$.
 
 Then, we can look at a vertex group $G_v$ of this decomposition, compute its peripheral structure
induced by $\calp$, and iterate this procedure. This will stop by Dunwoody's accessibility.
When it stops, we have a decomposition of $(G,\calp)$ over finite groups such that vertex group
don't split over finite groups relative to the peripheral subgroups they contain.
\end{proof}

The following Lemma concludes the proof of Theorem \ref{thm_whitehead_marque}.

\begin{lem}
The extended isomorphism problem is solvable for hyperbolic groups with marked peripheral structures
whose peripheral subgroups are infinite.
\end{lem}

\begin{proof}
The proof is similar to the argument in Section \ref{sec_ends}.
We define $\cald$ the deformation space of Stallings-Dunwoody decompositions of $G$ relative to $\calp$.
As in Lemma \ref{lem;DF_fini}, $\cald/\Outm(G,\calp)$ is finite.
Indeed, we view each element of $\cald/\Outm(G,\calp)$  as a graph of groups where each vertex group $\Gamma_v$
has a marked peripheral structure $\calp_v$.
Since, $\cald$ is connected by slides, the vertex groups $(\Gamma_v,\calp_v)$ with their marked peripheral structure don't depend
on the decomposition considered. Since each $\Gamma_v$ has finitely many conjugacy classes of finite groups,
there are only finitely many graph of groups that can be constructed from the given $(\Gamma_v,\calp_v)$.

By  Lemma \ref{lem_whitehead_1bout}, 
we have a solution to the extended isomorphism problem for vertex groups with marked peripheral structure $(\Gamma_v,\calp_v)$.

Since peripheral subgroups are infinite, they don't fix any edge, 
so 
 we can apply Lemma \ref{coro_IP_gog} and \ref{coro_gene_Aut_gog} to decide whether two decompositions
are in the same orbit under $\Outm(G,\calp)$. Therefore, Lemmas \ref{lem;around_a_vertex} and \ref{lem;stallings_DF} still apply.
This allows us to decide whether $(G,\calp)\simeq (G',\calp')$ 
so we can solve the extended isomorphism problem as in Theorem \ref{theo;IP}.
\end{proof}

Theorem \ref{thm_whitehead_marque} is now proved.
\end{proof}

\bibliographystyle{alpha2}
\bibliography{published,unpublished}

 \end{document}